\documentclass[11pt]{article}
\usepackage{url}
\usepackage{tikz}
\usepackage{amsmath}
\usepackage{mathtools}
\usepackage{scalerel}
\usetikzlibrary{shapes.geometric, arrows}

\usepackage{enumerate,mathrsfs,comment}
\usepackage{bbold}
\usepackage[numbers]{natbib}
\usepackage[OT1]{fontenc}
\usepackage[hypertexnames=false, colorlinks,linkcolor=red,anchorcolor=blue,citecolor=blue]{hyperref}
\usepackage{fullpage}

\usepackage[protrusion=false,expansion=true]{microtype}
\newcommand{\RNum}[1]{\uppercase\expandafter{\romannumeral #1\relax}}

\RequirePackage{amsmath}
\RequirePackage{amssymb}
\RequirePackage{amsthm}
\RequirePackage{bm}
\RequirePackage{url}
\usepackage{multirow}
\usepackage{graphicx}
\usepackage{subfigure}
\usepackage{makecell}
\usepackage{booktabs}
\usepackage{array}
\usepackage{url}

\usepackage[ruled]{algorithm2e} 
\usepackage{algpseudocode}
\usepackage{dsfont}

\let\hat\widehat

\newcommand{\bg}{\bm{g}}

\newcommand{\bZ}{\bm{Z}}

\newcommand{\m}{\mathcal}

\newcommand{\cH}{\mathbb{H}}

\newcommand{\cX}{\mathcal{X}}

\newcommand{\EE}{\mathbb{E}}

\newcommand{\PP}{\mathbb{P}}

\newcommand{\balpha}{\bm{\alpha}}
\newcommand{\bbeta}{\bm{\beta}}

\newcommand{\argmin}{\mathop{\mathrm{argmin}}}

\newcommand{\mx}{\mbox}

\DeclareMathOperator{\Var}{{\rm Var}}

\newcommand{\bbR}{\mathbb{R}}

\DeclareMathOperator{\tr}{tr}
\theoremstyle{plain}
\numberwithin{equation}{section}
\newtheorem{Theorem}{Theorem}[section]
\newtheorem{Lemma}[Theorem]{Lemma}
\newtheorem{Corollary}[Theorem]{Corollary}

\newtheorem{Assumption}{Assumption}

\newtheorem{lemma}{\indent \bf Lemma}

\newtheoremstyle{mytheoremstyle} %
    {\topsep}                    %
    {\topsep}                    %
    {\normalfont}                   %
    {}                           %
    {\bfseries}                   %
    {.}                          %
    {.5em}                       %
    {}  %

\theoremstyle{mytheoremstyle}

\ifx\BlackBox\undefined
\newcommand{\BlackBox}{\rule{1.5ex}{1.5ex}}  %
\fi

\ifx\QED\undefined
\def\QED{~\rule[-1pt]{5pt}{5pt}\par\medskip}
\fi

\ifx\proof\undefined
\newenvironment{proof}{\par\noindent{\bf Proof\ }}{\hfill\BlackBox\\[2mm]}
\fi

\ifx\theorem\undefined
\newtheorem{theorem}{Theorem}
\fi
\ifx\example\undefined
\newtheorem{example}[theorem]{Example}
\fi
\ifx\property\undefined
\newtheorem{property}{Property}
\fi
\ifx\lemma\undefined
\newtheorem{lemma}[theorem]{Lemma}
\fi
\ifx\proposition\undefined
\newtheorem{proposition}[theorem]{Proposition}
\fi
\ifx\remark\undefined
\newtheorem{remark}[theorem]{Remark}
\fi
\ifx\corollary\undefined
\newtheorem{corollary}[theorem]{Corollary}
\fi
\ifx\definition\undefined
\newtheorem{definition}[theorem]{Definition}
\fi
\ifx\conjecture\undefined
\newtheorem{conjecture}[theorem]{Conjecture}
\fi
\ifx\fact\undefined
\newtheorem{fact}[theorem]{Fact}
\fi
\ifx\claim\undefined
\newtheorem{claim}[theorem]{Claim}
\fi
\ifx\assumption\undefined

\fi
\numberwithin{equation}{section}
\numberwithin{theorem}{section}

\begin{document}

\title{Scalable Statistical Inference in Non-parametric Least Squares} 

\author
{
Meimei Liu\thanks{Department of Statistics,Virginia Tech, Blacksburg, VA. Email: meimeiliu@vt.edu }\, , Zuofeng Shang\thanks{Mathematical Sciences, NJIT,Newark, NJ. Email:zshang@njit.edu} \ and Yun Yang\thanks{Department of Statistics, UIUC, Champaign, IL. Email: yy84@illinois.edu}
}

\date{}
\maketitle

\begin{abstract}
Stochastic approximation (SA) is a powerful and scalable computational method for iteratively estimating the solution of optimization problems in the presence of randomness, particularly well-suited for large-scale and streaming data settings. In this work, we propose a theoretical framework for stochastic approximation (SA) applied to non-parametric least squares in reproducing kernel Hilbert spaces (RKHS), enabling online statistical inference in non-parametric regression models. We achieve this by constructing asymptotically valid pointwise (and simultaneous) confidence intervals (bands) for local (and global) inference of the nonlinear regression function, via employing an online multiplier bootstrap approach to a functional stochastic gradient descent (SGD) algorithm in the RKHS. Our main theoretical contributions consist of a unified framework for characterizing the non-asymptotic behavior of the functional SGD estimator and demonstrating the consistency of the multiplier bootstrap method. The proof techniques involve the development of a higher-order expansion of the functional SGD estimator under the supremum norm metric and the Gaussian approximation of suprema of weighted and non-identically distributed empirical processes. Our theory specifically reveals an interesting relationship between the tuning of step sizes in SGD for estimation and the accuracy of uncertainty quantification.
\end{abstract}

\section{Introduction}
Stochastic approximation (SA)~\cite{robbins1951stochastic,ruppert1988efficient,bottou2008tradeoffs} is a class of iterative stochastic algorithms to solve the stochastic optimization problem
$\min_{\theta\in\Theta}\big\{\m L(\theta):\,= \EE_{Z}[\ell(\theta;Z)]\big\}$,
where $\ell(\theta;z)$ is some loss function, $Z$ denotes the internal random variable, and $\Theta$ is the domain of the loss function. 
Statistical inference, such as parameter estimation, can be viewed as a special case of stochastic optimization where the goal is to estimate the minimizer $\theta^\ast=\argmin_{\theta\in\Theta} \m L(\theta)$ of the expected loss function $\m L(\theta)$ based on a finite number of i.i.d.~observations $\{Z_1,\ldots,Z_n\}$. Classical estimation procedures based on minimizing an empirical version $\m L_n(\theta)=n^{-1}\sum_{i=1}^n \ell(\theta;Z_i)$ of the loss correspond to the sample average approximation (SAA)~\cite{kleywegt2002sample,kim2015guide} for solving the stochastic optimization problem. However, directly minimizing $L_n$ with massive data is computationally wasteful in both time and space, and may pose numerical challenges. For example, in applications involving streaming data where new and dynamic observations are generated on a continuous basis, it may not be necessary or feasible to store all historical data. Instead, stochastic gradient descent (SGD), or Robbins-Monro type SA algorithm~\cite{robbins1951stochastic}, is a scalable approximation algorithm for parameter estimation with constant per-iteration time and space complexity. SGD can be viewed as a stochastic version of the gradient descent method that uses a noisy gradient, such as $\nabla\ell(\cdot\,;Z)$ based on a single $Z$, to replace the true gradient $\nabla \m L(\cdot)$. In this work, we explore the use of SA for statistical inference in infinite-dimensional models where $\Theta$ is a functional space or, more precisely, in solving non-parametric least squares in reproducing kernel Hilbert spaces (RKHS).

Consider the standard random-design non-parametric regression model
\begin{equation}\label{eq:model:0}
Y_i = f^\ast(X_i) + \epsilon_i,  \quad \epsilon_i \sim N(0,\sigma^2) \quad \; \textrm{for} \; i=1,\cdots, n,
\end{equation}
with $X_i\in\cX$ denoting the $i$-th copy of random covariate $X$, $Y_i$ the $i$-th copy of response $Y$, and $f^\ast$ the unknown regression function in a reproducing kernel Hilbert space (RKHS, \cite{aronszajn1950theory,wahba1990spline}) $\mathbb H$ to be estimated. For simplicity, we assume that $\m X=[0,1]^d$ is the unit cube in $\mathbb R^d$.
Since $f^\ast$ minimizes the population-level expected squared error loss objective
$\m L(f) = \mathbb E\big[\ell\big(f;\,(X,Y)\big)\big]$ over all functions $f:\,\m X\to \mathbb R$, with $\ell\big(f;\,(X,Y)\big)=(f(X)-Y)^2$ representing the squared loss function, one can adopt the SSA approach to estimate $f^\ast$ by minimizing a penalized sample-level squared error loss objective. 
Given a sample $\{(X_i,Y_i)\}_{i=1}^n$ of size $n$, a commonly used SAA approach for estimating $f$ is kernel ridge regression (KRR). KRR incorporates a penalty term that depends on the norm $\|\cdot\|_{\mathbb H}$ associated with the RKHS $\mathbb H$. 
Although the KRR estimator enjoys many attractive statistical properties~\cite{koltchinskii2006local, mendelson2002geometric,yang2017}, its computational complexity of $\mathcal{O}(n^3)$ time and $\mathcal{O}(n^2)$ space hinders its practicality in large-scale problems~\cite{saunders1998ridge}. 
In this work, we instead consider an SA-type approach for directly minimizing the functional $\mathcal{E}(f)$ over the infinite-dimensional RKHS. By operating SGD in this non-parametric setting (see Section~\ref{sec:sgd:problem} for details), the resulting algorithm achieves $\mathcal{O}(n^2)$ time complexity and $\mathcal{O}(n)$ space complexity. In a recent study~\cite{Bach2016}, the authors demonstrate that the online estimator of $f$ resulting from the SGD achieves optimal rates of convergence for a variety of $f\in \mathbb H$. It is interesting to note that since the functional gradient is defined with respect to the RKHS norm $\|\cdot\|_{\mathbb H}$, the functional SGD implicitly induces an algorithmic regularization due to the ``early-stopping" in the RKHS, which is controlled by the accumulated step sizes. Therefore, with a proper step size decaying scheme, no explicit regularization is needed to achieve optimal convergence rates.

The aim of this research is to take a step further by constructing a new inferential framework for quantifying the estimation uncertainty in the SA procedure. This will be achieved through the construction of pointwise confidence intervals and simultaneous confidence bands for the functional SGD estimator of $f$. Previous SGD algorithms and their variants, such as those discussed in~\cite{bottou2008tradeoffs,bottou1998online,le2011optimization,bottou2010large,cao2019generalization, cheridito2021non}, are mainly utilized to solve finite-dimensional parametric learning problems with a root-$n$ convergence rate.  In the parametric setting,  asymptotic properties of estimators arising in SGD, such as consistency and asymptotic normality, have been well established in literature; for example, see \cite{Bach2016,P1992,Fang2017,Chen2016}.  
However, the problem of uncertainty quantification for functional SGD estimators in non-parametric settings is rarely addressed in the literature.

In the parametric setting, several methods have been proposed to conduct uncertainty quantification in SGD. \cite{nesterov2008,nemirovski2009} appear to be among the first to formally characterize the magnitudes of random fluctuations in SA; however, their notion of confidence level is based on the large deviation properties of the solution and can be quite conservative. More recently, \cite{Fang2017} proposes applying a multiplier bootstrap method for the construction of SGD confidence intervals, whose asymptotic confidence level is shown to exactly match the nominal level. \cite{Chen2016} proposes a batch mean method to estimate the asymptotic covariance matrix of the estimator based on a single SGD trajectory. Due to the limited information from a single run of SGD, the best achievable error of their confidence interval (in terms of coverage probability) is of the order $\mathcal O(n^{-1/8})$, which is worse than the error of the order $\mathcal O(n^{-1/3})$ achieved by the multiplier bootstrap. 
\cite{su2018uncertainty} proposes a different method called Higrad. Higrad constructs a hierarchical tree of a number of SGD estimators and uses their outputs in the leaves to construct a confidence interval. 

In this work, we develop a multiplier bootstrap method for uncertainty quantification in SA for solving online non-parametric least squares. 
Bootstrap methods~\cite{efron1994introduction, diciccio1988review} are widely used in statistics to estimate the sampling distribution of a statistic for uncertainty quantification. Traditional resampling-based bootstrap methods are unsuitable for streaming data inference as the resampling step necessitates storing all historical data, which contradicts the objective of maintaining constant space and time complexity in online learning. Instead, we extend the parametric online multiplier bootstrap method from \cite{Fang2017} to the non-parametric setting. We achieve this by employing a perturbed stochastic functional gradient, which is represented as an element in the RKHS evaluated upon the arrival of each new covariate-response pair $(X_i,Y_i)$, to capture the stochastic fluctuation arising from the random streaming data.

To theoretically justify the use of the proposed multiplier bootstrap method, we make two main contributions. First, we build a novel theoretical framework to characterize the non-asymptotic behavior of the infinite-dimensional functional SGD estimator via expanding it into higher-orders under the supremum norm metric. This framework enables us to perform local inference to construct pointwise confidence intervals for $f$ and global inference to construct a simultaneous confidence band. Second, we demonstrate the consistency of the multiplier bootstrap method by proving that the perturbation injected into the stochastic functional gradient accurately mimics the randomness pattern in the online estimation procedure, so that the conditional law of the bootstrapped functional SGD estimator given the data asymptotically coincides with the sampling law of the functional SGD estimator. Our proof is non-trivial and contains several major improvements that refine the best (to our knowledge) convergence analysis of SGD for non-parametric least squares in~\cite{Bach2016}, and also advances consistency analysis of the multiplier bootstrap in a non-parametric setting. 
Concretely, in \cite{Bach2016}, the authors derive the convergence rate of the functional SGD estimator relative to the $L_2$ norm metric. Their theory only concerns the $L_2$ convergence rate of the estimation; hence, the proof involves decomposing the SGD recursion into a leading first-order recursion and the remaining higher-order recursions; and bounding their $L_2$ norms respectively by directly bounding their expectations. In comparison, our analysis for statistical inference in online non-parametric regression requires a functional central limit theorem type result and calls for several substantial refinements in proof techniques.

Our first improvement is to refine the SGD recursion analysis by using a stronger supremum norm metric. This enables us to accurately characterize the stochastic fluctuation of the functional estimator uniformly across all locations. As a result, we can study the coverage probability of simultaneous confidence bands in our subsequent inference tasks. Analyzing the supremum convergence is significantly more intricate than analyzing the $L_2$ convergence. In the proof, we introduce an augmented RKHS different from $\mathbb H$ as a bridge in order to better align its induced norm with the supremum metric; see Remark~\ref{remark:3_2} or equation~\eqref{eqn:augmented} in Section~\ref{sec:proof_sketch} for further details. Additionally, we have to employ uniform laws of large numbers and leverage ideas from empirical processes to uniformly control certain stochastic terms that emerge in the expansions. 
Our second improvement comes from the need of characterizing the large-sample distributional limit of the functional SGD estimator. By using the same recursion decomposition, we must now analyze a high-probability supremum norm bound for the all orders of recursions and determine the large-sample distributional limit of the leading term in the expansion. It is worth noting that the second-order recursion is the most complicated and challenging one to analyze. This recursion requires specialized treatment that involves substantially more effort than the remaining higher-order recursions. A loose analysis, achieved by directly converting an $L_2$ norm bound into the supremum norm bound using the reproducing kernel property the original RKHS $\mathbb H$ — which suffices for bounding the higher-order recursions — might result in a bound whose order is comparable to that of the leading term. This is where we introduce an augmented RKHS and directly analyze the supremum norm using empirical process tools.

Last but not least, in order to analyze the distributional limit of the leading bias and variance terms appearing in the expansion of the functional SGD estimator, we develop new tools by extending the recent technique of Gaussian approximation of suprema of empirical processes~\cite{chernozhukov2014gaussian} from equally weighted sum to a weighted sum. This extension is important and unique for analyzing functional SGD, since earlier-arrived data points will have larger weights in the leading bias and variance terms than later-arrived data points; see Remark~\ref{rem:asym_var} for more discussions. Towards the analysis of our bootstrap procedure, we further develop Gaussian approximation bounds for multiplier bootstraps for suprema of weighted and non-identically distributed empirical process, which can be used to control the Kolmogorov distance between the sampling distributions of the pointwise evaluation (local inference) of the functional SGD estimator or its supremum norm (global inference), and their bootstrapping counterparts. Our results also elucidate the interplay between early stopping (controlled by the step size) for optimal estimation and the accuracy of uncertainty quantification.
 
The rest of the article is organized as follows. In Section \ref{sec:background} we introduce the background of RKHS and the functional stochastic gradient descent algorithms in the RKHS; in Section \ref{sec:SGD_inference}, we establish the distributional convergence of SGD for non-parametric least squares; in Section \ref{sec:bootstrap_SGD}, we develop the scalable uncertainty quantification in RKHS via multiplier bootstrapped SGD estimators; Section \ref{sec:numerical} includes extensive numerical studies to demonstrate the performance of the proposed SGD inference. Section \ref{sec:proof_sketch} presents a sketched proof highlighting some important technical details and key steps; Section \ref{sec:dis} provides an overview and future direction for our work. Section \ref{sec:key_proof} includes some key proofs for the theorems. 

\medskip
\noindent {\bf Notation:}
In this paper, we use $C, C', C_1, C_2,\dots$ to denote generic positive constants whose values may change from one line to another, but are independent from everything else. We use the notation $\|f\|_{\infty}$ to denote the supremum norm of a function $f$, defined as $\|f\|_{\infty} = \sup_{x\in \cX} |f(x)|$, where $\cX$ is the domain of $f$. The notations $a \lesssim b$ and $a\gtrsim b$ denote inequalities up to a constant multiple; we write $a\asymp b$ when both $a\lesssim b$ and $a \gtrsim b$ hold. For $k>0$, let $\lfloor k \rfloor$ denote the largest integer smaller than or equal to $k$. For two operators $M$ and $N$, The order $M \preccurlyeq N$ if $N-M$ is positive semi-definite.

\section{Background and Problem Formulation}\label{sec:background}
We begin by introducing some background on reproducing kernel Hilbert space (RKHS) and functional stochastic gradient descent algorithms in the RKHS. 

\subsection{Reproducing kernel Hilbert spaces}\label{sub:rkhs}
To describe the structure of regression function $f$ in non-parametric regression model~\eqref{eq:model:0}, we adopt the standard framework of a reproducing kernel Hilbert space (RKHS,~\cite{wahba1990spline,berlinet2011reproducing,gu2013smoothing}) by assuming $f^\ast=\argmin_{f} \mathbb E\big[(f(X)-Y)^2\big]$ to belong to an RKHS $\mathbb H$. Let $\mathbb P_X$ to denote the marginal distribution of the random design $X$, and $L^2(\mathbb P_X)=\big\{f:\,\mathcal X\to\mathbb R\,\big|\,\int_{\mathcal X} f^2(X)\,\mathbb P_X(dx) <\infty\big\}$ to denote the space of all square-integrable functions over $\mathcal X$ with respect to $\mathbb P_X$.
Briefly speaking, an RKHS is a Hilbert space $\mathbb H\subset L^2(\mathbb P_X)$ of functions defined over a set $\mathcal X$, equipped with inner product $\langle\cdot,\,\cdot\rangle_{\mathbb H}$, so that for any $x\in\mathcal X$, the evaluation functional at $x$ defined by $L_x(f) = f(x)$ is a continuous linear functional on the RKHS. Uniquely associated with $\mathbb H$ is
a positive-definite function $K:\mathcal X\times \mathcal X\to \mathbb R$, called the reproducing kernel. The key property of the reproducing kernel is that it satisfies the reproducing property: the evaluation functional $L_x$ can be represented by the reproducing kernel function $K_x:\,=K(x,\,\cdot)$ so that $f(x) = L_x(f) =\langle K_x,\,f\rangle_{\mathbb H}$. According to Mercer's theorem~\cite{aronszajn1950theory}, kernel function $K$ has the following spectral decomposition: 
\begin{equation}\label{M:decom:K}
K(x,x') = \sum_{j=1}^\infty \mu_j \,\phi_j(x)\,\phi_j(x'),
\,\,\,\,x,x'\in\mathcal{X},
\end{equation}
where the convergence is absolute and uniform on $\mathcal X\times\mathcal X$.
Here, $\mu_1 \geq \mu_2 \geq \cdots \geq 0$ 
is the sequence of eigenvalues, and $\{\phi_{j}\}_{j=1}^{\infty}$ are the corresponding eigenfunctions forming an orthonormal
basis in $L^2(\mathbb P_X)$, with the following property: for any $j,k \in \mathbb{N}$,
$$
\langle \phi_{j}, \phi_{k}\rangle_{L^2(\mathcal P_X)} = \delta_{jk} \quad \mbox{and} \quad \langle \phi_{j}, \phi_{k}\rangle_{\mathbb{H}} = \delta_{jk}/\mu_j,
$$ 
where $\delta_{jk} =1$ if $j=k$ and $\delta_{jk} =0$ otherwise.
Moreover, any $f\in \mathbb H$ can be decomposed into $f=\sum_{j=1}^\infty f_j \phi_j$ with $f_j=\langle f, \phi_j \rangle_{L_2(\mathbb P_X)}$, and its RKHS norm can be computed via $\|f\|_{\mathbb H}^2 = \sum_{j=1}^\infty \mu_j^{-1} f_j^2$.

We introduce some technical conditions on the reproducing kernel $K$ in terms of its spectral decomposition.

\begin{Assumption}\label{asmp:A1}
    The eigenfunctions $\{\phi_k\}_{k=0}^\infty$ of $K$ are uniformly bounded on $\cX$, i.e., there exists a finite constant $c_\phi>0$ such that $\sup_{k\ge 1} \|\phi_k\|_{\infty} \le c_\phi$. Moreover, they satisfy the Lipschitz condition $|\phi_k(s)-\phi_k(t)|\leq L\, k\, |s-t|$
for any $s, t \in[0,1]$, where $L$ is a finite constant. 
\end{Assumption} 

\begin{Assumption}\label{asmp:A2}
The eigenvalues $\{\mu_k\}_{k=1}^\infty$ of $K$ satisfy $\mu_k \asymp k^{-\alpha}$ for some $\alpha> 1$. 
\end{Assumption}

The uniform boundedness condition in Assumption \ref{asmp:A1} is common in the literature \cite{mendelson2010regularization}. Assumption \ref{asmp:A2} assumes the kernel to have polynomially decaying eigenvalues. Assumption \ref{asmp:A1}-\ref{asmp:A2} together also implies the kernel function is bounded as $\sup_x K(x,x)\leq c^2_\phi \sum_{k=1}^\infty k^{-\alpha} := R^2$. 
One special class of kernels satisfying Assumptions \ref{asmp:A1}-\ref{asmp:A2} is composed of translation-invariant kernels $K(t,s)=g(t-s)$ for some even function $g$ of period one. In fact, by utilizing the Fourier series expansion of the kernel function $g$, we observe that the eigenfunctions of the corresponding kernel matrix $K$ are trigonometric functions 
$$
\phi_{2k-1}(x) = \sin (\pi k x), \quad \phi_{2k}(x)= \cos (\pi k x), \quad k=1,2, \dots
$$
on $\mathcal{X}=[0,1]$. It is easy to see that we can choose $c_\phi=1$ and $L = \pi$ to satisfy Assumption~\ref{asmp:A1}. 
Although we primarily consider kernels with eigenvalues that decay polynomially for the sake of clarity in this paper, it is worth mentioning that our theory extends to other kernel classes, such as squared exponential kernels and polynomial kernels~\cite{bach2017equivalence}.

\subsection{Stochastic gradient descent in RKHS}\label{sec:sgd:problem}
To motivate functional SGD in RKHS, we first review SGD in Euclidean setting for minimizing the expected loss function $\m L(\theta) = \mathbb E_Z[\ell(\theta;Z)]$, where $\theta\in \mathbb{R}^d$ is the parameter of interest, $\ell:\,\mathbb R^d \times \mathcal Z\to\mathbb R$ is the loss function and $Z$ denotes a generic random sample, e.g.~$Z=(X,Y)$ in the non-parametric regression setting~\eqref{eq:model:0}. 
By first-order Taylor's expansion, one can locally approximate $\m L(\theta + s)$ for any small deviation $s$ by $\m L(\theta+s) \approx \m L(\theta) +  \langle \nabla \m L(\theta),\, s\rangle$, 
where $\nabla \m L(\theta)$ denotes the gradient (vector) of $\m L(\cdot)$ evaluated at $\theta$. The gradient $\nabla \m L(\theta)$ therefore encodes the (infinitesimal) steepest descent direction of $L$ at $\theta$, leading to the following \emph{gradient decent} (GD) updating formula:
 \begin{align*}
\widehat{\theta}_i =  \widehat{\theta}_{i-1} - \gamma_i \,\nabla L (\widehat{\theta}_{i-1}), \quad\mbox{for}\quad i=1,2,\ldots,
\end{align*}
starting from some initial value $\widehat\theta_0$, where $\gamma_i>0$ is the step size (also called learning rate) at iteration $i$.
GD typically requires the computation of the full gradient $\nabla \m L(\theta)$, which is unavailable due to the unknown data distribution of $Z$. In stochastic approximation, SGD takes a more efficient approach by using an unbiased estimate of the gradient as $G_i(\theta)= \nabla \ell(\theta,Z_i)$ based on one sample $Z_i$ to substitute $\nabla \m L(\theta)$ in the updating rule. 

Accordingly, the SGD updating rule takes the form of
\begin{align*}
\widehat{\theta}_i =  \widehat{\theta}_{i-1} - \gamma_i \,G_i(\widehat{\theta}_{i-1}), \quad\mbox{for}\quad i=1,2,\ldots.
\end{align*}

Let us now extend the concept of SGD from minimizing an expected loss function in Euclidean space to minimizing an expected loss functional in function space. Here for concreteness, we develop SGD for minimizing the expected squared error loss $\m L(f)=\mathbb E\big[(f(X) - Y)^2\big]$ over an RKHS $\mathbb H$ equipped with inner product $\langle \cdot,\cdot\rangle_{\mathbb H}$. Let us begin by extending the concept of the ``gradient". By identifying the gradient (operator) $\nabla L:\mathbb H \to \mathbb H$ of functional $\m L(\cdot)$ as a steepest descent ``direction" in $\mathbb H$ through the following first-order ``Taylor expansion"
\begin{align*}
\m L(f) = \m L(g) + \langle \nabla \m L(g),\, f-g\rangle_{\mathbb H} + \mathcal O\big(\|f-g\|_{\mathbb H}^2\big),\ \ \mbox{as }f\to g,\notag
\end{align*}
we obtain after some simple algebra that 
\begin{align*}
   \langle \nabla \m L(g),\, f-g\rangle_{\mathbb H}  + \mathcal O\big(\|f-g\|_{\mathbb H}^2\big) &\,= \m L(f) - \m L(g) \\
   & \,= \mathbb E\big[\big(f(X) - g(X)\big)\cdot \big(g(X)-Y\big)\big] + \mathbb E\big[(f(X) - g(X))^2\big].
\end{align*}
Now by using the reproducing property $h(x) = \langle h,\, K_x\rangle_{\mathbb H}$ for any $h\in\mathbb H$, we further obtain
\begin{align}\label{eqn:RKHS_grad}
    \langle \nabla \m L(g),\, f-g\rangle_{\mathbb H} = \big\langle \mathbb E\big[\big(g(X) -Y\big)K_X\big],\, f-g\big\rangle_{\mathbb H} + \mathcal O\big(\|f-g\|_{\mathbb H}^2\big).
\end{align}
Here, we have used the fact that by Cauchy-Schwarz inequality, 
$$(f(x)-g(x))^2 = \langle f-g,\, K_x\rangle^2 \leq \|f-g\|_{\mathbb H}^2 \cdot \|K_x\|_{\mathbb H}^2 = K(x,x)\,\|f-g\|_{\mathbb H}^2 = \mathcal O\big(\|f-g\|_{\mathbb H}^2\big),$$
since Assumptions~\ref{asmp:A1}-\ref{asmp:A2} together with Mercer's expansion~\eqref{M:decom:K} imply $K$ to be uniformly bounded, or $K(x,x) \leq c_K\sum_{j=1}^\infty \mu_k \leq C\sum_{j=1}^\infty j^{-\alpha} \leq C'$, as long as $\alpha>1$. From equation~\eqref{eqn:RKHS_grad}, we can identify the gradient $\nabla \m L(g)$ at $g\in\mathbb H$ as 
\begin{align*}
    \nabla \m L(g) = \mathbb E\big[\big(g(X) -Y\big)K_X\big] \in \mathbb H.
\end{align*}
Throughout the rest of the paper, we will refer to above $\nabla \m L(g)$ as the RKHS gradient of functional $L$ at $g$.

Upon the arrival of the $i$th data point $(X_i,Y_i)$, we can form an unbiased estimator  $G_i(g)$ of the RKHS gradient $\nabla \m L(g)$ as $G_i(g) = \big(g(X_i) -Y_i\big)K_{X_i}$. This leads to the following SGD in RKHS for solving non-parametric least squares: for a given initial estimate $\widehat{f}_0$, the SGD recursively updates the estimate of $f$ upon the arrival of each data point as 
\begin{equation}\label{eq:sgd:ini}
\widehat{f}_i = \widehat{f}_{i-1} - \gamma_i \,G_i(\widehat{f}_{i-1}) = \widehat{f}_{i-1} + \gamma_i \big(Y_i - \widehat{f}_{i-1}(X_i)\big) K_{X_i},\quad \mbox{for }i=1,2,\ldots.
\end{equation}
By utilizing the reproducing property, the above iterative updating formula can be rewritten as
\begin{align}
\widehat{f}_i =  \widehat{f}_{i-1} + \gamma_i \,\big(Y_i - \langle \widehat{f}_{i-1}, K_{X_i}\rangle_{\mathbb H} \big)\, K_{X_i} 
 =  (I - \gamma_i\, K_{X_i}\otimes K_{X_i})\, \widehat f_{i-1} + \gamma_i \,Y_i\, K_{X_i}, \label{eq:stand:sgd} 
\end{align}
where $I$ denotes the identity map on $\mathbb H$, and
$\otimes$ is the tensor product operator defined through
$g \otimes h(f)=\langle f, h \rangle_{\mathbb H} \,g $ for all $g,h,f \in \mathbb H$. Formula~\eqref{eq:sgd:ini} is more straightforward to use for practical implementation, while formula~\eqref{eq:stand:sgd} provides a more tractable expression that will facilitate our theoretical analysis.
Following \cite{ruppert1988efficient} and \cite{polyak1992acceleration}, we consider the so-called Polyak averaging scheme to further improve the estimation accuracy by averaging over the entire updating trajectory,
i.e.~we use $\bar{f}_n = n^{-1}\sum_{i=1}^n\, \widehat{f}_i$ as the final functional SGD estimator of $f$ based on a dataset of sample size $n$.
Note that this averaged estimator can be efficiently computed without storing all past estimators by using the
recursively updating formula $\bar{f}_i = (1-i^{-1})\,\bar{f}_{i-1}\, +\, i^{-1}\,\widehat{f}_i$ for $i=1,\dots, n$ on the fly.
We will refer to the above SGD as functional SGD in order to differentiate it from the SGD in Euclidean space, and $\bar f_n$ as the functional SGD estimator (using $n$ samples). Throughout the remainder of the paper, we consider a zero initialization, $\widehat{f}_{0}=0$, without loss of generality.

In functional SGD with total sample size (time horizon) $n$, the only adjustable component is the step size scheme $\{\gamma_i:\,i=1,2,\ldots,n\}$, which is crucial for achieving fast convergence and accurate estimations (c.f.~Remark~\ref{rk:step_size}). 
We examine two common schemes~\cite{bottou2010large,bottou2007tradeoffs}: (1) constant step size scheme where $\gamma_i \equiv \gamma = \gamma(n)$ only depends on the total sample size $n$; (2) non-constant step size scheme where $\gamma_i = i^{-\xi}$ decays polynomially in $i$ for $i=1,2,\ldots,n$ and some $\xi>0$. 
While the constant step scheme is more amenable to theoretical analysis, it suffers from two notable drawbacks: (1) it assumes prior knowledge of the sample size $n$, which is typically unavailable in streaming data scenarios, and (2) the optimal estimation error is only achieved at the $n$-th iteration, leading to suboptimal performance before that time point. In contrast, the non-constant step size scheme, despite significantly complicating our theoretical analysis, overcomes the aforementioned limitations and leads to a truly online algorithm that achieves rate-optimal estimation at any intermediate time point (c.f.~Theorem~\ref{le:bias_variance}). Due to this characteristic, we will also refer to the non-constant step size scheme as the online scheme.

Although functional SGD operates in the infinite-dimensional RKHS, it can be implemented using a finite-dimensional representation enabled by the kernel trick.
Concretely, upon the arrival of the $i$-th observation $(X_i,Y_i)$, we can express the time-$i$ intermediate estimator $\widehat f_i$ as $\widehat{f}_i = \sum_{j=1}^i \widehat{\beta}_{j}\, K_{X_j}$  due to equation~\eqref{eq:sgd:ini} and the zero initialization $\widehat f^\ast=0$ condition, where only the last entry $\widehat \beta_i$ in the coefficient vector $(\widehat{\beta}_{1},\,\widehat{\beta}_{2},\,\dots,\, \widehat{\beta}_{i})^\top$ needs to be updated,
\begin{align*}
    \widehat{\beta}_i= \gamma_i\,\big(Y_i-\widehat{f}_{i-1}(X_i)\big) = \gamma_i \,Y_i - \gamma_i \sum_{j=1}^{i-1} \widehat \beta_j \,K(X_j, \,X_i).
\end{align*}
Note that the computational complexity at time $i$ is $\mathcal O(i)$ for $i=1,2,\ldots,n$. 
Correspondingly, the functional SGD estimator at time $i$ can be computed through $\bar{f}_i = (1-i^{-1})\,\bar{f}_{i-1} + i^{-1} \widehat{f}_i=\sum_{j=1}^i \bar\beta_j\, K_{X_j}$, where (can be proved by induction)
\begin{align*}
    \bar\beta_j = \Big(1 - \frac{j-1}{i}\Big) \,\widehat \beta_j,\quad \mx{for}\quad j=1,2,\ldots,i.
\end{align*}
Consequently, the overall time complexity of the resulting algorithm is $\mathcal O(n^2)$, and the space complexity is $\mathcal O(n)$.

\subsection{Problem formulation}\label{sec:problem_formulation}
Our objective is to develop online inference for the non-parametric regression function $f^\ast$ based on the functional SGD estimator $\bar{f}_n$. Specifically, we aim to construct
level-$\beta$ pointwise confidence intervals (local inference) $CI_n(x;\,\beta) = [U_n(x;\,\beta),\, V_n(x;\,\beta)]$ for $f^\ast(x)$, where $x\in\m X$, and a level-$\beta$ simultaneous confidence band (global inference) $CB_n(\beta) = \big\{g:\, \m X\to\mathbb R \,\big|\, g(x)\in[\bar f_n(x) - b_n(\beta),\,\bar f_n(x) + b_n(\beta)],\ \forall x\in \m X\big\}$ for $f^\ast$. We require these intervals and band to be asymptotically valid, meaning that the coverage probabilities, i.e., the probabilities of $f^\ast(x)$ or $f^\ast$ falling within $CI_n(x;\,\beta)$ or $CB_n(\beta)$ respectively, are close to their nominal level $\beta$. Mathematically, this means $\PP[f^\ast(x)\in CI_n(x;\,\beta)]=\beta + o(1)$ and $\PP[f^\ast\in CB_n(\beta)]=\beta + o(1)$ as $n\to\infty$.

The coverage probability analysis of these intervals and band requires us to examine and prove the distributional convergence of two random quantities (with appropriate rescaling) based on the functional SGD estimator $\bar{f}_n$: the pointwise difference $\bar{f}_n(x) - f^\ast(x)$ for $x\in\mathcal{X}$ and the supremum norm $\|\bar{f}_n - f^\ast\|_{\infty}$ of $\bar{f}_n - f^\ast$. In particular, the appropriate rescaling choice determines a precise convergence rate of $\bar{f}$ towards $f^\ast$ under the supremum norm metric. The characterization of the convergence rate of a non-parametric regression estimator under the supremum norm metric is a challenging and important problem in its own right. We note that the distribution of the supremum norm $\|\bar{f}_n - f^\ast\|_{\infty}$ after a proper re-scaling behaves like the supreme norm of a Gaussian process in the asymptotic sense, which is not practically feasible to estimate. Therefore, for inference purposes, it is not necessary to explicitly characterize this distributional limit; instead, we will prove a bootstrap consistency by showing that the Kolmogorov distance between the sampling distributions of this supremum norm and its bootstrapping counterpart converges to zero as $n\to\infty$. 

In our theoretical development to address these problems, we will utilize a recursive expansion of the functional SGD updating formula to construct a higher-order expansion of $\bar{f}_n$ under the $\|\cdot\|_\infty$ norm metric.  Building upon this expansion, we will establish in Section~\ref{sec:SGD_inference} the distributional convergence of the two aforementioned random quantities and characterize their limiting distributions with an explicit representation of the limiting variance for $\bar{f}_n(x) - f^\ast(x)$ in the large-sample setting. However, these limiting distributions and variances depend on the spectral decomposition of the kernel $K$, the marginal distribution of the design variable $X$, and the unknown noise variance $\sigma^2$, which are either inaccessible or computationally expensive to evaluate in an online learning scenario. To overcome this challenge, we will propose a scalable bootstrap-based inference method in Section \ref{sec:bootstrap_SGD}, enabling efficient online inference for $f^\ast$.

\section{Finite-Sample Analysis of Functional SGD Estimator}\label{sec:SGD_inference}
In this section, we start by deriving a higher-order expansion of $\bar{f}_n$ under the $\|\cdot\|_\infty$ norm metric. We then proceed to establish the distributional convergence of $\bar f(x)-f^\ast(x)$ for any $x\in\mathcal{X}$ by characterizing the leading term in the expansion. These results will be useful for motivating our online local and global inference for $f^\ast$ in the following section.

\subsection{Higher-order expansion under supreme norm}\label{sec:higher-order}
We begin by decomposing the functional SGD update of $\widehat{f}_n-f^\ast$ into two leading recursive formulas and a higher-order remainder term. This decomposition allows us to distinguish between the deterministic term responsible for the estimation bias and the stochastic fluctuation term contributing to the estimation variance. Concretely, we obtain the following by plugging $Y_i=f^\ast(X_i) + \epsilon_i$ into the recursive updating formula~\eqref{eq:stand:sgd}, 
\begin{equation}\label{eq:sgd:recursion_f0}
\widehat{f}_i - f^\ast =  (I - \gamma_i\, K_{X_i}\otimes K_{X_i}) \,(\widehat{f}_{i-1}-f^\ast)  + \gamma_i \,\epsilon_i \,K_{X_i}. 
\end{equation}
Let $\Sigma:\,= \EE[K_{X_1}\otimes K_{X_1}]:\, \mathbb{H} \to \mathbb{H}$ denote the population-level covariance operator, so that for any $f$, $g \in \mathbb{H}$ we have $\langle f, \, \Sigma\, g \rangle_\mathbb{H} = \EE[f(X_1)\,g(X_1)]$.   
Now we recursively define the \emph{leading bias term} through
\begin{equation}\label{eq:bias:lead:main}
\eta_0^{bias,0}= \widehat f_0-f^\ast = -f^\ast \quad \mx{and}\quad  \eta_i^{bias,0}=  (I - \gamma_i\, \Sigma)\, \eta_{i-1}^{bias, 0} \quad \mbox{for} \quad i=1,2,\ldots
\end{equation}
that collects the leading deterministic component in~\eqref{eq:sgd:recursion_f0};
and the {\it leading noise term} through
\begin{align}
\eta_0^{noise,0}= 0 \quad\mx{and}\quad
   \eta_i^{noise,0} =  (I - \gamma_i\,\Sigma)\, \eta^{noise,0}_{i-1} + \gamma_i \,\epsilon_i\, K_{X_i}  \quad \mbox{for} \quad i=1,2,\ldots  \label{eq:lead:noise}
\end{align}
that collects the leading stochastic fluctuation component in~\eqref{eq:sgd:recursion_f0};
so that we have the following decomposition for the recursion:
\begin{equation}\label{eq:sgd:higher_order_exp:1}
\widehat{f}_i -f^\ast =\underbrace{\eta_i^{bias,0}}_\text{leading bias} + \underbrace{\eta_i^{noise,0}}_\text{leading noise} + \ \ \underbrace{\big(\widehat{f}_i -f^\ast -\eta_i^{bias,0} - \eta_i^{noise,0}\big)}_\text{remainder term}  \quad \mbox{for} \quad i=1,2,\ldots. 
\end{equation}
Correspondingly, we define $\bar{\eta}_i^{bias,0} = i^{-1}\sum_{j=1}^i \eta_j^{bias,0}$ and $\bar{\eta}_i^{noise,0} = i^{-1}\sum_{j=1}^i \eta_j^{noise,0}$ as the leading bias and noise terms, respectively, in the functional SGD estimator (after averaging).
The following Theorem~\ref{le:bias_variance} presents finite-sample bounds for the two leading terms and the remainder term associated with $\bar f_n$ under the supreme norm metric. The results indicate that the remainder term is of strictly higher order (in terms of dependence on $n$) compared to the two leading terms, validating the term ``leading" for them. 

\begin{Theorem}[Finite-sample error bound under supreme norm]
\label{le:bias_variance}
Suppose that the kernel $K$ satisfies Assumptions~\ref{asmp:A1}-\ref{asmp:A2}. Assume $f^\ast\in \mathbb{H}$ satisfies $\sum_{\nu=1}^\infty \langle f^\ast, \phi_\nu \rangle_{L_2}\mu_\ell^{-1/2} < \infty$. 
\begin{enumerate}
\item (constant step size) Assume that the step size $\gamma_i\equiv\gamma$ satisfies $\gamma \in(0,\, \mu_1^{-1})$, then we have 
\begin{equation*}
\sup_{x\in \m X}|\bar{\eta}_n^{bias,0}(x)|\leq C \frac{1}{\sqrt{n\gamma}},\quad \textrm{and}\;  
\sup_{x\in \m X} \Var(\bar{\eta}_n^{noise,0}(x))\leq C' \frac{(n\gamma)^{1/\alpha}}{n},
\end{equation*}
where $C,\, C'$ are constants independent of $(n,\gamma)$. 
Furthermore, assume that the step size $0<\gamma < n^{-\frac{2}{2+3\alpha}}$, we have
\begin{equation*}
\PP \Big(\|\bar{f}_n - f^\ast - \bar{\eta}_n^{bias,0}-\bar{\eta}_n^{noise,0}\|^2_{\infty} \geq \gamma^{1/2} (n\gamma)^{-1}+ \gamma^{1/4} (n\gamma)^{1/\alpha}n^{-1}\log n\Big) \leq C/n + C\gamma^{1/4},
\end{equation*}
where the randomness is with respect to the randomness in $\{(X_i, \epsilon_i)\}_{i=1}^n$. 

\item (non-constant step size)  Assume the step size to satisfy $\gamma_i = i^{-\xi}$ for some $\xi\in(0,\, 1/2)$, then we have 
\begin{equation*}
\sup_{x\in \m X}|\bar{\eta}_n^{bias,0}(x)|\leq C \frac{1}{\sqrt{n\gamma_n}},\quad \textrm{and}\;  
\sup_{x\in \m X}\Var(\bar{\eta}_n^{noise,0}(x))\leq C' \frac{(n\gamma_n)^{1/\alpha}}{n},
\end{equation*}
where $C,\,C'$ are constants independent of $(n,\gamma_n)$.
For the special choice of $\xi = \frac{1}{\alpha+1}$, we have 
$$
\PP \Big(\|\bar{f}_n - f^\ast - \bar{\eta}_n^{bias,0}-\bar{\eta}_n^{noise,0}\|^2_{\infty} \geq \gamma_n^{1/2} (n\gamma_n)^{-1}+ \gamma_n^{1/2} (n\gamma_n)^{1/\alpha}n^{-1}\log n\Big) \leq C/n + C\gamma_n^{1/2}. 
$$
\end{enumerate} 
\end{Theorem}
\noindent A proof of this theorem is based on a higher-order recursion expansion and careful supreme norm analysis of the recursive formula; see Remark \ref{remark:3_2} and proof sketch in Section \ref{sec:proof_sketch}.  The detailed proof is outlined in \cite{liu2023supp}.

\begin{remark}\label{rk:step_size}
As demonstrated in Theorem~\ref{le:bias_variance}, the selection of the step size $\gamma$ (or $\gamma_n$ for non-constant step size) in the SGD estimator entails a trade-off between bias and variance. A larger $\gamma$ (or $\gamma_n$) increases bias while reducing variance, and vice versa. This trade-off can be optimized by choosing the (optimal) step size $\gamma_n = n^{-\frac{1}{\alpha+1}}$. This is why we specifically focus on this particular choice in the non-constant step size setting in the theorem, which also significantly simplifies the proof. Interestingly, the step size (scheme) in the functional SGD plays a similar role as the regularization parameter in regularization-based approaches in preventing overfitting according to Theorem~\ref{le:bias_variance}. To see this, let us consider the classic kernel ridge regression (KRR), where the estimator $\widehat{f}_{n,\lambda}$ is constructed as
$$
\widehat{f}_{n,\lambda} = \argmin_{f\in \mathbb{H}} \Big\{ 
\frac{1}{n}\sum_{i=1}^n \big(Y_i - f(X_i)\big)^2 + \lambda \|f\|_{\mathbb{H}}^2
\Big\},
$$
where $\lambda$ serves as the regularization parameter to avoid overfitting. It can be shown (e.g.,~\cite{yang2017}) that the squared bias of $\widehat{f}_{n,\lambda}$ has an order of $\lambda$, while the variance has an order of $d_\lambda/n$, where $d_\lambda = \sum_{\nu=1}^\infty \min\{1,\lambda \mu_\nu\}$ represents the effective dimension of the model and is of order $\lambda^{-1/\alpha}$ under Assumption~\ref{asmp:A2}. In comparison, the squared bias and variance of the functional SGD estimator $\bar f_n$ are of order $(n\gamma_n)^{-1}$ and $(n\gamma_n)^{1/\alpha} / n$ respectively. Therefore, $(n\gamma_n)^{-1}$ and $(n\gamma_n)^{1/\alpha}$ respectively play the same role as the regularization parameter $\lambda$ and effective dimension $d_\lambda$ in KRR. More generally, a step size scheme $\{\gamma_i\}_{i=1}^n$ corresponds to an effective regularization parameter of the order $\lambda = \big(\sum_{i=1}^n \gamma_i\big)^{-1}$, which in our considered settings is of order $(n\gamma_n)^{-1}$. Note that the accumulated step size $\sum_{i=1}^n \gamma_i$ can be interpreted as the total path length in the functional SGD algorithm.  This total path length determines the early stopping of the algorithm, effectively controlling the complexity of the learned model and preventing overfitting.
\end{remark}

\begin{remark}\label{remark:3_2}
The higher-order recursion expansion and the supreme norm bound in the theorem provide a finer insight into the distributional behavior of $\bar{f}_n$ and paves the way for inference. That is, we only need to focus on the leading noise recursive term for statistical inference.
In our proof of bounding the supremum norm for the remainder term in equation (\ref{eq:sgd:higher_order_exp:1}), we further decompose the remainder $\bar{f}_n - f^\ast - \bar{\eta}_n^{bias,0}-\bar{\eta}_n^{noise,0}$ into two parts:    
the bias remainder and the noise remainder. Note that a loose analysis of bounding the noise remainder under the $\|\cdot\|_{\infty}$ metric by directly converting an $L_2$ norm bound into the supremum norm bound using the reproducing kernel property of the original RKHS $\mathbb H$ would result in a bound whose order is comparable to that of the leading term. This motivates us to introduce an augmented RKHS $\mathbb{H}_a = \{f= \sum_{\nu=1}^\infty f_\nu \phi_\nu \mid \sum_{\nu=1}^\infty f_\nu^2 \mu_\nu^{2a-1}< \infty\}$ with $0\leq a\leq 1/2-1/(2\alpha)$ equipped with the kernel function $K^a(x,y)= \sum_{\nu=1}^\infty \phi_\nu(X)\phi_\nu(y)\mu_\nu^{1-2a}$ and norm $\|f\|_a=\big(\sum_{\nu=1}^\infty f_\nu^2 \mu_\nu^{2a-1}\big)^{1/2}$ for any $f=\sum_{\nu=1}^\infty f_\nu \phi_\nu\in\cH$. This augmented RKHS norm weakens the impact of high-frequency components compared to the norm $\|f\|_{\cH}=\big(\sum_{\nu=1}^\infty f_\nu^2 \mu_\nu^{-1}\big)^{1/2}$
and its induced norm turns out to be better aligns with the functional supremum norm in our context. As a result, we have  $\|f\|_{\infty}\leq c_a \|f\|_{a} \leq c_k \|f\|_{\cH}$ for any $f\in\cH$, where $(c_a,\,c_k)$ are constants. In particular, a supremum norm bound based on controlling the $\|f\|_{a}$ norm with appropriate choice of $a$ could be substantially better than that based on $\|f\|_{\cH}$; see Section \ref{sec:proof_sketch} and Section \ref{app:le:rem_bias:con} for further details.
\end{remark}

As we discussed in Section~\ref{sec:problem_formulation}, for inference purposes, it is not necessary to explicitly characterize distributional convergence limit of the supremum norm $\|\bar{f}_n - f^\ast\|_{\infty}$; instead, we will prove a bootstrap consistency by showing that the Kolmogorov distance between the sampling distributions of this supremum norm and its bootstrapping counterpart converges to zero as $n\to\infty$. However, the pointwise convergence limit of $\bar{f}_n(z_0) - f^\ast(z_0)$ for fixed $z_0\in[0,1]$ has an easy characterization. Therefore, we present the pointwise convergence limit and use it to discuss the impact of online estimation in the non-parametric regression model in the following subsection.

\subsection{Pointwise distributional convergence}
According to Theorem~\ref{le:bias_variance}, the large-sample behavior of the functional SGD estimator $\bar f_n$ is completely determined by the two leading processes: bias term and noise term. 
According to (\ref{eq:bias:lead:main}), under the constant step size $\gamma$,
the leading bias term has an explicit expression as 
\begin{equation}\label{eq:local_bias}
   \begin{aligned}
   \bar{\eta}_n^{bias,0}(x)= & \frac{1}{n}\gamma^{-1}\Sigma^{-1}\,(I-\gamma \Sigma)\, [I-(I-\gamma\Sigma)^n\,]f^\ast(x) \\
   =  & \frac{1}{\sqrt{\gamma}n}\sum_{k=1}^n \sum_{\nu=1}^\infty  \langle f^\ast, \phi_\nu\rangle_{L_2} \mu_\ell^{-1/2} (1-\gamma \mu_\nu)^k (\gamma\mu_\nu)^{1/2}\phi_\nu(x) ,\quad \forall x\in\m X,
\end{aligned} 
\end{equation}
and the leading noise term is 
\begin{equation}\label{eq:local_noise}
\begin{aligned}
    \bar{\eta}_n^{noise,0}(x) =&\, \frac{1}{n}\sum_{k=1}^n \Sigma^{-1}\big[I-(I-\gamma \Sigma)^{n+1-k}\big]\, K(X_k,\,x)\, \epsilon_k \\
    =&\, \frac{1}{n}\sum_{k=1}^n \ \epsilon_k\, \cdot\, \underbrace{\bigg\{ \sum_{\nu=1}^\infty  \big[1-(1-\gamma \mu_\nu)^{n+1-k}\big]\,\phi_\nu(X_k)\,\phi_\nu(x)\bigg\}}_{\Omega_{n,k}(x)},\quad \forall x\in\m X.
\end{aligned}
\end{equation}
For each fixed $z_0\in\m X$, conditioning on the design $\{X_i\}_{i=1}^n$, the leading noise term $\bar{\eta}_n^{noise,0}(z_0)$ is a weighted average of $n$ independent and centered normally distributed random variables. This representation enables us to identify the limiting distribution of $\bar{\eta}_n^{noise,0}(z_0)$ (this subsection) and conduct local inference (i.e.~pointwise confidence intervals) by a bootstrap method (next section). Under Assumption~\ref{asmp:A2}, the weight $\Omega_{n,k}(z_0)$ associated with the $k$-th observation pair $(X_k, \, Y_k)$ is of order 
$\sum_{\nu=1}^\infty  \big[1-(1-\gamma \mu_\nu)^{n+1-k}\big] \asymp \big[(n+1-k)\gamma\big]^{1/\alpha}$, which decreases in $k$. This diminishing impact trend is inherent to online learning, as later observations tend to have a smaller influence compared to earlier observations. This characteristic is radically different from offline estimation settings, where all observations contribute equally to the final estimator, and will change the asymptotic variance (i.e., the $\sigma^2_{z_0}$ in Theorem~\ref{thm:local:main1}).

Furthermore, the entire leading noise process $\bar{\eta}_n^{noise,0}(\cdot)$ can be viewed as a weighted and non-identically distributed empirical process indexed by the spatial location. This characterization enables us to conduct global inference (i.e.~simultaneous confidence band) for non-parametric online learning by borrowing and extending the recent developments~\cite{chernozhukov2014gaussian,chernozhukov2016empirical,chernozhukov2014anti} on Gaussian approximation and multiplier bootstraps for suprema of (equally-weighted and identically distributed) empirical processes, which will be the main focus of next section.

In the following Theorem \ref{thm:local:main1}, we prove, by analyzing the leading noise term $\bar{\eta}_n^{noise,0}$, a finite-sample upper bound on the Kolmogorov distance between the sampling distribution of $\bar{f}_n(z_0)-f^\ast(z_0)$ and the distribution of a standard normal random variable (i.e.~supreme norm between the two cumulative distributions) for any $z_0\in\m X$.  

\begin{Theorem}[Pointwise convergence] \label{thm:local:main1}
Assume that the kernel $K$ to satisfy Assumptions~\ref{asmp:A1}-\ref{asmp:A2}. 
\begin{enumerate}                 
\item (Constant step size) Consider the step size $ \gamma(n) = \gamma$ with $0<  \gamma < n^{-\frac{2}{2+3\alpha}} $. For any fixed $z_0\in[0,1]$, we have 
$$
\sup_{u\in \mathbb R} \Big|\, \PP \Big(\sigma^{-1}_{z_0} \sqrt{n(n\gamma)^{-1/\alpha}}\big(\bar{f}_n (z_0) - f^\ast(z_0) - \bar{\eta}_n^{bias,0}(z_0)\big)\leq u \Big) - \Phi(u)\Big|\leq \frac{C_1}{\sqrt{n(n\gamma)^{-1/\alpha}}} + \kappa_n,
$$
where $\kappa_n =C_2\sqrt{\gamma^{1/2}(n\gamma)^{-1}}+ \sqrt{\gamma^{1/2}(n\gamma)^{1/\alpha}n^{-1}}$. Here, 
the bias term has an explicit expression as given in~\eqref{eq:local_bias}, and the (limiting) variance is
$$\sigma_{z_0}^2=\sigma^2(n\gamma)^{-1/\alpha}n^{-1} \sum_{k=1}^n \sum_{\nu=1}^\infty 
 \big[\big(1-(1-\gamma\mu_\nu)^{n+1-k}\big)^2\big]\, \phi_\nu^2(z_0).$$ 

\item (Non-constant step size) 
Consider the step size $\gamma_i = i^{-\frac{1}{\alpha+1}}$ for $i=1,\dots, n$. For any fixed $z_0\in[0,1]$, we have 
$$
\sup_{u\in \mathbb R} \big| \PP \Big(\sigma^{-1}_{z_0} \sqrt{n(n\gamma_n)^{-1/\alpha}}\big(\bar{f}_n (z_0) - f^\ast(z_0) - \bar{\eta}_n^{bias,0}(z_0)\big)\leq u \Big) - \Phi(u)\big|\leq \frac{C_1}{\sqrt{n(n\gamma_n)^{-1/\alpha}}}. 
$$
Here, the bias term takes an explicit expression as $\bar{\eta}_n^{bias,0}(z_0)= n^{-1}\sum_{k=1}^n 
\prod_{i=1}^k(I-\gamma_i\Sigma)\,f^\ast(z_0)$, and the variance is 
$$\sigma_{z_0}^2=\frac{\sigma^2}{n^2} \sum_{k=1}^n \gamma_k^2 \,\sum_{\nu=1}^\infty \mu_\nu^2\,  \phi_\nu^2(z_0) \Big(\sum_{j=k}^n \prod_{i=k+1}^j (1-\gamma_i\mu_\nu)\Big)^2.$$
\end{enumerate}       
\end{Theorem}

Theorem \ref{thm:local:main1} establishes that the sampling distribution of $\bar f_n-f^*$ at any fixed $z_0$ can be approximated by a normal distribution $N(\bar{\eta}_n^{bias,0}(z_0), n^{-1}(n\gamma_n)^{1/\alpha}\sigma^2_{z_0})$. According to Theorem \ref{le:bias_variance}, the bias $\bar{\eta}_n^{bias,0}(z_0)$ has the order of $(n\gamma_n)^{-1/2}$ while the variance has the order of $n^{-1}(n\gamma_n)^{1/\alpha}$; Theorem \ref{thm:local:main1} also implies that  the minimax convergence rate $n^{-\frac{1}{2(\alpha+1)}}$ of estimating $f^\ast$ can be achieved with $\gamma=\gamma_n = n^{-\frac{1}{\alpha+1}}$, which attains an optimal bias-variance tradeoff. In practice, the bias term can be suppressed by applying a undersmoothing technique; see Remark~\ref{rem:undersmooth} for details.

\begin{remark} \label{rem:asym_var} From the theorem, we see that the (limiting) variance $\sigma_{z_0}^2$ is precisely the variance of the scaled leading noise $\sqrt{n(n\gamma_n)^{-1/\alpha}}\bar{\eta}_n^{noise,0}(z_0)$ at $z_0$, that is,  $\Var\big(\sqrt{n(n\gamma_n)^{-1/\alpha}}\bar{\eta}_n^{noise,0}(z_0)\big)$; and $\sigma^2_{z_0}$  has the same $\mathcal O(1)$ order for both the constant and non-constant cases. The contribution of for each data point to the variance differs between the constant and non-constant step size cases.  Concretely, in the constant step size case,  let ${\bf{C}}=(c_1, \dots, c_n)$ be the vector of variation, where $c_k$ $(k=1,\dots, n)$ represents the contribution to $\sigma^2_{z_0}$ from the $k$-th arrival observation $(X_k, Y_k)$. According to equation (\ref{eq:local_noise}), $c_k = \EE \Omega^2_{n,k}(z_0)\asymp (n\gamma)^{-1/\alpha}n^{-1} \big((n+1-k)\gamma\big)^{1/\alpha}$ and is of order $(n-k)^{1/\alpha}$ in the observation index $k$, which decreases monotonically to nearly $0$ as $k$ grows to $n$. 
In comparison, in the online (nonconstant) step case, we denote ${\bf{O}}=(o_1,\dots, o_k)$ as the vector of variation  with $o_k$ being the contribution from the $k$-th observation.  A careful calculation shows that $o_k= n^{-2}\gamma^2_k\sum_{\nu=1}^\infty \mu_\nu^2 \, \phi_\nu^2(z_0) \Big(\sum_{j=k}^n \prod_{i=k+1}^j (1-\gamma_i\mu_\nu)\Big)^2$, which has order $n^{-2}\gamma^2_k \gamma^{-2}_{n+1-k} \big((n+1-k)\gamma_{n+1-k}\big)^{1/\alpha} + \big((n+1-k)\gamma_k\big)^{1/\alpha}$ and decreases slower than the constant step size case. This means that the nonconstant step scheme yields a more balanced weighted average over the entire dataset, which tends to lead to a smaller asymptotic variance. 

Figure \ref{fig:var_comp} compares the individual variation contribution for both the constant and non-constant step cases. We keep the total step size budget the same for both cases (which also makes the two leading bias terms roughly equal); that is, we choose constant $B$ in the nonconstant step size $\gamma_i= B \cdot i^{-\frac{1}{\alpha+1}}$ so that $n\gamma = \sum_{i=1}^n \gamma_i$ with $\gamma = n^{-\frac{1}{\alpha+1}}$ being the constant step size. The data index $k$ is plotted on the $x$ axis of Figure \ref{fig:var_comp} (A), with the variation contribution summarized by the $y$ axis. As we can see, the variation contribution from each observation decreases as observations arrive later in both cases. However, the pattern is flatter in the non-constant step case. Figure \ref{fig:var_comp} (B) is a violin plot visualizing the distributions of the components in $\bf{C}$ and $\bf{O}$.
Specifically, the variation among $\{o_k\}_{k=1}^n$ (depicted by the short blue interval) is smaller in the non-constant case, suggesting reduced fluctuation in individual variation for this setting. As detailed in Section \ref{sec:numerical}, our numerical analysis further confirms that using a nonconstant learning rate outperforms that using a constant learning rate (e.g., Figure \ref{fig:sim1}). An interesting direction for future research might be to identify an optimal learning rate decaying scheme by minimizing the variance $\sigma^2_{z_0}$ as a function of $\{\gamma_i\}_{i=1}^n$. It is also interesting to determine whether this scheme results in an equal contribution from each observation. However, this is beyond the scope of this paper.
\end{remark}

\begin{figure}
\centering
\includegraphics[width=0.9\textwidth]{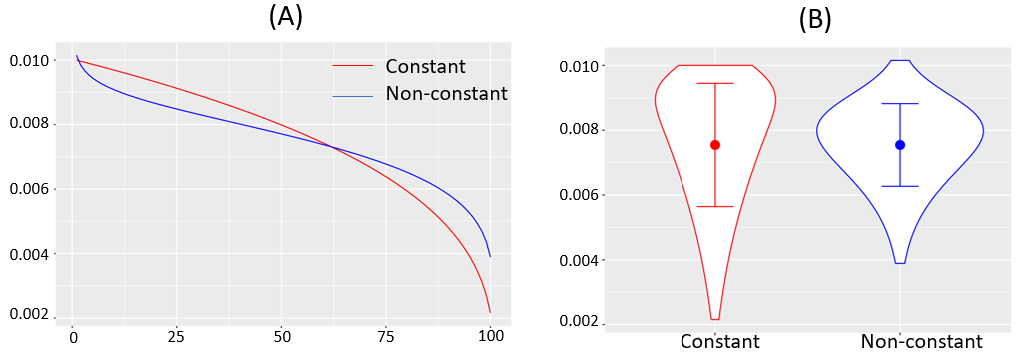} 
\caption{Compare the individual variation contribution for each observation in two cases: constant step size case (red curve) and non-constant step size case (blue curve). In (A), $x$-axis is the  observation index, $y$-axis is the variance contributed by the $t$-th observation. (B) is the violin plot of individual variance contribution for two cases; the solid dots represent mean while the intervals represent variance. }
  \label{fig:var_comp}
\end{figure}

\begin{remark}
The Kolmogorov distance bound between the sampling distribution of $\sigma^{-1}_{z_0} \sqrt{n(n\gamma)}\big(\bar{f}_n(z_0)-f^\ast(z_0)\big)$ and the standard normal distribution depends on the step size $\gamma_n$ and sample size $n$. In particular, $\kappa_n$ is the remainder bound stated in Theorem~\ref{le:bias_variance}, which is negligible compared to $\frac{C_1}{\sqrt{n(n\gamma)^{-1/\alpha}}}$ when $\gamma > n^{-\frac{2}{\alpha+2}}$ in the constant step size case.  Consequently, a smaller $\gamma$  or larger sample size $n$ leads to a smaller Kolmogorov distance. 
The same conclusion also applies to the non-constant step size case if we choose $\gamma_i= i^{-\frac{1}{\alpha+1}}$.
\end{remark}

Although Theorem \ref{thm:local:main1} explicitly characterizes the distribution of the SGD estimator, the expression of the standard deviation $\sigma_{z_0}$ depends on the eigenvalues and eigenfunctions of $\mathbb{H}$, the underlying distribution of the design $X$, and the unknown noise variance $\sigma^2$, which are typically unknown in practice. One approach is to use plug-in estimators for these unknown quantities, such as empirical eigenvalues and eigenfunctions obtained through SVD decomposition of the empirical kernel matrix $\mathbf{K}\in \mathbb{R}^{n\times n}$, whose $ij$-th element is $\mathbf{K}_{ij}= K(X_i, X_j)$. However, computing these plug-in estimators requires access to all observed data points  $\{(X_i,\,Y_i)\}_{i=1}^n$ and has a computational complexity of $\mathcal O(n^3)$, which undermines the sequential updating advantages of SGD. In the following section, we develop a scalable inference framework that uses multiplier-type bootstraps to generate randomly perturbed SGD estimators upon arrival of each observation.  This approach enables us to bypass the evaluation of $\sigma_{z_0}$ when constructing confidence intervals.

\section{Online Statistical Inference via Multiplier Bootstrap}\label{sec:bootstrap_SGD} 
In this section, we first propose a multiplier bootstrap method for inference based on the functional SGD estimator. After that, we study the theoretical properties of the proposed method, which serve as the cornerstone for proving bootstrap consistency for the local inference of constructing pointwise confidence intervals and the global inference of constructing simultaneous confidence bands. Finally, we describe the resulting online inference algorithm for non-parametric regression based on the functional SGD estimator.

\subsection{Multiplier bootstrap for functional SGD}\label{sec:MBootstrap}
Recall that Theorem~\ref{le:bias_variance} provides a high probability decomposition of the functional SGD estimator $\bar f_n$ (relative to supreme norm metric) into the following sum
\begin{align*}
    \bar f_n \ =\  f^\ast \ +\  \bar{\eta}_n^{bias,0}\ + 
 \ \bar{\eta}_n^{noise,0} \ +\ \mx{smaller remainder term},
\end{align*}
where $\bar{\eta}_n^{bias,0}$ is the leading bias process defined in equation~\eqref{eq:local_bias} and $\bar{\eta}_n^{noise,0}$ is the leading noise process defined in equation~\eqref{eq:local_noise}. Motivated by this result, we propose in this section a multiplier bootstrap method to mimic and capture the random fluctuation from this leading noise process
$\bar{\eta}_n^{noise,0}(\cdot) = n^{-1}\sum_{k=1}^n  \epsilon_k \cdot \Omega_{n,k}(\cdot)$,
where recall that term $\Omega_{n,k}$ only depends on the $k$-th design point $X_k$, and the primary source of randomness in $\bar{\eta}_n^{noise,0}$ is coming from random noises $\{\epsilon_k\}_{k=1}^n$ that are i.i.d.~normally distributed under a standard non-parametric regression setting.

Our online inference approach is inspired by the multiplier bootstrap idea proposed in~\cite{Fang2017} for online inference of parametric models using SGD. Remarkably, we demonstrate that their development can be naturally adapted to enable online inference of non-parametric models based on functional SGD.
The key idea is to perturb the stochastic gradient in the functional SGD by incorporating a random multiplier upon the arrival of each data point. 
r Specifically, let $w_1$, $w_2$, $\ldots$ denote a sequence of i.i.d.~random bootstrap multipliers, whose mean and variance are both equal to one. At time $i$ with the observed data point $(X_i,\, Y_i)$, we use a randomly perturbed functional SGD updating formula as:
\begin{equation}\label{eq:bootstrap:sgd}
   \begin{aligned}
\widehat{f}^b_i = &\, \widehat{f}_{i-1}^b + \gamma_i\, w_{i}\, G_i(\widehat{f}^b_{i-1}) =
\widehat{f}_{i-1}^b + \gamma_i \,w_{i}\, (Y_i - \langle \widehat{f}^b_{i-1}, \,K_{X_i}\rangle_{\mathbb{H}})\, K_{X_i}, \\
  = & \, (I - \gamma_i\, w_{i}\,  K_{X_i}\otimes K_{X_i}) \,\widehat{f}^b_{i-1} + \gamma_i\, w_{i} \,Y_i \,K_{X_i}\quad \mx{for}\quad i=1,2,\ldots,
\end{aligned} 
\end{equation}
which modifies equations~\eqref{eq:sgd:ini} and~\eqref{eq:stand:sgd} for functional SGD by multiplying the stochastic gradient $G_i(\widehat{f}^b_{i-1})$ with random multiplier $w_i$. We adopt the same zero initialization $\widehat{f}^b_0 = \widehat{f}_0=0$ and call the (Polyak) averaged estimator $\bar{f}_n^b = n^{-1}\sum_{i=1}^n \widehat{f}_i^b$ as the bootstrapped functional SGD estimator (with $n$ samples).

\subsection{Bootstrap consistency}\label{sec:bs-consistency}
Let us now proceed to derive a higher-order expansion of $\bar{f}_n^b$ analogous to Section~\ref{sec:higher-order} and compare its leading terms with those associated with the original functional SGD estimator $\bar f_n$. 
Utilizing equation~\eqref{eq:bootstrap:sgd} and plugging in $Y_i=f^\ast(X_i)+\epsilon_i$, we obtain the following expression:
$$ 
\widehat{f}^b_i - f^\ast =
(I - \gamma_i\, w_{i} \, K_{X_i}\otimes K_{X_i}) \,(\widehat{f}^b_{i-1}-f^\ast) + \gamma_i \,w_{i}\, \epsilon_i \,K_{X_i}.
$$ 
Since $w_i$ has a unit mean, we have an important identity $\Sigma =\EE (w_nK_{X_n}\otimes K_{X_n})$. Similar to equation \eqref{eq:sgd:recursion_f0}-\eqref{eq:sgd:higher_order_exp:1}, due to this key identity, we can still recursively define the leading bootstrapped bias term through
\begin{equation}\label{eq:bootstrap:bias:lead}
\eta_0^{b,bias,0}=\widehat f_0^{b} -f^\ast = -f^\ast \quad \mx{and}\quad  \eta_i^{b,bias,0}=  (I - \gamma_n\, \Sigma) \,\eta_{n-1}^{bias, 0} \quad\mx{for}\quad i=1,2,\ldots,
\end{equation}
which coincides with the original leading bias term, i.e.~${\eta}_i^{b,bias,0}\equiv {\eta}_i^{bias,0}$;
and the leading bootstrapped noise term through
\begin{align*}
\eta_0^{b,noise,0}= 0 \quad \mx{and}\quad
   \eta_i^{b,noise,0} =  (I - \gamma_i\,\Sigma) \eta^{b,noise,0}_{i-1} + \gamma_i \,w_i\,\epsilon_i\, K_{X_i},  \quad\mx{for}\quad i=1,2,\ldots,
\end{align*}
so that a similar decomposition as in equation~\eqref{eq:sgd:higher_order_exp:1} holds,
\begin{equation}\label{eq:sgd:higher_order_exp}
\widehat{f}^b_i -f^\ast =\ \underbrace{\eta_i^{b,bias,0}}_\text{leading bias} \ +\  \underbrace{\eta_i^{b,noise,0}}_\text{leading noise} \ + \ \ \underbrace{\big(\widehat{f}^b_i -f^\ast -\eta_i^{b,bias,0} - \eta_i^{b,noise,0}\big)}_\text{remainder term}  \quad \mbox{for} \quad i=1,2,\ldots. 
\end{equation}
Corresponding, we define $\bar{\eta}_i^{b,bias,0} = i^{-1}\sum_{j=1}^i\eta_j^{b,bias,0}$ and $\bar{\eta}_i^{b,noise,0} = i^{-1}\sum_{j=1}^i\eta_j^{b,noise,0}$ as the leading bootstrapped bias and noise terms, respectively, in the bootstrapped functional SGD estimator. 

Notice that $\bar{\eta}_i^{b,bias,0}$ also coincides with the original leading bias term $\bar{\eta}_i^{bias,0}$, i.e.~$\bar{\eta}_i^{b,bias,0}\equiv \bar{\eta}_i^{bias,0}$.
Therefore, $\bar{\eta}_i^{b,bias,0}$ has the same explicit expression as equation~\eqref{eq:local_bias}; while the leading bootstrapped noise term $\bar{\eta}_i^{b,noise,0}$ has a slightly different expression that incorporates the bootstrap multipliers as
\begin{equation}\label{eq:local_noise_bs}
\begin{aligned}
    \bar{\eta}_n^{b,noise,0}(x) =&\, \frac{1}{n}\sum_{k=1}^n w_k\cdot \epsilon_k \cdot \Omega_{n,k}(x),\quad \forall x\in\m X,
\end{aligned}
\end{equation}
where recall that $\Omega_{n,k}(\cdot)$ is defined in equation~\eqref{eq:local_noise} and only depends on $X_k$. 
By taking the difference between $\bar{\eta}_n^{b,noise,0}$ and $\bar{\eta}_n^{noise,0}$, we obtain
\begin{align}\label{eq:boot:noise}
    \bar{\eta}_n^{b,noise,0}(x) - \bar{\eta}_n^{noise,0}(x) = \frac{1}{n}\sum_{k=1}^n (w_k-1)\cdot \epsilon_k \cdot \Omega_{n,k}(x),\quad \forall x\in\m X.
\end{align}
This expression also takes the form of a weighted and non-identically distributed empirical process with ``effective" noises $\big\{(w_i-1)\epsilon_i\big\}_{i=1}^n$. Since $w_i$ has unit mean and variance, these effective noises have the same first two-order moments as the original noises $\{\epsilon_i\}_{i=1}^n$, suggesting that the difference $\bar{\eta}_n^{b,noise,0}(\cdot) - \bar{\eta}_n^{noise,0}(\cdot)$, conditioning on data $\{(X_i,\,Y_i\}_{i=1}^n$, tends to capture the random pattern of the original leading noise term $\bar{\eta}_n^{noise,0}(\cdot)$, leading to the so-called bootstrap consistency as formally stated in the theorem below.

\begin{Assumption}\label{asmp:weights}
For $i=1,\dots, n$, bootstrap multipliers $w_{i}$s are i.i.d.~samples of a random variable $W$ that satisfies $\EE(W)=1$, $\Var(W)=1$ and $\PP(|W|\geq t)\leq 2 \exp(-t^2/C)$ for all $t\geq 0$ with a constant $C>0$.  
\end{Assumption}

One simple example that satisfies Assumption \ref{asmp:weights} is $W\sim N(1,1)$. A second example is bounded random variables, such as a scaled and shifted uniform random variable on the interval $[-1, 3]$. One popular choice in practice is discrete random variables, such as $W$ such that $\PP(W=3)= \PP(W=-1)= 1/2$. 

Let $\m D_n:\,=\{X_i, Y_i\}_{i=1}^n$ denote the data of sample size $n$, and $\PP^*(\cdot)= \PP(\,\cdot\, |\, \mathcal{D}_n)$ denote the conditional probability measure given $\mathcal{D}_n$.
We first establish the bootstrap consistency for local inference of leading noise term in the following Theorem \ref{thm:boot_local:main1}. 

\begin{Theorem}[Bootstrap consistency for local inference of leading noise term]\label{thm:boot_local:main1}
Assume that kernel $K$ satisfies Assumptions~\ref{asmp:A1}-\ref{asmp:A2} and multiplier weights $\{w_i\}_{i=1}^n$ satisfy Assumption~\ref{asmp:weights}. 
\begin{enumerate}
\item (Constant step size) Consider the step size $\gamma(n)=\gamma$ with $\gamma\in(0,\, n^{-\frac{\alpha-3}{3}})$ for some $\alpha>1$. 
Then for any $z_0\in [0,1]$, we have with probability at least $1-C n^{-1}$, 
\begin{align*}
& \sup_{u \in \mathbb{R}} \Big|\, \PP^* \Big(\sqrt{n (n\gamma)^{-\frac{1}{\alpha}}}\, \big(\bar{\eta}_n^{b,noise,0}(z_0)- \bar{\eta}_n^{noise,0}(z_0)  \big)\leq u \Big)  
 -\, \PP\Big( \sqrt{n (n\gamma)^{-\frac{1}{\alpha}}}\, \bar{\eta}_n^{noise,0}(z_0) \leq u \Big)  \Big|\\
 &\qquad \leq 
 C' (\log n)^{3/2} (n(n\gamma)^{-1/\alpha})^{-1/6}, 
\label{eq:GP_approx:local:con}
\end{align*}
where $C, C'$ are constants independent of $n$.

\item (Non-constant step size) 
Consider the step size $\gamma_i =i^{-\xi}$, $i=1,\dots, n$, for some $\xi\in(\min\{0, 1-\alpha/3\},\,1/2)$. 
Then the following bound holds with probability at least $1-2n^{-1}$, 
\begin{align}
\sup_{u \in \mathbb{R}} \Big|\, \PP^* \Big( \sqrt{n (n\gamma_n)^{-\frac{1}{\alpha}}} \,\big(\bar{\eta}_n^{b,noise,0}(z_0)- \bar{\eta}_n^{noise,0}(z_0)  \big) \leq u \Big) -&\, \PP\Big( \sqrt{n (n\gamma_n)^{-\frac{1}{\alpha}}} \,\bar{\eta}_n^{noise,0}(z_0)  \leq u \Big)  \Big| \nonumber\\
 &\qquad \leq \frac{C' (\log n)^{3/2}}{\sqrt{n(n\gamma_n)^{-3/(2\alpha)}}}. \end{align} 
\end{enumerate}
\end{Theorem}

\begin{remark}
Recall that from (\ref{eq:local_noise}) and (\ref{eq:boot:noise}), we can express $\bar{\eta}_n^{noise,0}(z_0)= \frac{1}{n}\sum_{k=1}^n \epsilon_k \cdot \Omega_{n,k}(z_0)$ and $ \bar{\eta}_n^{b,noise,0}(z_0) - \bar{\eta}_n^{noise,0}(z_0) = \frac{1}{n}\sum_{k=1}^n (w_k-1)\cdot \epsilon_k \cdot \Omega_{n,k}(z_0)$. 
 Theorem \ref{thm:local:main1} shows that $ \sum_{k=1}^n \epsilon_k \cdot \Omega_{n,k}(z_0)$ can be approximated by a normal distribution $\Phi\big(0, n^{-1}(n\gamma_n)^{1/\alpha}\sigma^2_{z_0}\big)$. 
To prove Theorem \ref{thm:boot_local:main1}, we introduce an intermediate empirical process evaluated at $z_0$ as $\sum_{k=1}^n (e_k-1)\cdot \epsilon_k \cdot \Omega_{n,k}(z_0)$ where $e_k$'s are independent and identically distributed standard normal random variables, such that $\sum_{k=1}^n (e_k-1)\cdot \epsilon_k \cdot \Omega_{n,k}(z_0)\mid \mathcal{D}_n$ has the same (conditional) variance as the (conditional) variance of $\big(\bar{\eta}_n^{b,noise,0}(z_0)- \bar{\eta}_n^{noise,0}(z_0) \big)\mid \mathcal{D}_n$. 
\end{remark}

\begin{Theorem}[Bootstrap consistency for global inference of leading noise term]\label{thm:global:main1}
Assume that kernel $K$ satisfies Assumptions~\ref{asmp:A1}-\ref{asmp:A2} and multiplier weights $\{w_i\}_{i=1}^n$ satisfy Assumption~\ref{asmp:weights}. 
\begin{enumerate}
\item (Constant step size) Consider the step size $\gamma(n)=\gamma$ with $\gamma\in(0,\, n^{\frac{\alpha-3}{3}})$ for some $\alpha>2$. 
Then the following bound holds with probability at least $1-5 n^{-1}$ (with respect to the randomness in data $\m D_n$)
\begin{align*}
\sup_{u\in \mathbb{R}} \Big|\, \PP^* \Big(\sqrt{n (n\gamma)^{-\frac{1}{\alpha}}}\, \|\,\bar{\eta}_n^{b,noise,0}- \bar{\eta}_n^{noise,0}\,\|_{\infty}  \big)\leq u \Big)  
 -&\, \PP\Big( \sqrt{n (n\gamma)^{-\frac{1}{\alpha}}}\, \| \,\bar{\eta}_n^{noise,0}\,\|_{\infty} \leq u \Big)  \Big|\\
 &\quad\leq 
  C(\log n)^{3/2}\big(n(n\gamma)^{-3/\alpha}\big)^{-1/8}.
 \label{eq:GP_approx:con}
\end{align*}

\item (Non-constant step size) 
Consider the step size $\gamma_i =i^{-\xi}$, $i=1,\dots, n$, for some $\xi\in(\min\{0, 1-\alpha/3\},\,1/2)$. 
Then the following bound holds with probability at least $1-5n^{-1}$, 
\begin{align*}
\sup_{u \in \mathbb{R}} \Big|\, \PP^* \Big( \sqrt{n (n\gamma_n)^{-\frac{1}{\alpha}}} \,\|\,\bar{\eta}_n^{b,noise,0}- \bar{\eta}_n^{noise,0}\,\|_{\infty}\leq u \Big) -&\, \PP\Big( \sqrt{n (n\gamma_n)^{-\frac{1}{\alpha}}} \,\|\,\bar{\eta}_n^{noise,0}\,\|_{\infty} \leq u \Big)  \Big| \nonumber\\
 &\quad\leq 
  C(\log n)^{3/2}\big(n(n\gamma_n)^{-3/\alpha}\big)^{-1/8}.
\end{align*}
\end{enumerate}
\end{Theorem}

\begin{remark}
Theorem \ref{thm:global:main1} demonstrates that the sampling distribution of $\sqrt{n (n\gamma_n)^{-1/\alpha}} \|\bar{\eta}_n^{noise,0}\|_{\infty}$ can be approximated closely by the conditional distribution of $\sqrt{n (n\gamma_n)^{-1/\alpha}} \|\bar{\eta}_n^{b,noise,0}- \bar{\eta}_n^{noise,0}\|_{\infty}$ given data set $\mathcal{D}_n$. This theorem serves as the theoretical foundation for adopting the multiplier bootstrap method detailed in Section~\ref{sec:MBootstrap} for global inference. 
Recall that the optimal step size for achieving the minimax optimal estimation error is $\gamma = n^{-\frac{1}{\alpha+1}}$ for the constant step size and $\gamma_i= i^{-\frac{1}{\alpha+1}}$ for the non-constant step size (Theorem~\ref{thm:local:main1}). To ensure that the Kolmogorov distance bound in Theorem~\ref{thm:global:main1} decays to $0$ as $n \to \infty$ under these step sizes, we require $\alpha > 2$. It is likely that our current Kolmogorov distance bound, which is dominated by an error term that arises from applying the Gaussian approximation to analyze $\|\bar{\eta}_n^{noise,0}\|_{\infty}$ and $\|\bar{\eta}_n^{b,noise,0}- \bar{\eta}_n^{noise,0}\|_{\infty}$ through space-discretization (see Section~\ref{sec:sketch:pf:GA}), can be substantially refined.  We leave this improvement of the Kolmogorov distance bound, which would consequently lead to a weaker requirement on $\alpha$, to future research.
\end{remark}

Since the leading noise terms $\bar{\eta}_n^{noise,0}$ and $\bar{\eta}_n^{b,noise,0}$ contribute to the primary source of randomness in the functional SGD and its bootstrapped counterpart (Theorem~\ref{le:bias_variance}), 
Theorem \ref{thm:global:main1} then implies the bootstrap consistency for statistical inference of $f^\ast$ based on bootstrapped functional SGD.
Particularly, we present the following Corollary, which establishes a high probability supremum norm bound for the remainder term in the bootstrapped functional SGD decomposition~\eqref{eq:sgd:higher_order_exp}. Such a bound further implies that the sampling distribution of $\sqrt{n (n\gamma_n)^{-1/\alpha}} \|\bar{f}_n - f^\ast\|_{\infty}$ can be effectively approximated by the conditional distribution of $\sqrt{n (n\gamma_n)^{-1/\alpha}} \|\bar{f}^b_n - \bar{f}_n\|_{\infty}$ given data $\m D_n$. Recall that we use $\PP^*(\cdot)= \PP(\,\cdot\, |\, \mathcal{D}_n)$ to denote the conditional probability measure given $\mathcal{D}_n=\{X_i, Y_i\}_{i=1}^n$.

\begin{Corollary}[Bootstrap consistency for functional SGD inference]\label{cor:global:main}
Assume that kernel $K$ satisfies Assumptions~\ref{asmp:A1}-\ref{asmp:A2} and  multiplier weights $\{w_i\}_{i=1}^n$ satisfies Assumption~\ref{asmp:weights}. 
\begin{enumerate}
\item (Constant step size) Consider the step size $\gamma(n)=\gamma$ with $\gamma\in(0,\, n^{\frac{\alpha-3}{3}})$ for some $\alpha>2$. Then it holds with probability at least 
$1-\gamma^{1/4}-\gamma^{1/2}-1/n$ 
with respect to the randomness of $\mathcal{D}_n$ that
\begin{equation}\label{eq:bootstrap_remainder}
 \PP^* \Big(\|\,\bar{f}^b_n - f_n - \bar{\eta}_n^{b,bias,0}-\bar{\eta}_n^{b,noise,0}- \bar{\eta}_n^{bias,0}+ \bar{\eta}_n^{noise,0}\,\|^2_{\infty} \geq \gamma^{1/4} (n\gamma)^{1/\alpha}n^{-1}\Big) \leq \gamma^{1/4}+ \gamma^{1/2}+ 1/n.
 \end{equation} 
Furthermore, for $0<\gamma< n^{-\frac{4}{7\alpha +1}}(\log n)^{-3/2}$, it holds with probability at least $1-5n^{-1}-3\gamma^{1/2}-\gamma^{-1/4}$, that
\begin{align*}
 \sup_{u \in \mathbb{R}} \Big|\, \PP^* \Big(\sqrt{n (n\gamma)^{-1/\alpha}} \,\|\,\bar{f}^b_n \,- & \,\bar{f}_n\,\|_{\infty}\leq u \Big) - \PP\Big( \sqrt{n (n\gamma)^{-1/\alpha}}\, \|\, \bar{f}_n - f^\ast-\mx{Bias}(f^\ast)\,\|_{\infty} \leq u \Big)  \Big| \\
&\qquad \leq C_1(\log n)^{3/2} n^{-1/8} (n\gamma)^{3/(8\alpha)} +C\gamma^{1/4}
\end{align*}
where $\mx{Bias}(f^\ast) = \bar \eta_n^{bias,0}$ denotes the bias term, $C_1, C>0$ are constants.

\item (Non-constant step size)
Consider the step size $\gamma_i = i^{-\frac{1}{\alpha+1}}$ for $i=1,\dots, n$.  
Then it holds with probability at least $1-\gamma_n^{1/4}-\gamma_n^{1/2}-1/n$ that 
$$ \PP^* \Big(\|\,\bar{f}^b_n - f^\ast - \bar{\eta}_n^{b,bias,0}-\bar{\eta}_n^{b,noise,0}\,\|^2_{\infty} \geq \gamma_n^{1/4} (n\gamma_n)^{1/\alpha}n^{-1}\Big) \leq \gamma_n^{1/4}+ \gamma_n^{1/2}+ 1/n.$$ 
Furthermore, it holds with probability at least $1-5n^{-1}-\gamma_n^{1/2}-\gamma_n^{1/4}$ that  
\begin{align*}
 \sup_{u \in \mathbb{R}} \Big| \,\PP^* \Big(\sqrt{n (n\gamma_n)^{-1/\alpha}}\, \|\,\bar{f}^b_n \,- &\,\bar{f}_n\,\|_{\infty}\leq u \Big) - \PP\Big( \sqrt{n (n\gamma_n)^{-1/\alpha}} \,\|\,\bar{f}_n - f^\ast-\mx{Bias}(f^\ast)\,\|_{\infty} \leq u \Big)  \Big| \\
&\qquad \lesssim  (\log n)^{3/2} n^{-1/8}(n\gamma_n)^{\frac{3}{8\alpha}}+ \gamma_n^{1/4}.
\end{align*}
\end{enumerate}
\end{Corollary}

\begin{remark}\label{rem:undersmooth}
Corollary \ref{cor:global:main} suggests that a smaller step size $\gamma$ (or $\gamma_n$) and a larger sample size $n$ result in more accurate uncertainty quantification. As discussed in Section~\ref{sec:bs-consistency}, the functional SGD estimator and its bootstrap counterpart share the same leading bias term, which eliminates the bias in the conditional distribution of $\bar f_n^b - \bar f_n$ given $\m D_n$. However, the bias term $\mx{Bias}(f^\ast)$ still exists in the sampling distribution of $\bar f_n - f^\ast$. According to Theorem~\ref{le:bias_variance}, this bias term can be bounded by $O(1/\sqrt{n\gamma})$ with high probability, while the convergence rate of the leading noise term under the supremum norm metric is of order $O(1/\sqrt{n(n\gamma)^{-1/\alpha}})$. Therefore, to make the bias term asymptotically negligible, we can adopt the common practice of ``undersmoothing"~\cite{neumann1998simultaneous,armstrong2020simple}. In our context, this means slightly enlarging the step size as $\gamma=\gamma(n)= n^{-\frac{1}{\alpha+1}+\varepsilon}$ (constant step size) or $\gamma_i = i^{-\frac{1}{\alpha+1}+\varepsilon}$ for $i=1, \dots, n$ (non-constant step size), where $\varepsilon$ is any small positive constant.
\end{remark}

\subsection{Online inference algorithm}\label{sec:online_alg}

\begin{algorithm}[ht!]
\KwData{Number of bootstrap samples $J$, initial step size $\gamma_0>0$, initial estimates $\widehat{f}^{b,j}_0= \widehat{f}_0$, $j=1,\dots, J$, confidence level $(1-\alpha)$.}
\For{ $i=1,2, \dots, n$}{
Update $\widehat{f}_i = \widehat{f}_{i-1} - \gamma_i \nabla\ell_i(\widehat{f}_{i-1})$\\
Update $\bar{f}_i = (i-1)\bar{f}_{i-1}/i + \widehat{f}_i/i$\\
\For{$j=1,\dots, J$}{
Update $\widehat{f}^{b,j}_i = \widehat{f}^{b,j}_{i-1} - \gamma_i w_{i,j}\nabla\ell_i(\widehat{f}^{b,j}_{i-1})$\\
Update $\bar{f}^{b,j}_i = (i-1)\bar{f}^{b,j}_{i-1}/i + \widehat{f}^{b,j}_i/i$.
}
}
Output: SGD estimators $\bar{f}_{n}$ and the Bootstrap estimates $\{\bar{f}^{b,j}_{n}\}_{j=1}^J$. Calculate $\{\bar{f}^{b,j}_{n} - \bar{f}_{n}\}_{j=1}^J$. \\

Construct the $100(1-\alpha)\%$ confidence interval for $f$ evaluated at any fixed $z_0$ via  
\begin{enumerate}
\item Normal CI: $(\bar{f}_n(z_0) - z_{\alpha/2}\sqrt{T_n^b(z_0)},\, \bar{f}_n(z_0) + z_{\alpha/2}\sqrt{T_n^b(z_0)})$, where $T_n^{b}(z_0)= \frac{1}{J-1}\sum_{j=1}^J \big(\bar{f}_n^{b,j}(z_0)- \bar{f}_n(z_0)\big)^2$. 
\item Percentile CI: $\big(\bar{f}_n(z_0)-C_{\alpha/2},\,\bar{f}_n(z_0)+  C_{1-\alpha/2} \big)$, where $C_{\alpha/2}$ and $C_{1-\alpha/2}$ are the sample $\alpha/2$-th and $(1-\alpha/2)$-th quantile of $\{\bar{f}^{b,j}_{n}(z_0) - \bar{f}_{n}(z_0)\}_{j=1}^J$. 
\end{enumerate}

Construct the $ 100(1-\alpha)\%$ confidence band for $f$ at any $x\in \cX$: \\
Step 1: Evenly choose $t_1, \dots, t_M \in \cX$. \\
Step 2: For $j \in 1, \dots, J$, calculate
$\max_{1\leq m \leq M} \big|\bar{f}_n^{b,j}(t_m)-\bar{f}_n(t_m)\big|.$ \\ 
Step 3: Calculate the sample $\alpha/2$-th and the $(1-\alpha/2)$-th quantiles of 
$$\max_{1\leq m \leq M} \big|\bar{f}_n^{b,1}(t_m)-\bar{f}_n(t_m)\big|, \dots, \max_{1\leq m \leq M} \big|\bar{f}_n^{b,J}(t_m)-\bar{f}_n(t_m)\big|,$$
 and denote them by $Q_{\alpha/2}$ and $Q_{1-\alpha/2}$.\\
 Step 4: Construct the $100(1-\alpha) \%$ confidence band as $\big\{g:\, \m X\to\mathbb R \,\big|\, g(x)\in[\bar f_n(x) - Q_{\alpha/2},\,\bar f_n(x) + Q_{1-\alpha/2}],\ \forall x\in \m X\big\}$.
\caption{Algorithm 1 (Online Bootstrap Confidence Band for Non-parametric Regression)}\label{algorithm1}
\end{algorithm}

As we demonstrated in Theorem \ref{thm:global:main1}, the sampling distribution of $\sqrt{n(n\gamma_n)^{-1/\alpha}}(\bar{f}_n - f^\ast)$ can be effectively approximated by the conditional distribution of $\sqrt{n(n\gamma_n)^{-1/\alpha}}(\bar{f}_n^b - \bar{f}_n)$ given data $\m D_n$ using the bootstrap functional SGD.  This result provides a strong foundation for conducting online statistical inference based on bootstrap. Specifically, we can run $J$ bootstrapped functional SGD in parallel, producing $J$ estimators $\bar{f}_n^{b,j} = \frac{1}{n}\sum_{i=1}^n \widehat{f}_i^{b,j}$ for $j=1, \dots, J$ with 
$$
\widehat{f}^{b,j}_i =  \widehat{f}_{i-1}^{b,j} + \gamma_n w_{i,j} (Y_i - \langle \widehat{f}^{b,j}_{i-1}, K_{X_i}\rangle_{\mathbb{H}}) K_{X_i}, \quad \mx{for}\quad i=1,2,\ldots,
$$
where $w_{i,j}$ are i.i.d.~bootstrap weights satisfying Assumption~\ref{asmp:weights}. 
Then we can approximate the sampling distribution of $(\bar{f}_n - f^\ast)$ using the empirical distribution of $\{\widehat{f}^{b,j}_n - \bar{f}_n, \, j =1, \dots, J\}$ conditioning on $\m D_n$, and further construct the point-wise confidence intervals and simultaneous confidence band for $f^\ast$.
We can also use the empirical variance of $\{\bar{f}_n^{b,j},\, j=1,\dots,J\}$ to approximate the variance of $\bar{f}_n$. Based on these quantities, we can construct the point-wise confidence interval for $f^\ast(x)$ for any fixed $x\in\m X$ in two ways: 
\begin{enumerate}
    \item Normal CI - giving the sequence of bootstrapped estimators $\bar{f}_n^{b,j}(x)$ 
      for $j=1,\dots, J$, we calculate the variance as $T_n^{b}(x)= \frac{1}{J-1}\sum_{j=1}^J \big(\bar{f}_n^{b,j}(x)- \bar{f}_n(x)\big)^2$, and construct the $100(1 - \alpha)\%$ confidence interval for $f^\ast(x)$ as $(\bar{f}_n(x) - z_{\alpha/2}\sqrt{T_n^b(x)}, \bar{f}_n(x) + z_{\alpha/2}\sqrt{T_n^b(x)})$;
    \item Percentile CI -  giving 
      the sequence of bootstrapped estimators $\bar{f}_n^{b,j}(x)$ for $j=1,\dots, J$, we calculate $\{\bar{f}_n^{b,j}(x)-\bar{f}_n(x)\}_{j=1}^J$, and its $\alpha/2$-th and $(1-\alpha/2)$-th quantiles as $C_{\alpha/2}$ and $C_{1-\alpha/2}$, then construct the $100(1 - \alpha)\%$ CI for $f^\ast(x)$ as $\big(\bar{f}_n(x)-C_{\alpha/2},\bar{f}_n(x)+  C_{1-\alpha/2} \big)$. 
\end{enumerate}

To construct the simultaneous confidence band, we first choose a dense grid points $t_1,\dots, t_M\in \cX$; then for each $j\in\{1,\dots, J\}$, we calculate $\max_{1\leq m \leq M} \big|\bar{f}_n^{b,j}(t_m)-\bar{f}_n(t_m)\big| $ to approximate $\sup_t |\bar{f}_n^{b,j}(t)-\bar{f}_n(t)|$. Accordingly, we obtain the following $J$ bootstrapped supremum norms:
\begin{equation}\label{eq:CB} 
\max_{1\leq m \leq M} \big|\bar{f}_n^{b,1}(t_m)-\bar{f}_n(t_m)\big|\ , \ \max_{1\leq m \leq M} \big|\bar{f}_n^{b,2}(t_m)-\bar{f}_n(t_m)\big|\ ,\  \dots \ ,\  \mbox{and}\ \max_{1\leq m \leq M} \big|\bar{f}_n^{b,J}(t_m)-\bar{f}_n(t_m)\big|. 
\end{equation}
Denote the sample $\alpha/2$-th and the $(1-\alpha/2)$-th quantiles of (\ref{eq:CB}) as $Q_{\alpha/2}$ and $Q_{1-\alpha/2}$. Then we construct a $100(1 - \alpha)\%$  confidence band for $f^\ast$ as 
$\big\{g:\, \m X\to\mathbb R \,\big|\, g(x)\in[\bar f_n(x) - Q_{\alpha/2},\,\bar f_n(x) + Q_{1-\alpha/2}],\ \forall x\in \m X\big\}$. 

Our online inference algorithm is computationally efficient, as it only requires one pass over the data, and the bootstrapped functional SGD can be computed in parallel. The detailed algorithm is summarized in Algorithm \ref{algorithm1}.

\section{Numerical Study}\label{sec:numerical}
In this section, we test our proposed online inference approach via simulations.
Concretely, we generate synthetic data in a streaming setting with a total sample size of $n$. We use $(X_t, Y_t)$ to represent the $t$-th observed data point for $t=1, \dots, n$. We evaluate the performance of our proposed method as described in Algorithm~\ref{algorithm1} for constructing confidence intervals for $f(x)$ at $x=X_t$ for $t=501,1000,1500,2000,2500,3000,3500,4000$, and compare our method with three
existing alternative approaches, which we refer to as ``offline" methods. ``Offline" methods involve calculating the confidence intervals after all data have been collected, up to the $t$-th observation's arrival, which necessitates refitting the model each time new data arrive. We also evaluate the coverage probabilities of the simultaneous confidence bands constructed in Algorithm~\ref{algorithm1}. We first enumerate the compared offline confidence interval methods as follows:

\begin{enumerate}[(i)]

\item Offline Bayesian confidence interval (Offline BA) proposed in \cite{wahba1983bayesian}:
According to \cite{wahba1978improper}, a smoothing spline method corresponds to a Bayesian procedure when using a partially improper prior. Given this relationship between smoothing splines and Bayes estimates, confidence intervals can be derived from the posterior covariance function of the estimation. In practice, we implement Offline BA using the ``gss" R package \cite{gu2013smoothing}. 

\item Offline bootstrap normal interval (Offline BN) proposed in \cite{wang1995bootstrap}: 
Let $\hat{f}_{\lambda}$ and $\hat{\sigma}$ denote the estimates of $f$ and $\sigma$ respectively, achieved by minimizing (\ref{eq:plr}) with $\{X_i, Y_i\}_{i=1}^t$ as below. 
\begin{equation}\label{eq:plr}
\sum_{i = 1}^{t} (Y_i - f(X_i))^2 + \frac{t}{2}\lambda \int_{0}^1 (f^{''}(u))^2 du
\end{equation}
where $\lambda$ is the roughness penalty and $f^{''}(u)$ is the second derivative evaluated at $u$. A bootstrap sample is generated from
\begin{equation*}
Y_i^\dag = \hat{f}_{\lambda}(X_i) + \epsilon^\dag,\quad  i = 1,\dots,t
\end{equation*}
where $\epsilon^\dag_i$s are i.i.d.~Gaussian white noise with variance $\hat{\sigma}^2$. Based on the bootstrap sample, we calculate the bootstrap estimate as $\hat{f}_\lambda^\dag$. Repeating $J$ times, we have a sequence of offline bootstrap estimates $\bar{f}_\lambda^{\dag,1}, \dots, \bar{f}_\lambda^{\dag,J}$. We estimate the variance of $\hat f_\lambda(X_t)$ as $T^\dag_t= \frac{1}{J-1}\sum_{j=1}^J\big(\hat{f}_\lambda^{\dag,j}(X_t)-\hat{f}_\lambda(X_t) \big)^2$. 
A $100(1 - \alpha)\%$ offline normal bootstrap confidence interval for $\hat f_\lambda(X_t)$ is then constructed as $\big(\,\hat{f}_\lambda(X_t) - z_{\alpha/2}\sqrt{T^\dag_t}, \,\hat{f}_\lambda(X_t) + z_{\alpha/2}\sqrt{T^\dag_t}\,\big)$.

\item Offline bootstrap percentile interval (Offline BP): We apply the same data bootstrapping procedure in Offline BN, which produces the estimate $\hat{f}^\dag_\lambda(X_t)$ based on the bootstrap sample. The confidence interval is then constructed using the percentile method suggested in \cite{efron1982jackknife}. Specifically, let $C^\dag_{\alpha/2}(X_t)$ and $C^\dag_{1-\alpha/2}(X_t)$ represent the $\alpha/2$-th quantile and the $(1-\alpha/2)$-th quantile of the empirical distribution of $\big\{\hat{f}_\lambda^{\dag,j}(X_t)-\hat{f}^\dag_\lambda(X_t)\big\}_{j=1}^J$, respectively. A $100(1 -\alpha)\%$ confidence interval for $\hat f_\lambda(X_t)$ is then constructed as $\big(\,\hat{f}_\lambda(X_t) -C^\dag_{\alpha/2}(X_t),\,\hat{f}_\lambda(X_t) +C^\dag_{1-\alpha/2}(X_t)\,\big)$.
\end{enumerate}

As $t$ increases, offline methods lead to a considerable increase in computational cost. For instance, Offline BA/BN theoretically has a total time complexity of order $\mathcal O(t^4)$ (with an $\mathcal O(t^3)$ cost at time $t$). In contrast, online bootstrap confidence intervals are computed sequentially as new data points become available, making them well-suited for streaming data settings. They have a theoretical complexity of at most $\mathcal O(t^2)$ (with an $\mathcal O(t)$ cost at time $t$).
We examine both the normal CI and percentile CI, as outlined in Algorithm \ref{algorithm1}, when constructing the confidence interval. 

We examine the effects of various step size schemes. Specifically, we consider a constant step size $\gamma=\gamma(t)= t^{-\frac{1}{\alpha+1}}$, where $t$ represents the total sample size at which the CIs are constructed, and an online step size $\gamma_i = i^{-\frac{1}{\alpha+1}}$ for $i=1,\dots, t$. A limitation of the constant-step size method is its dependency on prior knowledge of the total time horizon $t$. Consequently, the estimator is only rate-optimal at the $t$-th step.
We assess our proposed online bootstrap confidence intervals in four different scenarios: (i) Online BNC, which uses a constant step size for the normal interval; (ii) Online BPC, which uses a constant step size for the percentile interval; (iii) Online BNN, which employs a non-constant step size for the normal interval; and (iv) Online BPN, which utilizes a non-constant step size for the percentile interval.

We generate our data as i.i.d.~copies of random variables $(X,Y)$, where $X$ is drawn from a uniform distribution in the interval $(0,1)$, and $Y = f(X) + \epsilon$. Here $f$ is the unknown regression function to be estimated, $\epsilon$ represents Gaussian white noise with a variance of $0.2$. We consider the following three cases of $f=f_\ell$, $\ell=1,2,3$:
\begin{align*}
\mbox{Case 1:}\quad & f_1(x)= \sin (3\pi x/2), \\
\mbox{Case 2:}\quad & f_2(x) = \frac{1}{3}\beta_{10,5}(x) + \frac{1}{3}\beta_{7,7}(x) + \frac{1}{3}\beta_{5,10}(x), \\
\mbox{Case 3:}\quad & f_3(x) = \frac{6}{19}\beta_{30,17}(x) + \frac{4}{10}\beta_{3,11}(x).
\end{align*}
Here, $\beta_{p,q}= \frac{x^{p-1}(1-x)^{q-1}}{B(p,q)}$ with $B(p,q) = \frac{\Gamma(p)\Gamma(q)}{\Gamma(p+q)}$ denoting the beta function, and $\Gamma$ is the gamma function with $\Gamma(p)=p!$ when $p\in\mathbb N_+$. Cases 2 and 3 are designed to mimic the increasingly complex ``truth'' scenarios similar to the settings in \cite{wahba1983bayesian, wang1995bootstrap}.

We draw training data of size $n=3000$ from these models. In our online approaches, we first use 500 data points to build an initial estimate and then employ SGD to derive online estimates from the 501st to the 3000th data point. Given that our framework is designed for online settings, we can construct the confidence band based on the datasets of size $501$, $1000$, $1500$, $2000$, $2500$ and $3000$, i.e., using the averaged estimators $\bar{f}_t$ at $t=501,1000, 1500, 2000, 2500, 3000$.
We repeat the data generation process $200$ times for each case. For each replicate, upon the arrival of a new data point, we apply the proposed multiplier bootstrap method for online inference, using $500$ bootstrap samples (i.e., $J=500$ in Algorithm~\ref{algorithm1}) with bootstrap weight $W$ generated from a normal distribution with mean $1$ and standard deviation $1$. We then construct $95\%$ confidence intervals based on Algorithm~\ref{algorithm1}.
Our results will show the coverage and distribution of the lengths of the confidence intervals built at $t=501, 1000, 1500, 2000, 2500$, and $3000$.

\begin{figure*}
\centering 
\begin{tabular}{cc}
{\footnotesize \hspace{-80pt} {\bf Case 1:} \hspace{10pt}(A1)  Coverage }& {\hspace{-30pt}\footnotesize (A2) Length of Confidence Interval} \\
\centering
\includegraphics[width=0.49\columnwidth]{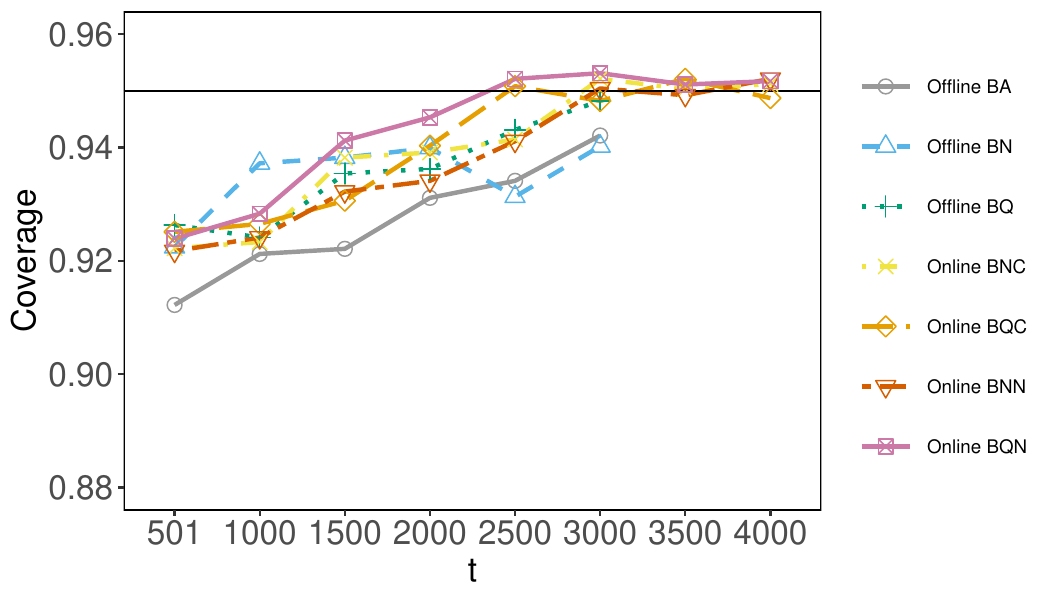} &\includegraphics[width=0.49\columnwidth]{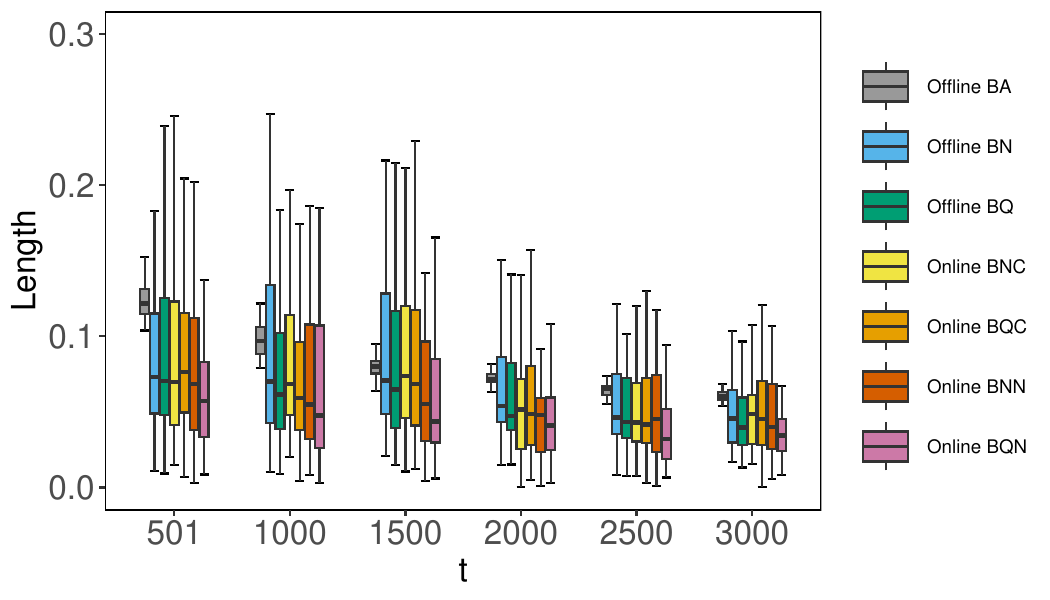}\\
{\footnotesize \hspace{-80pt} {\bf Case 2:} \hspace{10pt} (B1) Coverage }& {\hspace{-30pt}\footnotesize (B2) Length of Confidence Interval} \\
\centering
\includegraphics[width=0.49\columnwidth]{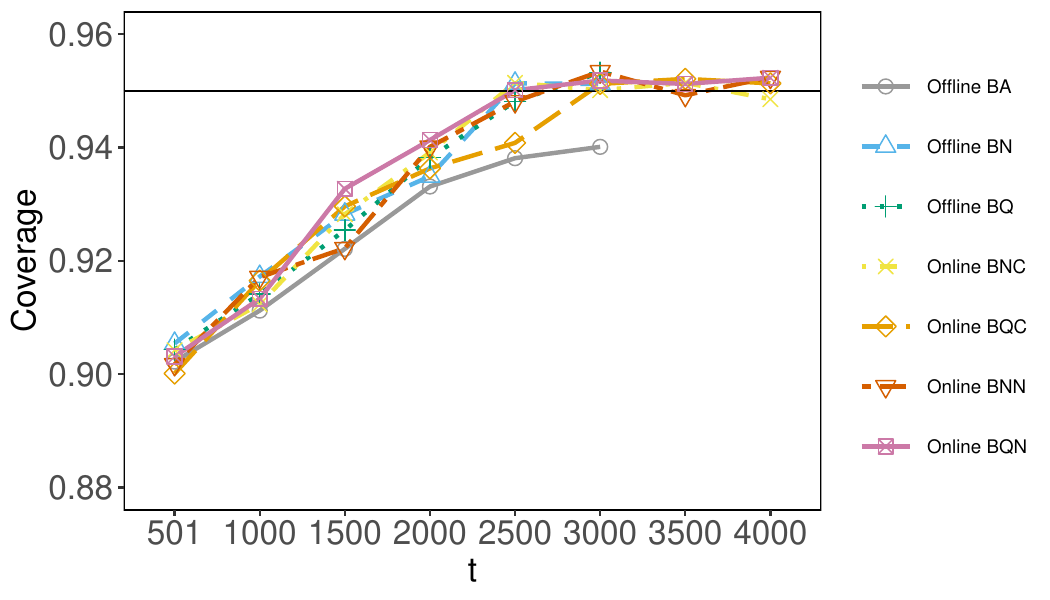} &\includegraphics[width=0.49\columnwidth]{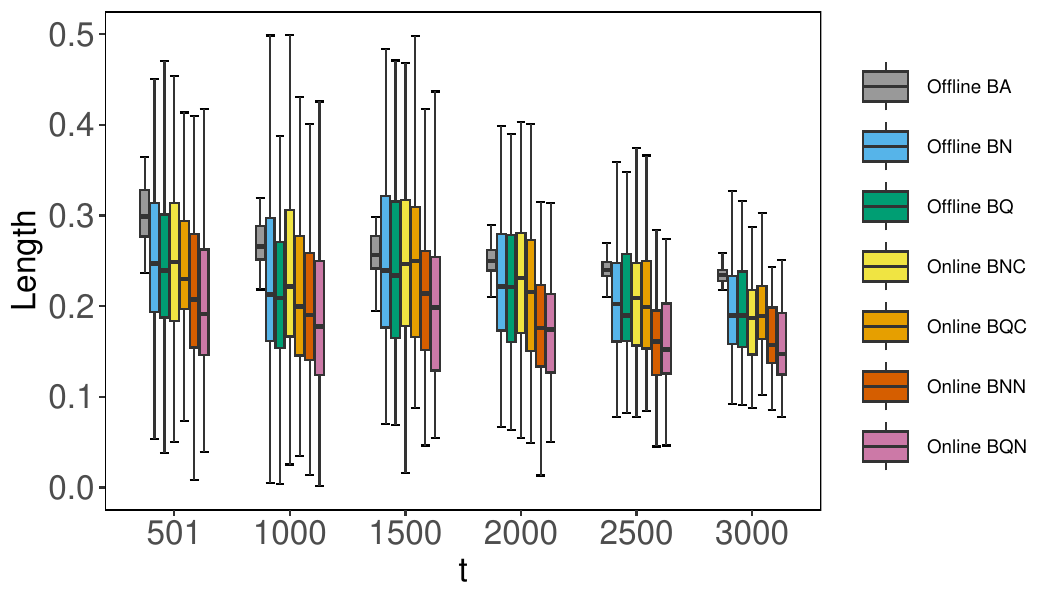}\\
{\footnotesize \hspace{-80pt} {\bf Case 3:} \hspace{10pt} (C1) Coverage }& {\hspace{-30pt}\footnotesize (C2) Length of Confidence Interval} \\
\centering
\includegraphics[width=0.49\columnwidth]{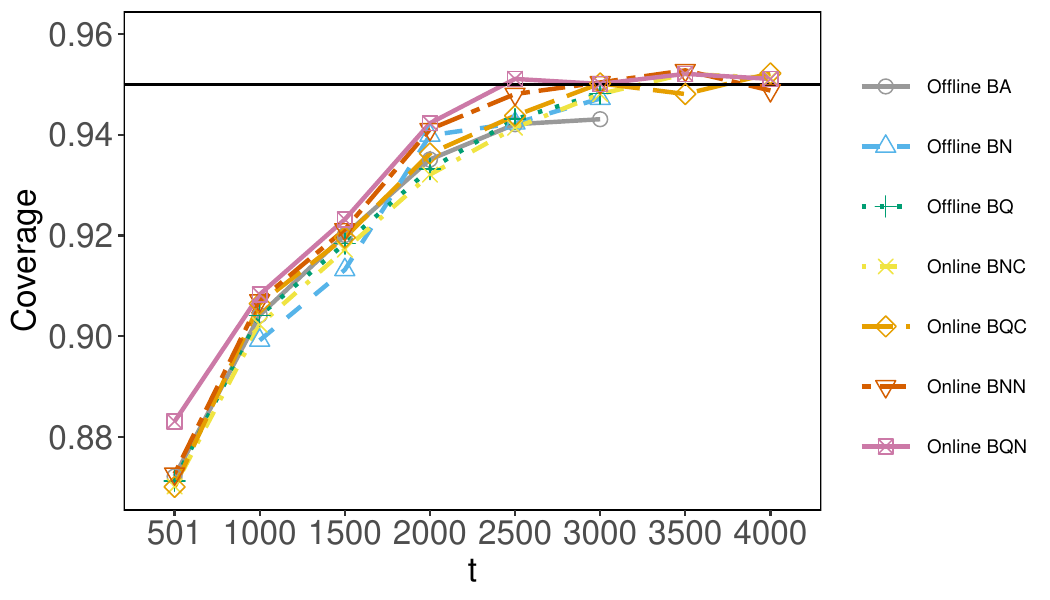} &\includegraphics[width=0.49\columnwidth]{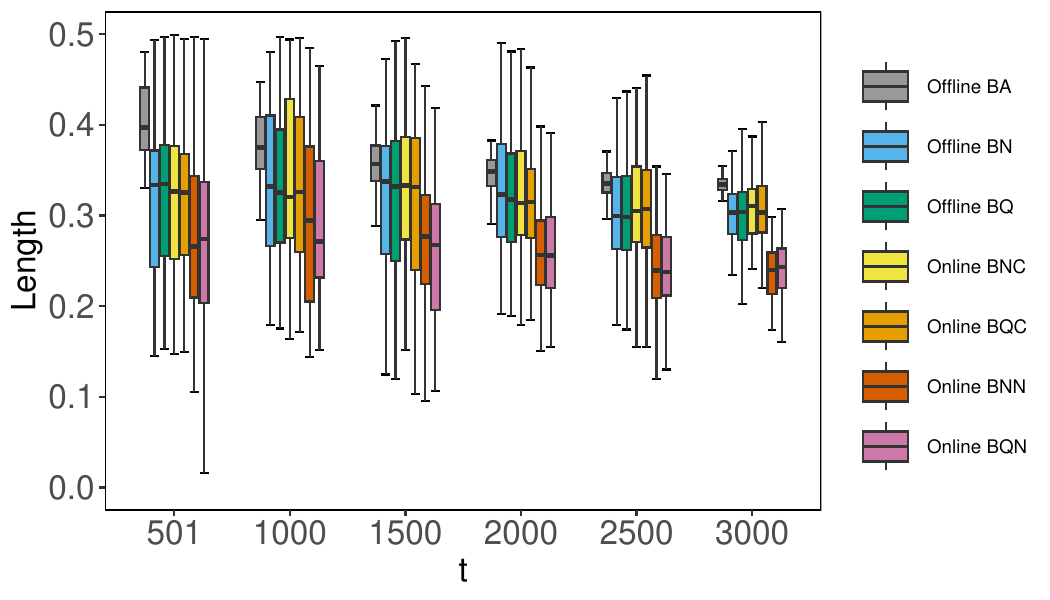}
\end{tabular}
\caption{We compared seven different CI construction approaches: Offline Bayesian confidence interval (Offline BA), Offline Bootstrap normal interval (Offline BN), Online bootstrap normal interval with constant step size (Online BNC), Online bootstrap percentile interval with constant step size (Online BPC), Online Bootstrap normal interval with non-constant step size (Online BNN), and Online Bootstrap percentile interval with Non-constant step size (Online BPN).  The coverage of the CI is shown in (A1), (B1), (C1). The mean and variance of the length of CIs are represented by the solid center and colored interval in (A2), (B2), (C2) respectively.}
\label{fig:sim1}
\end{figure*}

\begin{figure}
\centering
\includegraphics[width=0.9\textwidth]{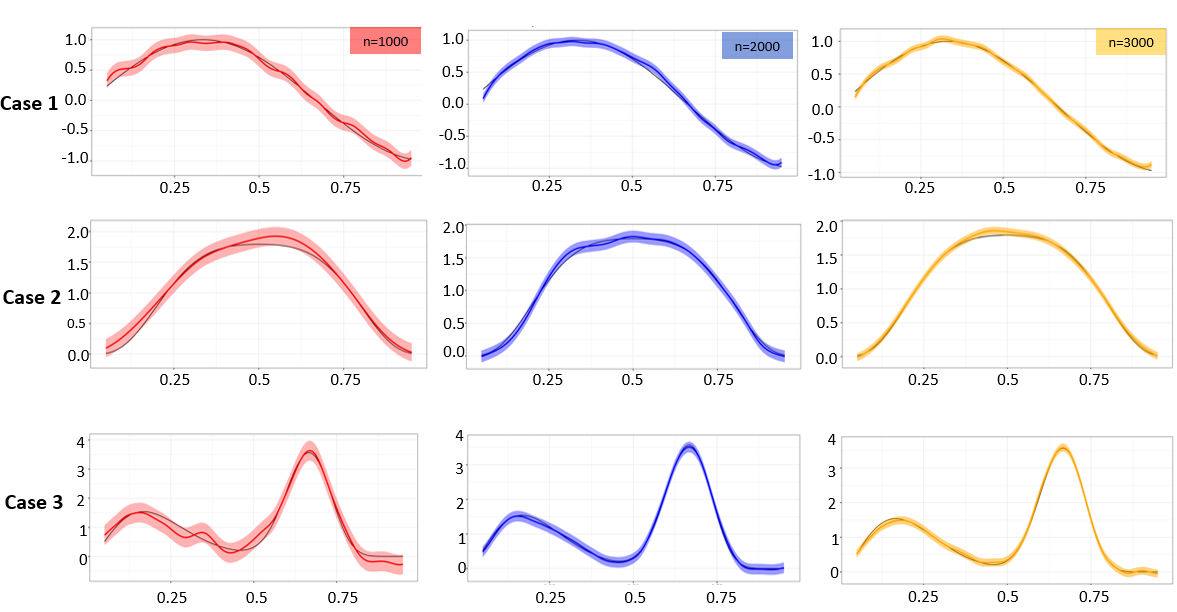}
\caption{Confidence band constructed using an online bootstrap approach with a non-constant step size. Data are generated in three cases with sample sizes of 1000 (red), 2000 (blue), and 3000 (yellow). The colored band represents the confidence band, the solid black curve is the true function curve, and the colored curve is the estimated function curve based on SGD.}
\label{fig:band}
\end{figure}

As shown in Figure \ref{fig:sim1}, the coverage of all methods approaches the predetermined level of $95\%$ as $t$ increases. The offline Bayesian method exhibits the lowest coverage of all. While it has the longest average confidence interval length in Cases 1-3, it also has the smallest variance in confidence interval lengths. The offline bootstrap-based methods demonstrate higher coverage and shorter average confidence interval lengths than the offline Bayesian method. The variance in confidence interval lengths for these bootstrap-based methods is larger, due to the bootstrap multiplier resampling procedure or the random step size used in our proposed online bootstrap procedures. As the sample size grows, the variance in the length of the confidence interval diminishes for all methods. Our online bootstrap procedure with a non-constant step size outperforms the others regarding both the average length and the variance of the confidence interval. It offers the shortest average confidence interval length and the smallest variance, compared to the Bayesian confidence interval, offline bootstrap methods, and the online bootstrap procedure with a constant step size. Moreover, the online bootstrap method with a non-constant step size achieves the predetermined coverage level of $95\%$ more quickly than the other methods. We only tested our methods (online BNN and online BQN) with an increased $t$ at $t=3500$ and $t=4000$ due to computational costs. As observed in Figure \ref{fig:sim1} (A1), (B1), and (C1), the coverage stabilizes at the predetermined coverage level of $95\%$.
We also use our proposed online bootstrap method, as outlined in Algorithm \ref{algorithm1}, to construct a confidence band of level of $95\%$ with a step size of $\gamma_i = i^{-\frac{1}{\alpha+1}}$ at $n=1000,2000,3000$. As seen in Figure \ref{fig:band}, the average width of the confidence band decreases as the sample size increases for Case 1-3, and all of them cover the true function curve represented by the solid black curve, indicating that the accuracy of our confidence band estimates improves with a larger sample size.

\begin{figure}
\centering
\begin{tabular}{cc}
{\footnotesize (A) Cumulative Computation Time }& {\footnotesize (B) Current Computation Time} \\
\includegraphics[width=0.45\columnwidth]{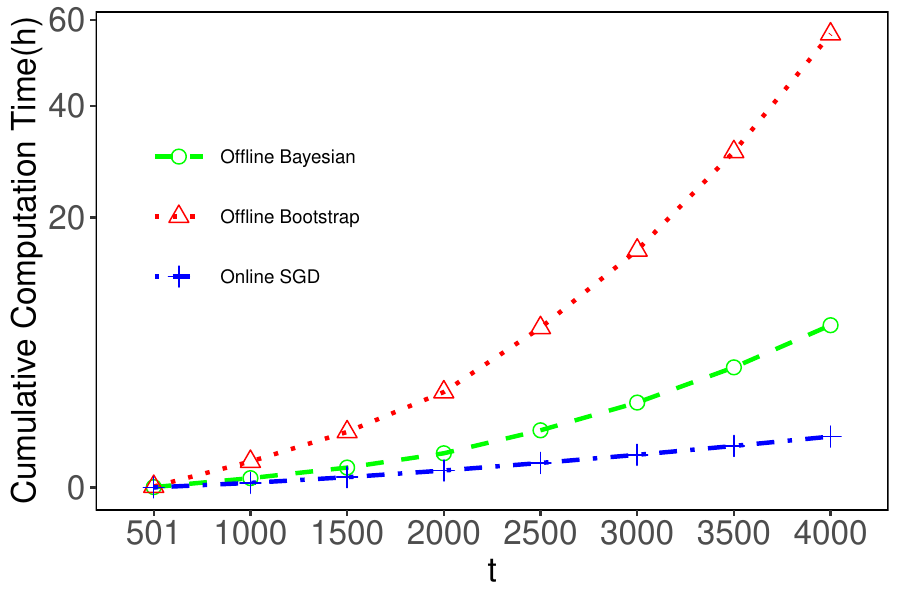}&
\includegraphics[width=0.45\columnwidth]{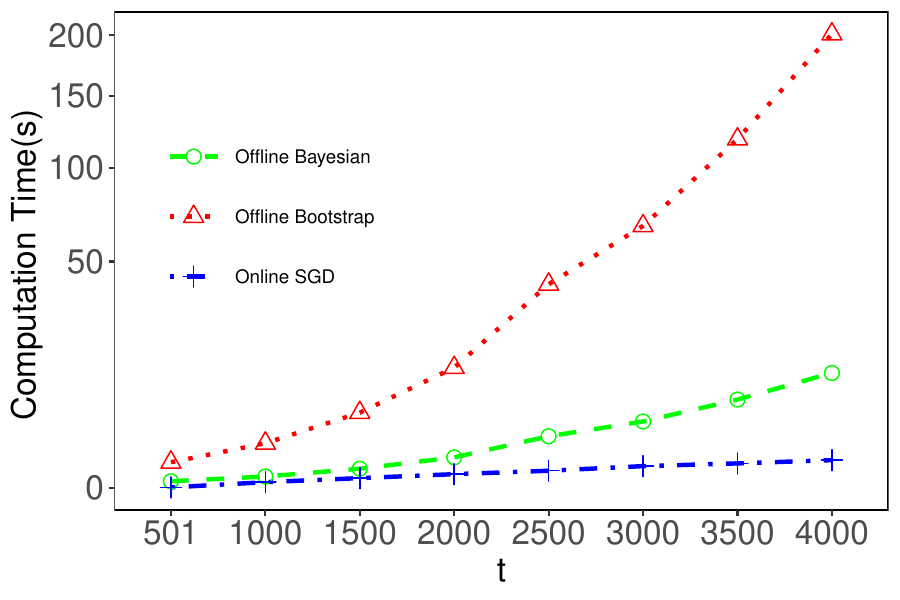}
\end{tabular}
\caption{ (A) Cumulative Computation time is recorded as $t$ increasing from $501$ to $4000$. (B) The computation time of constructing the confidence interval is recorded at different time points. The computation time is represented on the Y-axis with a scaled interval to differentiate between the blue and green curves.}
\label{fig:computing}
\end{figure}

Finally, we compared the computational time of various methods in constructing confidence intervals on a computer workstation equipped with a 32-core 3.50 GHz AMD Threadripper Pro 3975WX CPU and 64GB RAM. We recorded the computational times as data points $501, 1000, \dots, 4000$ arrived and calculated the cumulative computational times up to $t=501, 1000, \cdots, 4000$ for both offline and online algorithms. The normal and percentile Bootstrap methods displayed similar computational times, so we chose to report the computational time of the percentile bootstrap interval for both offline and online approaches. Despite leveraging parallel computing to accelerate the bootstrap procedures, the offline bootstrap algorithms still demanded significant computation time. This is attributed to the need to refit the model each time a new data point arrived, which substantially raises the computational expense. The computational complexity of offline methods for computing the estimate of $f$ at time $t$ is $\mathcal O(t^3)$, leading to a cumulative computational complexity of order $\mathcal O(t^4)$. Including the bootstrap cost, the total computational complexity at $t$ becomes $\mathcal O(Bt^3)$, leading to a cumulative computational complexity of order $\mathcal O(Bt^4)$. As shown in Figure \ref{fig:computing}, the cumulative computational time reaches approximately $60$ hours for the offline bootstrap method and around $8$ hours for the Bayesian bootstrap method. Conversely, the cumulative computational time for our proposed bootstrap method grows almost linearly with $t$, and requires less than $30$ minutes up to $t=4000$. At $t=4000$, offline bootstrap methods take about $200$ seconds, and the Bayesian confidence interval necessitates roughly $30$ seconds to construct the confidence interval. Our proposed online bootstrap method requires fewer than $3$ seconds, demonstrating its potential for time-sensitive applications such as medical diagnosis and treatment, financial trading, and traffic management, where real-time decision-making is essential as data continuously flows in.

\section{Proof Sketch of Main Results}\label{sec:proof_sketch}
In this section, we present sketched proofs for the expansion of the functional SGD estimator relative to the supremum norm (Theorem~\ref{le:bias_variance}) and the bootstrap consistency for global inference (Theorem~\ref{thm:global:main1}), while highlighting some important technical details and key steps.

\subsection{Proof sketch for estimator expansion under supremum norm metric}\label{subsec:sketch1}
Theorem \ref{le:bias_variance} establishes the supreum norm bound with high probability for the high-order expansion of the SGD estimator. This result is crucial for the inference framework, as we only need to focus on the distribution behavior of leading terms given the negligible remainders. 
In the sketched proof, we denote $\eta_n = \widehat{f}_{n}- f^\ast$. According to (\ref{eq:sgd:recursion_f0}), we have 
\begin{equation*} 
\eta_n = (I - \gamma_n K_{X_n}\otimes K_{X_n}) \eta_{n-1} + \gamma_n \epsilon_n K_{X_n}. 
\end{equation*}
We split the recursion of $\eta_n$ into two finer recursions: bias recursion of $\eta_n^{bias}$ and noise recursion of $\eta_n^{noise}$ such that $\eta_n = \eta_n^{bias} + \eta_n^{noise}$, where
\begin{align*}
    \mbox{bias recursion:} \qquad &\eta_n^{bias} =  (I - \gamma_n K_{X_n}\otimes K_{X_n}) \eta^{bias}_{n-1}  \quad \quad \textrm{with} \quad \eta_0^{bias}= -f^*;\\
    \mbox{noise recursion:} \qquad &
    \eta^{noise}_n = (I - \gamma_n K_{X_n}\otimes K_{X_n}) \eta^{noise}_{n-1} + \gamma_n \epsilon_n K_{X_n} \quad \quad \textrm{with} \quad \eta_0^{noise}= 0.
\end{align*}
To proceed, we further decompose the bias recursion into two parts: (1) the leading bias recursion $\eta_n^{bias,0}$; and (2) the remainder bias recursion $\eta_n^{bias}-\eta_n^{bias,0}$ as follows:  
\begin{align*}
 \eta_n^{bias,0}= &\, (I - \gamma_n \Sigma) \eta_{n-1}^{bias,0}  \quad \quad \textrm{with} \quad \eta_0^{bias}= -f^*; \\
 \eta_n^{bias} - \eta_n^{bias,0} = &\, (I- \gamma_n K_{X_n}\otimes K_{X_n}) (\eta_{n-1}^{bias} - \eta_{n-1}^{bias,0}) + \gamma_n (\Sigma -   K_{X_n}\otimes K_{X_n} )\eta_{n-1}^{bias,0}.
\end{align*}
It is worth noting that the leading bias recursion essentially replaces $K_{X_n}\otimes K_{X_n}$ by its expectation $\Sigma = \EE [K_{X_n}\otimes K_{X_n}]$.

To bound the residual term $\| \bar{\eta}_n^{bias} - \bar{\eta}_n^{bias,0}\|_\infty$ associated with the leading bias term of the averaged estimator, we introduce an augmented RKHS space (with $a\in[0,\, 1/2-1/(2\alpha)]$)
\begin{equation}\label{eqn:augmented}
    \mathbb{H}_a = \Big\{f= \sum_{\nu=1}^\infty f_\nu \phi_\nu \,\mid \,\sum_{\nu=1}^\infty f_\nu^2 \mu_\nu^{2a-1}< \infty\Big\}
\end{equation}
equipped with the kernel function $K^a(x,y)= \sum_{\nu=1}^\infty \phi_\nu(X)\phi_\nu(y)\mu_\nu^{1-2a}$. To verify $K^a(\cdot,\cdot)$ is the reproducing kernel of $\mathbb{H}_a$, we notice that
$$
\|K_x^a\|_a^2 = \|\sum_{\nu=1}^\infty \mu_\nu^{1-2a}\phi_\nu(x)\phi_\nu \|_a^2 = \sum_{\nu=1}^\infty (\phi_\nu(x)\mu_\nu^{1-2a})^2 \mu_\nu^{2a-1} = \sum_{\nu=1}^\infty \phi^2_\nu(x)\mu_\nu^{1-2a} < c^2_a,
$$
where $c_a$ is a constant. Moreover, $K_x^a(\cdot)$ also satisfies the reproducing property since
\begin{align*}
\langle K_x^a, f \rangle_a = &\langle \sum_{\nu=1}^\infty \mu_\nu^{1-2a}\phi_\nu(x)\phi_\nu, f \rangle_a =  \sum_{\nu=1}^\infty \phi_\nu(x)\mu_\nu^{1-2a} f_\nu \langle \phi_\nu, \phi_\nu \rangle_a = \sum_{\nu=1}^\infty f_\nu \phi_\nu(x) = f(x). 
\end{align*}
For any $f\in \cH \subset \mathbb{H}_a$, we can use the above reproducing property to bound the supremum norm of $f$ as
$\|f\|_{\infty} =\sup_{x\in[0,1]}|f(x)| = |\langle K_x^a, f \rangle_a| \leq \|f\|_a \cdot\|K_x^a\|_a^2< c_a \|f\|_a$.
Also note that 
for any $f\in \cH$, $ \|f\|^2_{\cH}= \sum_{\nu=1}^\infty f_\nu^2 \mu_\nu^{-1} \leq \sum_{\nu=1}^\infty f_\nu^2 \mu_\nu^{2a-1}= \|f\|^2_{a}$ for $a\geq 0$; therefore, we have the relationship $\|f\|_{\infty}\leq c_a \|f\|_{a} \leq c_k \|f\|_{\cH}$, meaning that $\|\cdot\|_a$ provides a tighter bound for the supremum norm compared with $\|\cdot\|_\cH$. In Section \ref{app:le:rem_bias:con} (Lemma \ref{app:le:bias_rem:con}), we use this augmented RKHS to show that the bias remainder term satisfies $\|\bar{\eta}_n^{bias}-\bar{\eta}_n^{bias,0}\|^2_{\infty} = o(\bar{\eta}_n^{bias,0})$ through computing the expectation $\EE\big[\|\bar{\eta}_n^{bias}-\bar{\eta}_n^{bias,0}\|^2_{a}\big]$ and applying the Markov inequality.

For the noise recursion of $\eta_n^{noise}$, we can similarly split it into the leading noise recursion term and residual noise recursion term as 
\begin{align*}
\eta_n^{noise,0} = &\, (I - \gamma_n\Sigma) \eta^{noise,0}_{n-1} + \gamma_n \epsilon_n K_{X_n} \quad \quad \textrm{with}\quad \eta_{0}^{noise, 0}=0;\\
\eta_n^{noise} - \eta_n^{noise,0} = & \, (I - \gamma_n K_{X_n}\otimes K_{X_n}) (\eta_{n-1}^{noise} - \eta_{n-1}^{noise,0})
+ \gamma_n (\Sigma - K_{X_n}\otimes K_{X_n} ) \eta_{n-1}^{noise, 0}.
\end{align*}
The leading noise recursion is described as a ``semi-stochastic'' recursion induced by $\eta_n^{noise}$ in \cite{Bach2016} since it keeps the randomness in the noise recursion $\eta_n^{noise}$ due to the noise $\{\epsilon_i\}_{i=1}^n$, but get ride of the randomness arising from $K_{X_n}\otimes K_{X_n}$, which is due to the random design $\{X_i\}_{i=1}^n$. 

For the residual noise recursion, directly bound $\|\bar{\eta}_n^{noise} - \bar{\eta}_n^{noise,0}\|_\infty$ is difficult. Instead, we follow~\cite{Bach2016} 
by further decomposing $\eta_n^{noise} - \eta_n^{noise,0}$ into a sequence of higher-order ``semi-stochastic'' recursions as follows. 
We first define a semi-stochastic recursion induced by $\eta_n^{noise} - \eta_n^{noise,0}$, denoted as $\eta_n^{noise,1}$: 
\begin{equation}
\eta_n^{noise,1} = (I-\gamma_n \Sigma)\eta_{n-1}^{noise,1} + \gamma_n (\Sigma -  K_{X_n}\otimes K_{X_n})\eta_{n-1}^{noise,0}.\label{eq:noise_1:recursion}
\end{equation}
Here, $\eta_n^{noise,1}$ replaces the random operator $K_{X_n}\otimes K_{X_n}$ with its expectation $\Sigma$ in the residual noise recursion for $\eta_n^{noise} - \eta_n^{noise,0}$, and can be viewed as a second-order term in the expansion of the noise recursion, or the leading remainder noise term. 
The rest noise remainder parts can be expressed as  
\begin{align*}
&\, \eta_n^{noise} - \eta_n^{noise,0} -\eta_n^{noise,1} 
= (I -\gamma_n  K_{X_n}\otimes K_{X_n}) (\eta_{n-1}^{noise}-\eta_{n-1}^{noise,0}) - (I-\gamma_n \Sigma)\eta_{n-1}^{noise,1}\\
&\qquad =  (I-\gamma_n K_{X_n}\otimes K_{X_n})(\eta_{n-1}^{noise}-\eta_{n-1}^{noise,0}-\eta_{n-1}^{noise,1}) + \gamma_n (\Sigma - K_{X_n}\otimes K_{X_n})\eta_{n-1}^{noise,1}.
\end{align*}
Then we can further define a semi-stochastic recursion induced by $\eta_n^{noise} - \eta_n^{noise,0} -\eta_n^{noise,1}$, and repeat this process.
If we define $\mathcal{E}_n^{r} = (\Sigma -  K_{X_n}\otimes K_{X_n})\,\eta_{n-1}^{noise,r-1}$ for $r\geq 1$, then we can expand $\eta_n^{noise}$ into $(r+2)$ terms as
$$
\eta_n^{noise} = \eta_n^{noise,0} + \eta_n^{noise,1} + \eta_n^{noise,2} +\cdots + \eta_n^{noise,r} + \textrm{Remainder},
$$
where $\eta_n^{noise,d} = (I-\gamma_n \Sigma)\eta_{n-1}^{noise,d} + \gamma_n \mathcal{E}_n^{d}$ for $1\leq d\leq r$,. The Remainder term $\eta_n^{noise}-\sum_{d=0}^r\eta_n^{noise,d}$ also has a recursive characterization: 
\begin{equation}\label{eq:remainder:recursion}
\eta_n^{noise} - \sum_{d=1}^r \eta_n^{noise,i} = (I -\gamma_n  K_{X_n}\otimes K_{X_n})(\eta_{n-1}^{noise} - \sum_{d=1}^r \eta_{n-1}^{noise,d}) + \gamma_n \mathcal{E}_n^{r+1}. 
\end{equation}

To establish the supreme norm bound of $\bar{\eta}_n^{noise} - \bar{\eta}_n^{noise,0}$, the idea is to show that $\bar{\eta}_n^{noise,r}$ decays as $r$ increases, that is, to prove
$$
\bar{\eta}_n^{noise} = \bar{\eta}_n^{noise,0} + \underbrace{\bar{\eta}_n^{noise,1}}_{=o(\bar{\eta}_n^{noise,0})} + \underbrace{\bar{\eta}_n^{noise,2}}_{=o(\bar{\eta}_n^{noise,1})} + \dots + \underbrace{\bar{\eta}_n^{noise,r}}_{=o(\bar{\eta}_n^{noise,r-1})}+ \underbrace{ \bar{\eta}_n^{noise} - \sum_{i=0}^r \bar{\eta}_n^{noise,i}}_{negligible}.
$$
Concretely, we consider the constant-step case for a simple presentation. By accumulating the effects of the iterations, we can further express $\eta_n^{noise,1}$ as 
\begin{align*} 
\eta_n^{noise,1} = &\gamma \sum_{i=1}^{n-1} (I-\gamma\Sigma)^{n-i-1}\big(\Sigma - K_{X_{i+1}}\otimes K_{X_{i+1}}\big)\eta_i^{noise,0}\\
= & \gamma^2 \sum_{i=1}^{n-1}\sum_{j=1}^i \epsilon_j (I-\gamma\Sigma)^{n-i-1}\big(\Sigma - K_{X_{i+1}}\otimes K_{X_{i+1}}\big)(I-\gamma\Sigma)^{i-j}K_{X_j}, 
\end{align*}
and accordingly, the averaged version is
$$
\bar{\eta}_n^{noise,1} = \frac{1}{n} \sum_{j=1}^{n-1}\, \epsilon_j \cdot \underbrace{\gamma^2 \,\Big[\sum_{i=j}^{n-1}(\sum_{\ell=i}^{n-1}(I-\gamma\Sigma)^{\ell-i})\big(\Sigma - K_{X_{i+1}}\otimes K_{X_{i+1}}\big)(I-\gamma\Sigma)^{i-j}K_{X_j}\Big]}_{g_j}. 
$$
This implies that conditioning  on the covariates $\{X_1,\dots, X_n\}$, the empirical process $\bar{\eta}_n^{noise,1}(\cdot) = \frac{1}{n}\sum_{j=1}^{n-1} \epsilon_j\cdot g_j (\cdot)$ over $[0,1]$ is a Gaussian process with (function) weights $\{g_j\}_{j=1}^n$. We can then prove a bound of $\|\bar{\eta}_n^{noise,1}\|_{\infty}$ by careful analyzing the random function $\sum_{j=1}^{n-1} g_j^2(\cdot)$; see Appendix \ref{app:le:rem_var:con} for further details. 
A complete proof of Theorem~\ref{le:bias_variance} under constant step size is included in \cite{liu2023supp}; see Figure~\ref{float_chart} for a float chart explaining the relationship among different components in its proof. The proof for the non-constant step size case is conceptually similar but is considerably more involved, requiring a much more refined analysis of the accumulated step size effect on the iterations of the recursions in \cite{liu2023supp}.

\begin{figure}
\centering
\begin{tikzpicture}[node distance=3cm, auto]
\def\arrowlength{1}
\node (start) {$\bar{\eta}_n= \bar{f}_n - f^*$};
  \node[below left of=start, node distance=4.5cm] (bias) {$\bar{\eta}_n^{bias}$};
  \node[below right of=start, node distance=4.5cm] (noise) {$\bar{\eta}_n^{noise}$};
   \node[below right of=bias] (bias1) {$\bar{\eta}_n^{bias,0}$};
    \node[below left of=bias] (bias2) {$\bar{\eta}_n^{bias}-\bar{\eta}_n^{bias,0}$};
     \node[below left of=noise,node distance=4cm] (noise1) {$\bar{\eta}_n^{noise,0}$};
       \node[below of=noise] (noise2) {$\bar{\eta}_n^{noise,1}$};
    \node[below right of=noise,node distance=4cm] (noise3) {$\bar{\eta}_n^{noise}-\bar{\eta}_n^{noise,0}-\bar{\eta}_n^{noise,1}$};
    \draw[->] (start) -- node[midway,above,sloped]  {bias} (bias);
  \draw[->] (start) -- node[midway,above,sloped]   {noise} (noise);
\draw[->] (bias) -- node [midway,above,sloped]  {\small Section \ref{append:proof:lemma:bias_variance:1}} node[midway,below,sloped] {lead} (bias1);
\draw[->] (bias) -- node[midway,above,sloped] {\small Lemma \ref{app:le:bias_rem:con}} node[midway,below,sloped] {rem}(bias2);
\draw[->] (noise) -- node[midway,above,sloped] {\small Section \ref{append:proof:lemma:bias_variance:1}} node[midway,below,sloped] {lead} (noise1);
\draw[->] (noise) -- node[midway,above,sloped] {\small Lemma \ref{app:le:noise_rem:con}} node[midway,below,sloped] {rem} (noise2);
\draw[->] (noise) -- node[midway,above,sloped] {\small Lemma \ref{app:le:noise_rem:con}} node[midway,below,sloped] {rem} (noise3);
\end{tikzpicture}
\caption{Float chart for proof of Theorem~\ref{le:bias_variance} under constant step size.} \label{float_chart}
\end{figure}
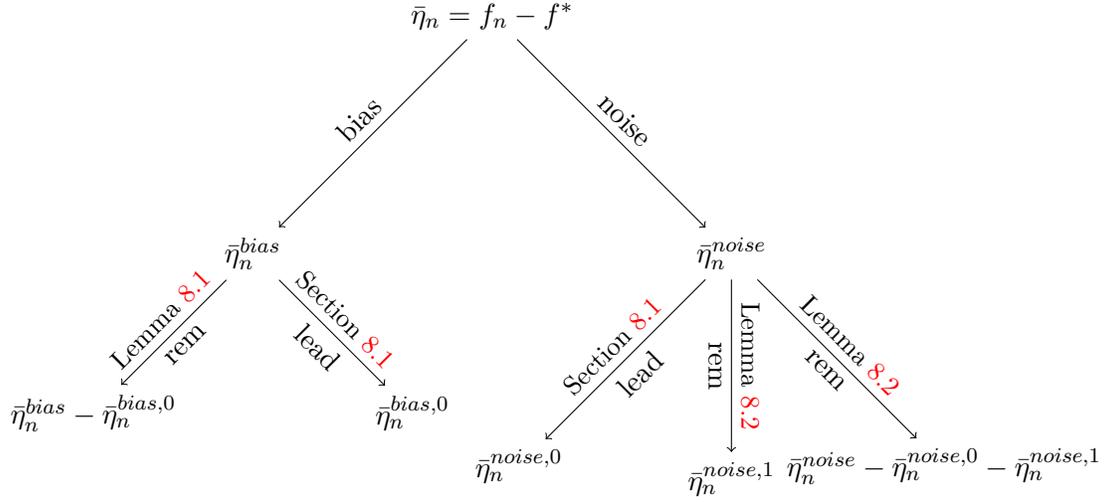

\subsection{Proof sketch for Bootstrap consistency of global inference}\label{sec:sketch:pf:GA}
Recall that $\mathcal{D}_n = \{X_i, Y_i\}_{i=1}^n$ represents the data. The goal is to bound the difference between the sampling distribution of $\sqrt{n(n\gamma)^{-1/\alpha}}\, \|\bar{\eta}_n^{noise,0}\|_\infty$ and the conditional distribution of $ \sqrt{n(n\gamma)^{-1/\alpha}}\, \|\bar{\eta}_n^{b,noise,0} - \bar{\eta}_n^{noise,0}\|_\infty$ given $\mathcal{D}_n$; see Section \ref{sec:bs-consistency} for detailed definitions of these quantities. We sketch the proof idea under the constant step size scheme. 

We will use the shorthand $\bar{\alpha}_n  = \sqrt{n(n\gamma)^{-1/\alpha}} \,\bar{\eta}_n^{noise,0}$ and $\bar{\alpha}_n^b = \sqrt{n (n\gamma)^{-1/\alpha}} \,\big(\bar{\eta}_n^{b,noise,0}- \bar{\eta}_n^{noise,0}\big)$. Recall that from equations~\eqref{eq:local_noise} and~\eqref{eq:boot:noise}, we have 
\begin{align*}
   \bar{\alpha}_n (\cdot)= \frac{1}{\sqrt{n(n\gamma)^{1/\alpha}}}\sum_{i=1}^n \epsilon_i \cdot \Omega_{n,i}(\cdot)    \quad\mbox{and}\quad
    \bar{\alpha}_n^b(\cdot) =   \frac{1}{\sqrt{n(n\gamma)^{1/\alpha}}}\sum_{i=1}^n (w_i-1)\cdot \epsilon_i \cdot \Omega_{n,i}(\cdot). 
\end{align*} 
From this display, we see that for any $t\in \cX$, $\bar{\alpha}_n (t)$ is a weighted sum of Gaussian random variables, with the weights being functions of covariates $\{X_i\}_{i=1}^n$; conditioning on $\mathcal{D}_n$, $\bar{\alpha}_n^b (t)$ is a weighted sum of sub-Gaussian random variables. In the proof, we also  require a sufficiently dense space discretization given by $0=t_1<t_2<\cdots<t_N=1$. This discretization forms an $\varepsilon$-covering for some $\varepsilon$ with respect to a specific distance metric that will be detailed later. 

To bound the difference between the distribution of $\|\bar{\alpha}_n\|_{\infty}$ and the conditional distribution of $\|\bar{\alpha}_n^b\|_{\infty}$ given $\mathcal{D}_n$, we introduce two intermediate processes: (1) 
$\bar{\alpha}_n^e (\cdot) =  \frac{1}{\sqrt{n(n\gamma)^{1/\alpha}}}\sum_{i=1}^n e_i \cdot \epsilon_i \cdot \Omega_{n,i}(\cdot)$ with  $e_i$ being i.i.d.~standard normal random variables for $i=1,\cdots, n$;  
(2) an $N$-dimensional multivariate normal random vector $\big(\bar{Z}_n (t_k) =  \frac{1}{\sqrt{n}}\sum_{i=1}^n Z_{i}(t_k),\, k=1,2,\ldots,N\big)$ (recall that $0=t_1<t_2<\cdots<t_N=1$ is the space discretization we defined earlier), where $\big\{\big(Z_1(t_1), Z_1(t_2),\ldots,Z_1(t_N)\big)\big\}_{i=1}^n$ are i.i.d.~(zero mean) normally distributed random vectors having the same covariance structure as $\big(\bar{\alpha}_n (t_1),\bar{\alpha}_n (t_2),\ldots,\bar{\alpha}_n (t_N))$; that is, $Z_{i}(t_k) \sim N \big(0, (n\gamma)^{-1/\alpha} \sum_{\nu=1}^\infty (1-(1-\gamma \mu_\nu)^{n-i})^2\phi_\nu^2(t_k)\big)$,
$\EE \big(Z_{i}(t_k)\cdot Z_{i}(t_\ell)\big) =(n\gamma)^{-1/\alpha} \sum_{\nu=1}^\infty (1-(1-\gamma \mu_\nu)^{n-i})^2\phi_\nu(t_k)\phi_\nu(t_\ell)$, and $\EE  \big(Z_{i}(t_k)\cdot Z_{j}(t_\ell)\big) = 0 $ for $(k,\ell)\in[N]^2$ and $(i,j)\in[n]^2$, $i\neq j$. These two intermediate processes are introduced so that the conditional distribution of $\max_{1\leq j\leq N} \bar{\alpha}_n^e(t_j)$ given $\mathcal{D}_n$ will be used to approximate the conditional distribution of $\|\bar{\alpha}_n^b\|_{\infty}$ given $\mathcal{D}_n$; while the distribution of $\max_{1\leq j\leq N} \bar{Z}_n (t_j)$ will be used to approximate the distribution of $\|\bar{\alpha}_n\|_{\infty}$. Since both the distribution of $\big(\bar{Z}_n (t_1),\bar{Z}_n (t_2),\ldots,\bar{Z}_n (t_N)\big)$ and  the conditional distribution of $\big(\bar{\alpha}_n^e (t_1),\bar{\alpha}_n^e (t_2),\ldots,\bar{\alpha}_n^e (t_N)\big)$ given $\mathcal{D}_n$ are centered multivariate normal distributions, we can use a Gaussian comparison inequality to bound the difference between them by bounding the difference between their covariances.

\tikzstyle{startstop} = [rectangle, rounded corners, 
minimum width=3cm, 
minimum height=1cm,
text centered, 
draw=black, 
fill=white!30]

\tikzstyle{arrow} = [thick,->,>=stealth]

\begin{figure}
  \centering
\begin{tikzpicture}[node distance=4cm]

\node (step1) [startstop] {$\|\bar{\alpha}_n\|_\infty$};
\node (step2) [startstop, right of=start, xshift=2cm] {$\underset{1\leq j\leq N}{\max}\bar{\alpha}_n(t_j)$};
\node (step3) [startstop, right of=step2, xshift=2cm] {$\underset{1\leq j\leq N}{\max}\bar{Z}_n(t_j)$};
\node (step4) [startstop, below of=step3] {$\underset{1\leq j\leq N}{\max}\bar{\alpha}^e_n(t_j) \mid \mathcal{D}_n$};
\node (step5) [startstop, below of=step2] {$\underset{1\leq j\leq N}{\max}\bar{\alpha}^b_n(t_j) \mid \mathcal{D}_n$};
\node (step6) [startstop, below of=step1] {$\|\bar{\alpha}^b_n\|_\infty \mid \mathcal{D}_n$};
\draw [arrow] (step1) -- (step2) node[midway, above] {\RNum{1}} node[midway, below] {};
\draw [arrow] (step2) -- (step3) node[midway, above] {\RNum{2}} node[midway, below] {Lemma \ref{app:le:GP:atoz:con}};
\draw [arrow] (step3) -- (step4) node[midway, left] {\RNum{3}} node[midway, right] {Lemma \ref{app:le:GP:ztoe:con}};
\draw [arrow] (step4) -- (step5) node[midway, above] {\RNum{4}} node[midway, below] {Lemma \ref{app:le:GP:etob:con}};
\draw [arrow] (step5) -- (step6) node[midway, above] {\RNum{5}} node[midway, below]{};
\draw [arrow] (step6) -- (step1) node[midway, left] {\RNum{6}} node[midway, right] {};
\end{tikzpicture}
  \caption{Flow chart of the bootstrap consistency proof.}
  \label{fig:flowchart:GA}
\end{figure}
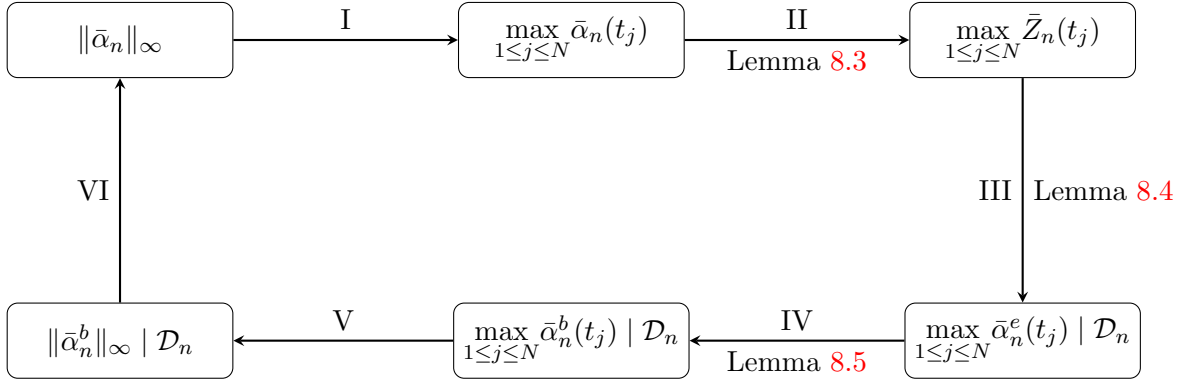

The actual proof is even more complicated, as we also need to control the discretization error. See Figure \ref{fig:flowchart:GA} for a flow chart that summarizes all the intermediate approximation steps and the corresponding lemmas in the appendix.
For Steps \RNum{1} and \RNum{5} in Figure \ref{fig:flowchart:GA}, we approximate the continuous supremum norms of $\bar{\alpha}_n$ and $\bar{\alpha}_n^b$ by the finite maximums of $\big(\bar{\alpha}_n(t_1),\ldots, \bar{\alpha}_n(t_N)\big)$ and $\big(\bar{\alpha}_n^b(t_1),\ldots, \bar{\alpha}_n^b(t_N)\big)$, respectively. Here, $N$ is chosen as the $\varepsilon$-covering number of the unit interval $[0,1]$ 
with respect to the metric defined by $e_P^2(t,s)= \EE \big[\big(\bar{\alpha}_n(t)-\bar{\alpha}_n(s)\big)^2\big]$ for $(t,s)\in[0,1]^2$; that is, there exist $t_1, \dots, t_N \in [0,1]$, such that for every $t\in [0,1]$, there exist $1\leq j \leq N$ with $e_P(t, t_j) < \varepsilon $. We refer a detailed proof in Step \RNum{1} (Step \RNum{5}) to Supplementary. Notice that $\bar{\alpha}_n$ is a weighted and non-identically distributed empirical process. In Step \RNum{2}, we further develop Gaussian approximation bounds to control the Kolmogorov distance between the sampling distributions of $\max_{1\leq j \leq N} \bar{\alpha}_n(t_j)$ and the distribution of $\max_{1\leq j \leq N} \bar{Z}_n(t_j)$; see the proof in Lemma \ref{app:le:GP:atoz:con}. In Step \RNum{4}, by noticing that conditional on $\mathcal{D}_n$, $\bar{\alpha}_n^b$ is a weighted and non-identically distributed sub-Gaussian process with randomness coming from the Bootstrap multiplier $\{w_i\}_{i=1}^n$, we adopt a similar argument as in Step \RNum{2} to bound the Kolmogorov distance between the distributions of $\max_{1\leq j \leq N} \bar{\alpha}_n^e(t_j)$ and $\max_{1\leq j \leq N} \bar{\alpha}_n^b(t_j)$ given $\mathcal D_n$. 
 
\section{Discussion}\label{sec:dis}
Quantifying uncertainty (UQ) in large-scale streaming data is a central challenge in statistical inference. We are developing multiplier bootstrap-based inferential frameworks for UQ in online non-parametric least squares regression. We propose using perturbed stochastic functional gradients to generate a sequence of bootstrapped functional SGD estimators for constructing point-wise confidence intervals (local inference) and simultaneous confidence bands (global inference) for function parameters in RKHS. Theoretically, we establish a framework to derive the non-asymptotic law of the infinite-dimensional SGD estimator and demonstrate the consistency of the multiplier bootstrap method.

This work assumes that random errors in non-parametric regression follow a Gaussian distribution. However, in many real-world applications, heavy-tailed distributions are more common and suitable for capturing outlier behaviors. One future research direction is to expand the current methods to address heavy-tailed errors, thereby offering a more robust approach to online non-parametric inference. Another direction to explore is the generalization of the multiplier bootstrap weights to independent sub-exponential random variables and even exchangeable weights. Finally, a promising direction is the consideration of online non-parametric inference for dependent data. Such an extension is necessary to address problems like multi-arm bandit and reinforcement learning, where data dependencies are frequent and real-time updates are essential. Adapting our methods to these problems could provide deeper insights into the interplay between statistical inference and online decision-making.

\section{Some Key Proofs} \label{sec:key_proof}
\subsection{Proof of leading terms in Theorem~\ref{le:bias_variance} in constant step size case }\label{append:proof:lemma:bias_variance:1}

Recall in Section \ref{subsec:sketch1}, we split the recursion of $\eta_n= \widehat{f}_n - f^* $ into the bias recursion and noise recursion. That is, $\eta_n = \eta_n^{bias} + \eta_n^{noise}$. Here $\eta_n^{bias}$ can be further decomposed as its leading bias term $\eta_n^{bias,0}$ and remainder $\eta_n^{bias}-\eta_n^{bias,0}$ satisfying the recursion 
\begin{align}
 \eta_n^{bias,0}= & (I - \gamma_n \Sigma) \eta_{n-1}^{bias,0} \quad \quad \textrm{with} \quad \eta_0^{bias,0}=f^\ast \label{eq:bias:lead}\\
 \eta_n^{bias} - \eta_n^{bias,0} = & (I- \gamma_n K_{X_n}\otimes K_{X_n}) (\eta_{n-1}^{bias} - \eta_{n-1}^{bias,0}) + \gamma_n (\Sigma -   K_{X_n}\otimes K_{X_n} )\eta_{n-1}^{bias,0}.  \label{eq:bias:rem}
\end{align}

We further decompose $\eta_n^{noise}$ to its main recursion terms and residual recursion terms as 
\begin{align}
\eta_n^{noise,0} = & (I - \gamma_n\Sigma) \eta^{noise,0}_{n-1} + \gamma_n \epsilon_n K_{X_n} \quad \quad \textrm{with} \quad \eta_{0}^{noise, 0}=0 \\
\eta_n^{noise} - \eta_n^{noise,0} = &  (I - \gamma_n K_{X_n}\otimes K_{X_n}) (\eta_{n-1}^{noise} - \eta_{n-1}^{noise,0})
+ \gamma_n (\Sigma - K_{X_n}\otimes K_{X_n} ) \eta_{n-1}^{noise, 0} \label{eq:noise:rem}
\end{align}

We focus on the averaged version $\bar{\eta}_n = \bar{f}_n - f^\ast $ with 
\begin{align*}
\bar{\eta}_n 
= & \bar{\eta}_n^{bias,0} + \bar{\eta}_n^{noise,0} + Rem_{noise} + Rem_{bias}, 
\end{align*}
where $Rem_{noise} = \bar{\eta}_n^{noise}-\bar{\eta}_n^{noise,0}$,  $Rem_{bias} = \bar{\eta}_n^{bias} - \bar{\eta}_n^{bias,0}$.

Theorem~\ref{le:bias_variance} for the constant step case includes three results as follows: 
\begin{align}
& \sup_{z_0\in\mathcal{X}}|\bar{\eta}_n^{bias,0}(z_0)|\lesssim  \frac{1}{\sqrt{n\gamma}} \label{eq:le:bias_variance:1}\\
&\sup_{z_0\in\mathcal{X}} \Var\big(\bar{\eta}_n^{noise,0}(z_0)\big)\lesssim  \frac{(n\gamma)^{1/\alpha}}{n}\label{eq:le:bias_variance:2}\\
& \PP \Big(\|\bar{\eta}_n - \bar{\eta}_n^{bias,0}-\bar{\eta}_n^{noise,0}\|^2_{\infty} \geq \gamma^{1/2} (n\gamma)^{-1}+ \gamma^{1/2} (n\gamma)^{1/\alpha}n^{-1}\log n\Big) \leq 1/n + \gamma^{1/2} \label{eq:le:bias_variance:3}. 
\end{align}

In this session, we will bound the sup-norm of the leading bias term $\bar{\eta}_n^{bias,0}$ and leading variance term $\bar{\eta}_n^{noise,0}$. To complete the proof of (\ref{eq:le:bias_variance:3}), we will bound  $\|Rem_{bias}\|_\infty $ in Section \ref{app:le:rem_bias:con} and $\|Rem_{noise}\|_\infty$ in Section \ref{app:le:rem_var:con}. 

We first provide a clear expression for $\eta_n^{bias,0}$ and $\eta_n^{noise,0}$. 

Denote 
\begin{align*}
D(k, n, \gamma_i) = & \prod_{i=k}^n (I - \gamma_i \Sigma) \quad \textrm{and} \quad D(k, n, \gamma) =  \prod_{i=k}^n (I - \gamma \Sigma)=  (I - \gamma \Sigma)^{n-k+1}\\
M(k, n, \gamma_i ) = & \prod_{i=k}^n (I - \gamma_i K_{X_i} \otimes K_{X_i}) \quad \textrm{and} \quad
M(k, n, \gamma) =  \prod_{i=k}^n (I - \gamma K_{X_i} \otimes K_{X_i}),
\end{align*}
with $D(n+1, n, \gamma_i)=D(n+1, n, \gamma)=1$. We have   \begin{equation}\label{eq:bias:lead:exp}
      \eta_n^{bias,0} = D(1, n, \gamma_i)f^\ast  \quad \textrm{and} \quad 
 \bar{\eta}_n^{bias,0}= \frac{1}{n}\sum_{k=1}^n D(1, k, \gamma_i)f^\ast; 
  \end{equation}
\begin{equation}\label{eq:noise:lead:exp}
\eta_n^{noise,0} = \sum_{i=1}^n D(i+1,n, \gamma_i) \gamma_i \epsilon_i K_{X_i}
\quad \textrm{and} \quad
\bar{\eta}_n^{noise,0} =  \frac{1}{n} \sum_{i=1} ^n \big(\sum_{j=i}^n D(i+1,j, \gamma_i) \big) \gamma_i  \epsilon_i K_{X_i}.
\end{equation}

\underline{\bf{Bound the leading bias term (\ref{eq:le:bias_variance:1})}} For the case of constant step size, based on (\ref{eq:bias:lead:exp}), we have 
\begin{align*}
    \bar{\eta}_n^{bias,0}(z_0)= & \frac{1}{n}\sum_{k=1}^n (I - \gamma \Sigma)^{n-k+1} f^\ast(z_0). 
\end{align*}
Note that any $f\in \mathbb{H}$ can be represented as $f=\sum_{\nu=1}^\infty \langle f, \phi_\nu \rangle_{L_2}\phi_\nu$, where $\{\phi_\nu\}_{\nu=1}^\infty$ satisfies $\|\phi_\nu\|_{L_2}^2 = 1 = \EE (\phi_\nu^2(x))$, $\langle \phi_\nu, \phi_\nu \rangle_{\mathbb{H}} = \mu_\nu^{-1}$, and $\Sigma \phi_\nu = \mu_\nu \phi_\nu$, $\Sigma^{-1}\phi_\nu = \mu_\nu^{-1}\phi_\nu$. 
Then for any $z_0\in \mathcal{X}$, $f^\ast(z_0) = \sum_{\nu=1}^\infty \langle f^\ast, \phi_\nu\rangle_{L_2} \phi_\nu(z_0)$. By the assumption that $f\in \mathbb{H}$ satisfies $\sum_{\nu=1}^\infty  \langle f^\ast, \phi_\nu\rangle_{L_2} \mu_\nu^{-1/2} < \infty$, we have 
\begin{align*}
\bar{\eta}_n^{bias,0}(z_0) = & \frac{1}{\sqrt{\gamma}n}\sum_{k=1}^n \sum_{\nu=1}^\infty  \langle f^\ast, \phi_\nu\rangle_{L_2} \mu_\nu^{-1/2} (1-\gamma \mu_\nu)^k (\gamma\mu_\nu)^{1/2}\phi_\nu(z_0)\\
\leq & c_\phi \frac{1}{\gamma^{1/2}n} \sum_{\nu=1}^\infty \langle f^\ast, \phi_\nu\rangle_{L_2} \mu_\nu^{-1/2} \Big(\sup_{0\leq x \leq 1} 
\big(\sum_{k=1}^n  (1-x)^k x^{1/2}\big) \Big) \\
\leq & c_\phi \frac{1}{\sqrt{n \gamma}} \sum_{\nu=1}^\infty \langle f^\ast, \phi_\nu\rangle_{L_2} \mu_\nu^{-1/2} \lesssim  \frac{1}{\sqrt{n \gamma}}, 
\end{align*}
where the inequality holds based on the bound that  $\sup_{0\leq x \leq 1} 
\big(\sum_{k=1}^n  (1-x)^k x^{1/2}\big)\leq \sqrt{n}$. 

{\underline{\bf{Bound the leading noise term (\ref{eq:le:bias_variance:2})}}}
We first deduce the explicit expression of $\bar{\eta}_n^{noise,0}(z_0)$ and its variance. 
Based on (\ref{eq:noise:lead:exp}), for constant step case, we have for any $z_0\in \cX$,
$$
\bar{\eta}_n^{noise,0}(z_0)  = \frac{1}{n}\sum_{k=1}^n \Sigma^{-1}\big(I - (I-\gamma\Sigma)^{n+1-k} \big) K(X_k,z_0)\epsilon_k. 
$$

Note that for any $x,z$, $K(x,z) = \sum_{\nu=1}^\infty \mu_\nu \phi_\nu(x)\phi_\nu(z)$. Then 
\begin{align*}
\Sigma^{-1}\big(I - (I-\gamma\Sigma)^{n+1-k} \big) K(X_k,z_0) = & \Sigma^{-1}\big(I - (I-\gamma\Sigma)^{n+1-k} \big) \big(\sum_{\nu=1}^\infty \mu_\nu\phi_\nu(X_k) \phi_\nu(z_0)\big)\\
= & \sum_{\nu=1}^\infty (1-(1-\gamma \mu_\nu)^{n+1-k})\phi_\nu(X_k)\phi_\nu(z_0).
\end{align*} 
Therefore, $\bar{\eta}_n^{noise, 0}(z_0) = \frac{1}{n} \sum_{k=1}^n \sum_{\nu=1}^\infty (1-(1-\gamma \mu_\nu)^{n+1-k})\phi_\nu(X_k)\phi_\nu(z_0)\epsilon_k$ with $\EE (\bar{\eta}_n^{noise,0}(z_0)) = 0$, and 
$$
\Var\big(\bar{\eta}_n^{noise,0}(z_0)\big) = \frac{\sigma^2}{n^2} \sum_{\nu=1}^\infty \phi_\nu^2(z_0)\sum_{k=1}^n[1-(1-\gamma \mu_\nu)^{n+1-k} ]^2.
$$
Note that 
$$
\sum_{k=1}^n \sum_{\nu=1}^\infty \big(1-(1-\gamma \mu_\nu)^k\big)^2 \asymp \sum_{k=1}^n \sum_{\nu=1}^\infty \min\bigl\{1, (k\gamma \mu_\nu)^2\bigr\}.
$$
On the other hand, $\sum_{\nu=1}^\infty \min\{1,(k\gamma \mu_\nu)^2\} = (k\gamma)^{1/\alpha} + \sum_{\nu=(k\gamma)^{1/\alpha}+1}^\infty (k\gamma \mu_\nu)^2$. Since 
$$
\sum_{\nu=(k\gamma)^{1/\alpha}+1}^\infty (k\gamma \mu_\nu)^2 \leq \sum_{\nu=(k\gamma)^{1/\alpha}+1}^\infty k\gamma\mu_\nu = k \gamma \sum_{\nu=(k\gamma)^{1/\alpha}+1}^\infty \nu^{-\alpha} \leq k \gamma \int_{(k\gamma)^{1/\alpha}}^\infty x^{-\alpha} dx = (k\gamma)^{1/\alpha},
$$
 we have 
 $$
 \sum_{k=1}^n \sum_{\nu=1}^\infty \big(1-(1-\gamma \mu_\nu)^k\big)^2  \asymp \sum_{k=1}^n (k\gamma)^{1/\alpha} \asymp \gamma^{1/\alpha} n^{(\alpha+1)/\alpha} = (n\gamma)^{1/\alpha} n.  
 $$
 Accordingly, $\Var(\bar{\eta}_n^{noise,0}(z_0))\lesssim  \frac{(n\gamma)^{1/\alpha}}{n}$. 

Meanwhile, $\sum_{\nu=1}^\infty \min\{1,(k\gamma \mu_\nu)^2\} \geq  (k\gamma)^{1/\alpha}$ leads to the result that
$
\sum_{k=1}^n \sum_{\nu=1}^\infty \big(1-(1-\gamma \mu_\nu)^k\big)^2 \geq n(n\gamma)^{1/\alpha}, 
$ thus $\Var(\bar{\eta}_n^{noise,0}(z_0))\geq  \frac{(n\gamma)^{1/\alpha}}{n}$. Therefore, $\Var(\bar{\eta}_n^{noise,0}(z_0)) \asymp \frac{(n\gamma)^{1/\alpha}}{n}$. 

{\underline{\bf{Bound the remaining  term (\ref{eq:le:bias_variance:3})}}} 
Recall  
\begin{align*}
\bar{f}_n - f^\ast = &  \bar{\eta}_n 
=  \bar{\eta}_n^{bias,0} +  \bar{\eta}_n^{noise,0}
+ \big(\bar{\eta}_n^{bias} - \bar{\eta}_n^{bias,0}\big) + \big(\bar{\eta}_n^{noise} - \bar{\eta}_n^{noise,0}\big). 
\end{align*}
To prove (\ref{eq:le:bias_variance:3}) in Theorem \ref{le:bias_variance}, we bound $\|\bar{\eta}_n^{bias} - \bar{\eta}_n^{bias,0}\|_\infty$ and $\|\bar{\eta}_n^{noise} - \bar{\eta}_n^{noise,0}\|_\infty$ separately in Section \ref{app:le:rem_bias:con} and Section \ref{app:le:rem_var:con}.

\subsection{Bound the bias remainder in constant step case}\label{app:le:rem_bias:con}

Recall in (\ref{eq:bias:rem}),  the bias remainder recursion follows 
\begin{align*}
\eta_n^{bias} - \eta_n^{bias,0} = & (I -\gamma  K_{X_n}\otimes K_{X_n})(\eta_{n-1}^{bias} - \eta_{n-1}^{bias,0}) + \gamma (\Sigma - K_{X_n}\otimes K_{X_n})\eta_{n-1}^{bias,0}. 
\end{align*} 
Our goal is to bound $\|\bar{\eta}_n^{bias} - \bar{\eta}_n^{bias,0}\|_{\infty}$.  
Let $\beta_{n} = \eta_n^{bias} - \eta_n^{bias,0}$ with $\beta_{0} = 0$, we have 
$$
\beta_{n} = (I -\gamma K_{X_n}\otimes K_{X_n})\beta_{n-1} + \gamma (\Sigma -K_{X_n}\otimes K_{X_n}) \eta_{n-1}^{bias,0}
$$
with $\eta_n^{bias,0} = (I-\gamma \Sigma)\eta_{n-1}^{bias,0}$ and $\eta_0^{bias,0}=f^\ast$. We first express $\beta_n$ in an explicit form as follows. 

Let $S_{n} = I-\gamma  K_{X_n}\otimes K_{X_n}$, $T_n=\Sigma- K_{X_n}\otimes K_{X_n}$ and $T= I - \gamma \Sigma$, we have 
$
\beta_{n} = S_n \beta_{n-1} + \gamma T_n \eta_{n-1}^{bias,0}.
$
We can further represent $\beta_{n}$ as 
$$
\beta_{n} = \gamma (T_n \eta_{n-1}^{bias,0} + S_n T_{n-1} \eta_{n-2}^{bias,0}+ \cdots + S_n S_{n-1}\dots S_{2}T_{1} \eta_0^{bias,0});
$$
on the other hand, $\eta_{i}^{bias,0} = (I-\gamma \Sigma)^i f^\ast $. Therefore, for any $1\leq i\leq n$, we have 
\begin{equation}\label{eq:beta}
\beta_{i} = \gamma (T_{i} T^{i-1} + S_{i} T_{i-1}T^{i-2}+ \cdots + S_{i} S_{i-1}\cdots S_{2} T_1)\cdot f^\ast \equiv \gamma U_{i}.
\end{equation}

Note that $\|\bar{\beta}_{n}\|_{\infty}\leq \|\Sigma^a \bar{\beta}_{n}\|_\mathbb{H}$. In the following lemma \ref{app:le:bias_rem:con}, we bound $\|\bar{\beta}_{n}\|_{\infty}$ through $\EE\langle \bar{\beta}_{n}, \Sigma^{2a}\bar{\beta}_{n}\rangle$, and show that $\|\bar{\eta}_n^{bias} - \bar{\eta}_n^{bias,0}\|^2_{\infty} = o(\bar{\eta}_n^{bias,0})$ with high probability. 

\begin{Lemma}\label{app:le:bias_rem:con}
Suppose the step size $ \gamma(n) = \gamma$ with $0< \gamma < \mu_1^{-1}$. Then
$$\PP \Big(\|\bar{\eta}_n^{bias} - \bar{\eta}_n^{bias,0}\|^2_{\infty} \geq \gamma^{1/2} (n\gamma)^{-1}\Big) \leq \gamma^{1/2}. 
$$
\end{Lemma}

\begin{proof}

To simplify the notation, we set $\langle \cdot, \cdot \rangle$ as $\langle \cdot, \cdot \rangle_\mathbb{H}$. For $\bar{\beta}_{n}= \eta_n^{bias} - \eta_n^{bias,0}$, by (\ref{eq:beta}), we have 
\begin{equation}
\EE \langle \bar{\beta}_{n}, \Sigma^{2a}\bar{\beta}_{n}\rangle = \EE \langle \frac{1}{n}\sum_{i=1}^n\gamma U_{i}, \Sigma^{2a}\frac{1}{n}\sum_{i=1}^n \gamma U_{i}\rangle 
=  \frac{\gamma^2}{n^2} \sum_{i=1}^n \EE \langle U_{i}, \Sigma^{2a}U_{i} \rangle  + \frac{2\gamma^2}{n^2}\sum_{i<j} \EE \langle U_{i}, \Sigma^{2a}U_{j} \rangle.
\end{equation}
That is, we split $\EE \langle \bar{\beta}_{n}, \Sigma^{2a}\bar{\beta}_{n}\rangle$ into two parts, and will bound each part separately.

We first bound $ \frac{\gamma^2}{n^2} \sum_{i=1}^n \EE \langle U_{i}, \Sigma^{2a}U_{i} \rangle $. 
Denote $H_{i\ell} = S_iS_{i-1}\cdots S_{\ell+1}T_\ell T^{\ell-1}f^\ast$ with $H_{ii}= T_i T^{i-1}f^\ast$, then $U_{i} = H_{ii}+ H_{i(i-1)}+ \cdots + H_{i1}$. 
\begin{align*}
\EE \langle U_{i}, \Sigma^{2a} U_{i} \rangle = & \EE \langle H_{ii}+ H_{i(i-1)}+ \cdots + H_{i1} , \Sigma^{2a}(H_{ii}+ H_{i(i-1)}+ \cdots + H_{i1}) \rangle\\
= & \sum_{j,k=1}^i \EE \langle H_{ij}, \Sigma^{2a}H_{ik}\rangle
= \sum_{j=1}^i \EE \langle H_{ij}, \Sigma^{2a} H_{ij}\rangle + \sum_{j\neq k} \EE \langle H_{ij}, \Sigma^{2a}H_{ik}\rangle.
\end{align*}

If $j\neq k$, suppose $i\geq j >k \geq 1$, then 
\begin{align*}
&\EE \langle H_{ij}, \Sigma^{2a} H_{ik} \rangle 
=  \EE \langle S_i, S_{i-1}\cdots S_{j+1}T_j T^{j-1}f^\ast, \Sigma^{2a} S_i S_{i-1}\cdots S_{k+1}T_kT^{k-1}f^\ast\rangle \\
= &\EE \Big[ \EE\big[\langle S_iS_{i-1}\cdots S_{j+1}T_jT^{j-1}f^\ast, \Sigma^{2a}S_iS_{i-1}\cdots S_{k+1}T_kT^{k-1}f^\ast | X_i,\dots, X_j, w_i, \dots, w_j\big]\Big]\\
= & \EE \big(\langle S_iS_{i-1}\cdots S_{j+1}T_jT^{j-1}f^\ast, \Sigma^{2a}S_iS_{i-1}\dots S_j \EE (S_{j-1}\cdots S_{k+1}T_k)T^{k-1}f^\ast\big)
= 0, 
\end{align*}
where the last step is due to $\EE(S_{j-1}\cdots S_{k+1}T_k)= \EE S_{j-1}\cdots \EE S_{k+1}\EE T_k = 0$ with the fact that $\EE T_k =0$. 
Therefore, we have $\EE \langle U_{i}, \Sigma^{2a} U_{i} \rangle = \sum_{j=1}^i \EE \langle H_{ij}, \Sigma^{2a} H_{ij}\rangle $. Furthermore, 
\begin{align*}
& \EE \langle H_{ij}, \Sigma^{2a}H_{ij} \rangle 
=  \EE \langle S_iS_{i-1}\cdots S_{j+1}T_jT^{j-1}f^\ast, \Sigma^{2a}S_iS_{i-1}\cdots S_{j+1}T_jT^{j-1}f^\ast \rangle \\
= & \langle f^\ast, \EE(T^{j-1}T_jS_{j+1}\cdots S_i\Sigma^{2a}S_iS_{i-1}\cdots S_{j+1}T_jT^{j-1})f^\ast\rangle
=  \langle f^\ast, \Delta f^\ast\rangle . 
\end{align*}
Note that $\Delta = 
 \EE \big( T^{j-1}T_jS_{j+1}\cdots  \EE(S_i\Sigma^{2a}S_i)S_{i-1}\cdots S_{j+1}T_jT^{j-1} \big)$, with 
\begin{align}
\EE (S_i \Sigma^{2a} S_i)  = & \EE \big((I-\gamma  K_{X_i}\otimes K_{X_i})\Sigma^{2a} (I-\gamma K_{X_i}\otimes K_{X_i})\big)\nonumber\\
 = &  \Sigma^{2a} - \gamma (\Sigma \cdot \Sigma^{2a} + \Sigma^{2a}\cdot \Sigma - 2\gamma S\Sigma^{2a})
 = \Sigma^{2a}-\gamma G\Sigma^{2a}, \label{eq:s_sigma_s}
\end{align}
where $G\Sigma^{2a}= \Sigma \cdot \Sigma^{2a} + \Sigma^{2a}\cdot\Sigma - 2 \gamma S\Sigma^{2a}$ with $S\Sigma^{2a}=\EE \big((K_{x}\otimes K_{x})\Sigma^{2a}(K_{x}\otimes K_{x})\big)$. 

To be abstract, for any $A$, 
$
\EE S_i A S_i =A - \gamma (\Sigma A + A \Sigma - 2\gamma SA) = A - \gamma GA = (I-\gamma G)A, 
$
where $GA= \Sigma A + A \Sigma - 2\gamma SA$. Then $\Delta$ can be written as 
\begin{align*}
\Delta = & \EE \big(T^{j-1} T_j S_{j+1}\cdots S_{i-1} (I-\gamma G) \Sigma S_{i-1} \cdots S_{j+1}T_j\big) T^{j-1}
=  \EE \big(T^{j-1}T_j(I-\gamma G)^{i-j}\Sigma^{2a}T_j T^{j-1}\big).
\end{align*}
Furthermore, for any $A$, 
\begin{align}
\EE T_j A T_j = & \EE ( K_{X_j}\otimes K_{X_j} - \Sigma)A ( K_{X_j}\otimes K_{X_j} - \Sigma)
= \EE  (K_{X_j}\otimes K_{X_j})A(K_{X_j}\otimes K_{X_j}) - \Sigma A \Sigma \label{eq:T1}\\
\leq & 2 \EE (K_{X_j}\otimes K_{X_j})A(K_{X_j}\otimes K_{X_j}) 
= 2 SA.\nonumber
\end{align}
Therefore, $\Delta \prec 2 T^{j-1}S(I-\gamma G)^{i-j}\Sigma^{2a}T^{j-1}$, and in (\ref{eq:T1}), we have 
$
\EE \langle H_{ij}, \Sigma^{2a} H_{ij} \rangle \leq 2 \langle f^\ast, T^{j-1}S(I-\gamma G)^{i-j}\Sigma^{2a}T^{j-1}f^\ast\rangle.
$
Then we have 
\begin{align*}
\sum_{i=1}^n \EE \langle U_{i}, \Sigma^{2a}U_{i} \rangle  = \sum_{i=1}^n \sum_{j=1}^i \EE \langle H_{ij}, \Sigma^{2a} H_{ij}\rangle
\leq & 2 \langle f^\ast, \sum_{i=1}^n \sum_{j=1}^i T^{j-i}S(I-\gamma G)^{i-j}\Sigma^{2a}T^{j-1}f^\ast \rangle. 
\end{align*}
Denote $P=\sum_{i=1}^n \sum_{j=1}^i T^{j-i}S(I-\gamma G)^{i-j}\Sigma^{2a}T^{j-1}$, then 
\begin{align*}
P = & \sum_{i=1}^n \sum_{j=1}^i T^{j-1}S(I-\gamma G)^{i-j} \Sigma^{2a} T^{j-1} 
=  \sum_{j=1}^n T^{j-1}S \sum_{i=j}^n (I-\gamma G)^{i-j}\Sigma^{2a} T^{j-1} 
\leq  n\sum_{j=1}^n T^{j-1}S\Sigma^{2a}T^{j-1}.
\end{align*}
Recall $S\Sigma^{2a} = \EE \big((K_X\otimes K_X)\Sigma^{2a}(K_X\otimes K_X)\big)$, we can bound $S\Sigma^{2a}\leq c_k \Sigma$ as follows. 
\begin{align*}
\langle (S\Sigma^{2a})f, f \rangle = & \langle \EE \big((K_X\otimes K_X)\Sigma^{2a}(K_X\otimes K_X)\big)f, f \rangle
=  \langle \EE f(X)\Sigma^{2a} K_X(X)K_X, f \rangle \\
= & \EE f^2(X) (\Sigma^{2a}K_X)(X) \leq c_k \EE f^2(X) = c_k \langle \Sigma f, f\rangle,
\end{align*}
where the last inequality is due to the fact that 
$$
\Sigma^{2a}K_X(X) = \sum_{\nu=1}^\infty \phi_\nu(X)\phi_\nu(X)\mu_\nu \mu_\nu^{2a} = \sum_{\mu=1}^\infty \phi_{\nu}(x)\phi_{\nu}(x)\mu_\nu^{2a+1} \leq \infty. 
$$
Accordingly, $P \leq n c_k \sum_{j=1}^n T^{2(j-1)}\Sigma \leq n c_k (I-T^2)^{-1}\Sigma \leq n c_k \gamma^{-1}I$; and  
\begin{equation}
\frac{\gamma^2}{n^2}\sum_{i=1}^n \EE \langle U_{i}, \Sigma^{2a}U_{i} \rangle = \frac{\gamma^2}{n^2}\langle f^\ast, P f^\ast \rangle 
\lesssim  \frac{\gamma}{n}\|f\|^2_{\mathbb{H}}.
\end{equation}

Next, we analyze $\EE \langle U_{i}, \Sigma^{2a}U_{j}\rangle$ in (\ref{eq:beta}) for $1\leq i < j \leq n$. Note that 
\begin{align*}
\EE \langle U_{i}, \Sigma^{2a}U_{j}\rangle = & \EE \langle H_{ii}+ \cdots + H_{i1}, \Sigma^{2a}(H_jj + \cdots + H_{j1})\rangle 
= \sum_{\ell=1}^j \sum_{k=1}^j \EE \langle H_{i\ell}, \Sigma^{2a}H_{jk}\rangle.
\end{align*}
We first consider $\ell \neq k$ and assume $\ell > k$, note that $i < j$, then 
\begin{align}
& \EE \langle H_{i\ell}, \Sigma^{2a} H_{jk} \rangle
=  \EE \langle S_iS_{i-1}\cdots S_{\ell+1}T_\ell T^{\ell-1}f^\ast, \Sigma^{2a}S_jS_{j-1}\cdots S_{k+1}T_kT^{k-1}f^\ast\rangle \nonumber \\
= & \EE \langle S_i S_{i-1}\cdots S_{\ell+1}T_{\ell}T^{\ell-1}f^\ast, \Sigma^{2a} S_j \cdots S_{\ell} \EE (S_{\ell -1}\cdots S_{k+1}T_k)T^{k-1}f^\ast \rangle
=  0.\label{eq:H:con}
\end{align}
Similarly, for $\ell <k$, 
$
\EE \langle H_{i\ell}, \Sigma^{2a}H_{jk} \rangle = \EE [\EE\langle H_{i\ell}, \Sigma^{2a}H_{jk} \rangle | X_j, \cdots X_k] = 0.
$
Therefore, 
\begin{align*}
 \EE \langle U_{i}, \Sigma^{2a} U_{j} \rangle 
= &  \sum_{\ell=1}^i \EE \langle H_{i\ell}, \Sigma^{2a}H_{j\ell} \rangle
= \sum_{\ell=1}^i \EE \langle S_i\cdots S_{\ell+1}T_\ell T^{\ell-1}f^\ast, \Sigma^{2a}S_j\cdots S_{\ell+1}T_\ell T^{\ell-1}f^\ast \rangle\\
= &\sum_{\ell=1}^i \EE \langle f^\ast, T^{\ell-1}T_\ell S_{\ell+1}\cdots S_i \Sigma^{2a}S_j \cdots S_i S_{i-1}\cdots S_{\ell+1}T_\ell T^{\ell-1}f^\ast \rangle \\
= & \sum_{\ell=1}^i \EE \langle f^\ast, T^{\ell-1}T_\ell S_{\ell+1}\cdots S_i \Sigma^{2a} (I-\gamma \Sigma)^{j-i}S_i \cdots S_{\ell+1}T_\ell T^{\ell-1}f^\ast \rangle.
\end{align*}
And 
$
\sum_{i<j} \EE \langle U_{i}, \Sigma^{2a} U_{j} \rangle
= \sum_{i=1}^{n-1} \sum_{\ell=1}^i \EE \langle f^\ast, T^{\ell-1}T_\ell S_{\ell+1} \cdots S_i \Sigma^{2a}(\sum_{j=i+1}^n (I -\gamma \Sigma)^{j-i})S_i \cdots S_{\ell+1}T_\ell T^{\ell-1}f^\ast\rangle.
$
Since 
$
\Sigma^{2a} \sum_{j=i+1}^n (I -\gamma \Sigma)^{j-i} =  \Sigma^{2a} \sum_{\ell=1}^{n-i}(I - \gamma \Sigma)^\ell 
\leq  \Sigma^{2a}(I -\gamma \Sigma)(\sum_{\ell=0}^{n-1}(I-\gamma \Sigma)^\ell) 
\leq  \Sigma^{2a}\sum_{\ell=1}^{n-1}(I-\gamma \Sigma)^{\ell} \equiv A, 
$
we have      
\begin{align}
& \sum_{i<j} \EE \langle U_{i}, \Sigma^{2a}U_{j} \rangle 
\leq  \sum_{i=1}^{n-1}\sum_{\ell=1}^i \EE \langle f^\ast, T^{\ell-1}T_\ell S_{\ell+1}\cdots S_i A S_i \cdots S_{\ell+1}T_\ell T^{\ell-1}f^\ast \rangle \nonumber\\
= & \sum_{\ell=1}^{n-1}\sum_{i=\ell}^{n-1} \langle f^\ast, T^{\ell-1}\EE (T_\ell S_{\ell+1}\cdots S_i A S_i \cdots S_{\ell+1}T_\ell)T^{\ell-1}f^\ast \rangle 
\leq \sum_{\ell=1}^{n-1} \langle f^\ast, T^{\ell-1}S (\sum_{i=\ell}^{n-1}(I-\gamma G)^{i-\ell})AT^{\ell-1}f^\ast \rangle \nonumber\\
\leq & \sum_{\ell=1}^{n-1} \langle f^\ast, T^{\ell-1}BAT^{\ell-1}f^\ast\rangle, \label{eq: bias_ij}
\end{align}
where $B=S\sum_{i=\ell}^{n-1}(I-\gamma G)^{i-\ell}$, and $BA= S(\sum_{i=0}^{n-1}(I-\gamma G)^i)A \leq nSA = 2n \EE (K_x \otimes K_x)A(K_x \otimes K_x) 
\leq  n\gamma^{-1}\EE \big((K_X \otimes K_X) \Sigma^{-1+2a}(K_X \otimes K_X)\big)
\leq  n \gamma^{-1} c_k \Sigma,$
where the last step is due to the fact that 
\begin{align*}
& \langle \EE\big((K_X \otimes K_X) \Sigma^{-1+2a}(K_X \otimes K_X)\big)f,f \rangle = 
\EE \langle (K_X \otimes K_X) \Sigma^{-1+2a}K_X f(X),f \rangle \\
= &\EE f(X)\langle K_X\Sigma^{-1+2a}K_X(X), f \rangle 
= \EE f^2(X)\langle \Sigma^{-1+2a}K_X, K_X \rangle \leq C \langle \Sigma f, f \rangle
\end{align*}
with $\langle \Sigma^{-1+2a}K_X, K_X \rangle = \sum_{\nu=1}^\infty \phi_\nu(X)\mu_\nu^{2a}\phi_\nu(X)\leq c_\phi^2 \sum_{\nu=1}^\infty \nu^{-2a\alpha}< \infty$ 
for $2a\alpha>1$.

By equation (\ref{eq: bias_ij}), we have $\sum_{i<j} \EE \langle U_{i}, \Sigma^{2a}U_{j} \rangle \leq \sum_{\ell=1}^{n-1} \langle f^\ast, T^{\ell-1}BAT^{\ell-1}f^\ast \rangle$. Recall $T= I - \gamma \Sigma$. For notation simlicity, let $C=BA$,  then $TCT$ can be written as  
$
(I-\gamma \Sigma) C (I-\gamma \Sigma) =  C - \gamma \Sigma C - \gamma C \Sigma + \gamma^2 \Sigma C \Sigma 
=  C- \gamma \Theta C = (I-\gamma \Theta)C, 
$
where $\Theta$ is an operator such that for any $C$, $\Theta C = \Sigma C + C \Sigma - \gamma^2 \Sigma C \Sigma$. Replacing $C$ with $BA$, we have $T^{\ell-1}BAT^{\ell-1} = (I-\gamma \Theta)^{\ell-1}BA$, and
$$ \sum_{i<j} \EE \langle U_{i}, \Sigma^{2a}U_{j} \rangle \leq \langle f^\ast, \sum_{\ell=1}^{n-1} (I-\gamma \Theta)^{\ell -1} BA f^\ast \rangle. 
$$
Since  $\sum_{\ell=1}^{n-1} (I-\gamma \Theta)^{\ell -1} \leq \gamma^{-1}\Theta^{-1}$, we further need to bound $\Theta^{-1}$. Let $C=\Theta^{-1}$, then $I = \Sigma \Theta^{-1} + \Theta^{-1}\Sigma - \gamma \Sigma \Theta^{-1}\Sigma$. Note that $\Sigma \Theta^{-1}\Sigma \leq tr(\Sigma) \Theta^{-1}\Sigma \leq c \Theta^{-1}\Sigma$, where $c$ is a constant. Then 
\begin{align*}
I \succeq & \Sigma \Theta^{-1} + \Theta^{-1}\Sigma - c\gamma\Theta^{-1}\Sigma
=  \Sigma \Theta^{-1} + (1-c\gamma) \Theta^{-1}\Sigma = (\Sigma \otimes I + (1-c\gamma)I \otimes \Sigma) \Theta^{-1}.
\end{align*}
Therefore, $\Theta^{-1} \preceq (\Sigma \otimes I + (1-c\gamma)I \otimes \Sigma)^{-1} I$, and 
\begin{align*}
\sum_{\ell=1}^{n-1} (I -\gamma \Theta)^{\ell -1} BA \preceq & \gamma^{-1}\Theta^{-1}n\gamma^{-1}\Sigma
\preceq \frac{1}{1+(1-c\gamma)} n\gamma^{-2} 
\end{align*}
Accordingly, we have $\frac{\gamma^2}{n^2}\sum_{i<j} \EE \langle U_{i}, \Sigma^{2a}U_{j} \rangle \lesssim  \frac{\gamma}{n\gamma}$.

Therefore, 
\begin{align*}
\EE \langle \bar{\beta}_{n}, \Sigma^{2a}\bar{\beta}_{n} \rangle = &
\frac{\gamma^2}{n^2} \sum_{i=1}^n \EE \langle U_{i}, \Sigma^{2a}U_{i} \rangle  + \frac{2\gamma^2}{n^2}\sum_{i<j} \EE \langle U_{i}, \Sigma^{2a}U_{j} \rangle 
 \lesssim \frac{\gamma}{n\gamma} \|f\|^2_{\mathbb{H}}.
\end{align*}
Then by Markov's inequality, we have 
$$
\PP\Big( \|\Sigma^a \big(\bar{\eta}_n^{bias} - \bar{\eta}_n^{bias,0}\big)\|^2_\mathbb{H} > \gamma^{-1/2}\EE \|\Sigma^a \bar{\beta}_{b,n}\|^2_\mathbb{H}\Big) \leq \gamma^{1/2} . 
$$
That is, $\|\bar{\eta}_n^{bias} - \bar{\eta}_n^{bias,0}\|^2_{\infty} \leq \frac{\gamma^{1/2}}{n\gamma }$ with probability at least $1-\gamma^{1/2}$. 
\end{proof}

\subsection{Proof the sup-norm bound of noise remainder in constant step case}\label{app:le:rem_var:con} 

Recall the noise remainder recursion follows 
$$
\eta_n^{noise} - \eta_n^{noise,0} = (I-\gamma K_{X_n}\otimes K_{X_n}) (\eta_{n-1}- \eta_{n-1}^{noise,0}) + \gamma(\Sigma-K_{X_n}\otimes K_{X_n})\eta_{n-1}^{noise,0}.
$$
Follow the recursion decomposition in Section \ref{subsec:sketch1}, we can split $\eta_n^{noise} - \eta_n^{noise,0}$ into higher order expansions as 
$$
\eta_n^{noise} = \eta_n^{noise,0} + \eta_n^{noise,1} + \eta_n^{noise,2} +\cdots + \eta_n^{noise,r} + \textrm{Remainder},
$$
where $\eta_n^{noise,d}$ can be viewed as $\eta_n^{noise,d} = (I-\gamma \Sigma)\eta_{n-1}^{noise,d} + \gamma \mathcal{E}_n^{d}$ and $\mathcal{E}_n^{d} = (\Sigma -  K_{X_n}\otimes K_{X_n})\eta_{n-1}^{noise,d-1}$  for $1\leq d\leq r$ and $r\geq 1$. The remainder term follows the recursion as  
$$
\eta_n^{noise} - \sum_{d=0}^r \eta_n^{noise,d} = (I -\gamma  K_{X_n}\otimes K_{X_n})(\eta_{n-1}^{noise} - \sum_{d=1}^r \eta_{n-1}^{noise,d}) + \gamma \mathcal{E}_n^{r+1}. 
$$

The following lemma (see \ref{app:le:noise_rem:con}) demonstrates that the high-order expansion terms $\bar{\eta}_n^{noise,d}$ (for $d\geq 1$) decrease as the value of $d$ increases. In particular, we first characterize the behavior of $\|\bar{\eta}_n^{noise,1}\|_\infty$ by representing it as a weighted empirical process and establish its convergence rate that $\|\bar{\eta}_n^{noise,1}\|_\infty=o(\|\bar{\eta}_n^{noise,0}\|_\infty)$ with high probability. Next, we show that $\|\bar{\eta}_n^{noise,d+1}\|_\infty= o(\|\bar{\eta}_n^{noise,d}\|_\infty)$ for $d\geq 1$ using mathematical induction. Finally, we bound $\bar{\eta}_n^{noise} - \sum_{d=0}^r \bar{\eta}_n^{noise,d}$ through its $\cH$-norm based on the property that $\|\bar{\eta}_n^{noise} - \sum_{d=0}^r \bar{\eta}_n^{noise,d}\|_\infty \leq \|\bar{\eta}_n^{noise} - \sum_{d=0}^r \bar{\eta}_n^{noise,d}\|_\cH$.

\begin{Lemma}\label{app:le:noise_rem:con}
Suppose the step size $ \gamma(n) = \gamma$ with $0< \gamma < n^{-\frac{2}{2+3\alpha}}$. 
Then 
\begin{enumerate}[$(a)$]
\item 
$$\PP\Big(\|\bar{\eta}_n^{noise,1}\|_\infty >  \sqrt{\gamma^{1/2}(n\gamma)^{1/\alpha}n^{-1}\log n} \Big)
\leq 2\gamma^{1/2}.$$ 

\item 
\begin{align*}
 \PP\Big( \|\bar{\eta}_n^{noise,d}\|^2_{\infty} \geq \gamma^{1/4}(n\gamma)^{1/\alpha} n^{-1} \Big)  
\leq &  (n\gamma)^{1/\alpha + 2 \varepsilon}\gamma^{d-1/4}.
\end{align*}
Furthermore, for $d\geq 2$ and $0<\gamma < n^{-\frac{2}{2+3\alpha}}$, we have $(n\gamma)^{1/\alpha + 2 \varepsilon}\gamma^{d-1/4}\leq \gamma^{1/4}$. 

\item  $$
\PP\Big( \|\bar{\eta}_i^{noise} - \sum_{d=0}^r \bar{\eta}_i^{noise,d}\|^2_{\infty} \geq \gamma^{1/4}(n\gamma)^{1/\alpha} n^{-1} \Big) 
\leq n^{-1} .
$$
with $r$ large enough. 
\end{enumerate}
Furthermore, combine (a)-(c), we have 
$$
\PP\Big( \|\bar{\eta}_n^{noise} - \bar{\eta}_n^{noise,0}\|^2_{\infty}  \geq C \gamma^{1/4}(n\gamma)^{1/\alpha} n^{-1} 
 \Big)\leq \gamma^{1/4}, 
$$
where $C$ is a constant. 

\end{Lemma}

\begin{proof}
\underline{\bf Proof of Lemma \ref{app:le:noise_rem:con} (a) by analyzing $\|\bar{\eta}_n^{noise,1}\|_{\infty}$.} First, we calculate the explicit expression of $\eta_n^{noise,1}$. 
Let $T= I-\gamma\Sigma$ and $T_n = \Sigma - K_{X_n}\otimes K_{X_n}$, then 
 $\eta_n^{noise,1}= T \eta_{n-1}^{noise,1} + \gamma T_n \eta_{n-1}^{noise,0}$ with $\eta_0^{noise,1}=0$. Therefore,  
\begin{align*}
\eta_n^{noise,1} = &\gamma \sum_{i=1}^{n-1} T^{n-i-1}T_{i+1}\eta_i^{noise,0}
=  \gamma^2 \sum_{i=1}^{n-1}\sum_{j=1}^i \epsilon_j T^{n-i-1}T_{i+1}T^{i-j}K_{X_j},
\end{align*}
where the last step is by plugging in $
\eta_i^{noise,0} = \gamma \sum_{j=1}^{i-1} T^{i-j}\epsilon_jK_{X_j} 
$ in (\ref{eq:noise:lead:exp}) with $\gamma= \gamma(n)$. Accordingly, 
\begin{align}
\bar{\eta}_n^{noise,1} = & \frac{\gamma^2}{n} \sum_{\ell=1}^{n-1}\sum_{i=1}^\ell \sum_{j=1}^i \epsilon_j T^{\ell-i}T_{i+1}T^{i-j}K_{X_j}
= \frac{\gamma^2}{n} \sum_{j=1}^{n-1}\big(\sum_{i=j}^{n-1}(\sum_{\ell=i}^{n-1}T^{\ell-i})T_{i+1}T^{i-j}K_{X_j}\big)\epsilon_j. \label{eq:eta_1:con}
\end{align}
Let $g_{j} = \sum_{i=j}^{n-1}(\sum_{\ell=i}^{n-1}T^{\ell-i})T_{i+1}T^{i-j}K_{X_j}$, where the randomness of $g_{j}$ involves $X_j,X_{j+1}, \dots, X_n$. Then $\bar{\eta}_n^{noise,1}(\cdot) = \frac{\gamma^2}{n} \sum_{j=1}^{n-1} \epsilon_j \cdot g_{j}(\cdot)$, which is a Gaussian process conditional on $(X_j,\dots, X_n)$. 

We can further express $g_{j}(\cdot)$ as a function of the eigenvalues and eigenfunctions that follows 
\begin{equation}\label{eq:gj:con}
g_{j}(\cdot) = \gamma^{-1} \sum_{\nu,k=1}^\infty \mu_\nu \sum_{i=j}^{n-1}(1-\gamma \mu_\nu)^{i-j}(1-(1-\gamma \mu_k)^{n-i})\phi_{i\nu k}\phi_\nu(X_j)\phi_k(\cdot)
\end{equation}
with $\phi_{i\nu k} = \phi_\nu(X_{i+1})\phi_k(X_{i+1})- \delta_{\nu k}$; we refer the proof to \cite{liu2023supp}.  Such expression can facilitate the downstream analysis of $\bar{\eta}_n^{noise, 1}$. Denote $a_{ij\nu}=(1-\gamma \mu_\nu)^{i-j}$ and $b_{ik}= 1-(1-\gamma \mu_k)^{n-i}$. Then $g_j$ can be simplified as  $g_{j}= \gamma^{-1}\sum_{\nu, k=1}^\infty \mu_\nu \big(\sum_{i=j}^{n-1}a_{ij\nu}b_{ik}\phi_{b,i\nu k} \big)\phi_\nu(X_j)\phi_k$. 

We are ready to prove that $\|\bar{\eta}_n^{noise,1} \|_{\infty} \leq \gamma^{\frac{1}{2}} n^{-\frac{1}{2}} (n\gamma)^{\frac{1}{2\alpha}}
$ where  $\bar{\eta}_n^{noise,1} (\cdot) = \frac{\gamma^2}{n}\sum_{j=1}^{n-1} \epsilon_j\cdot g_{j}(\cdot)$. It involves two steps: (1) for any fixed $s$, we see that $\bar{\eta}_n^{noise,1} (s) = \frac{\gamma^2}{n}\sum_{j=1}^{n-1} \epsilon_j\cdot g_{j}(s)$ is a weighted Gaussian random variable with variance $\frac{\gamma^4}{n^2}\sum_{j=1}^{n-1}g^2_{j}(s)$ conditional on $X_{1:n}= (X_1, \dots, X_n)$. Therefore, we first bound $\bar{\eta}_n^{noise,1} (s)$ with an exponentially decaying probability by characterizing $\sum_{j=1}^{n-1}g^2_{j}(s)$; (2) we then bridge $\bar{\eta}_n^{noise,1} (s)$ to $\|\bar{\eta}_n^{noise,1}\|_\infty$. We illustrate the details as follows. 

Conditional on $X_{1:n}$, 
$\bar{\eta}_n^{noise,1} (s) = \frac{\gamma^2}{n}\sum_{j=1}^{n-1} \epsilon_j\cdot g_{j}(s)$ is a weighted Gaussian random variable; by Hoeffding's inequality, 
\begin{equation}\label{eq:noise1:con:original}
\PP \Big( \frac{\gamma^2}{n}|\sum_{j=1}^{n-1} \epsilon_j\cdot g_{j}(s)| > u  \mid X_{1:n} \Big) \leq \exp\Big(-\frac{u^2 n^2}{\gamma^4 \sum_{j=1}^{n-1}g_j^2(s)}\Big). 
\end{equation}
We then bound $\sum_{j=1}^{n-1}\EE g_j^2(s)$. We separate $\sum_{j=1}^{n-1}g^2_j(s)$ as two parts as follows: 
\begin{align*} 
& \sum_{j=1}^{n-1}g^2_j(s) \\
\leq  & \gamma^{-2} \sum_{\nu, \nu'=1}^\infty \sum_{j=1}^{n-1} \mu_\nu \mu_{\nu'}(\phi_\nu(X_j)\phi_{\nu'}(X_j)-\delta_{\nu\nu'}) \sum_{i,\ell=j}^{n-1}a_{ij\nu}a_{\ell j\nu'}\sum_{k,k'=1}^\infty b_{ik}b_{\ell k'}\phi_{i\nu k}\phi_{\ell\nu' k'}\phi_k(s)\phi_{k'}(s)\\
& + \gamma^{-2}\sum_{\nu=1}^\infty \mu_\nu^2 \sum_{j=1}^{n-1}\sum_{i,\ell=j}^{n-1} a_{ij\nu}a_{\ell j\nu}\sum_{k,k'=1}^\infty b_{ik}b_{\ell k'}\phi_{i\nu k}\phi_{\ell \nu k'}\phi_k(s)\phi_{k'}(s)\\
= &\Delta_1 + \Delta_2,
\end{align*}   
where $\Delta_1$ involves the interaction terms indexed by $\nu, \nu'$ and $\Delta_2$ includes the terms that $\nu=\nu'$. 
Recall $b_{ik}=(1-(1-\gamma\mu_k)^{n-i})$. Then $b_{ik}<(1-(1-\gamma\mu_k)^{n})\equiv b_k$ for $1\leq i \leq n$. For $\Delta_1$, we have 
 \begin{align*}                           
\Delta_1 
 \leq &  \gamma^{-2}  \sum_{k,k'=1}^\infty b_{k}b_{k'}\phi_k(s)\phi_{k'}(s)
 \sum_{\nu, \nu'=1}^\infty \mu_\nu\mu_{\nu'}\sum_{j=1}^{n-1} \big(\phi_\nu(X_j)\phi_{\nu'}(X_j)-\delta_{\nu\nu'}\big)\sum_{i,\ell=j}^{n-1}a_{ij\nu}a_{\ell j\nu'}\phi_{i\nu k}\phi_{\ell \nu' k'}, 
 \end{align*}
 Take expectation on $\Delta_1$, we can see 
\begin{align}
& \EE |\sum_{\nu,\nu'=1}^\infty \mu_\nu\mu_{\nu'}\sum_{j=1}^{n-1} \big(\phi_\nu(X_j)\phi_{\nu'}(X_j)-\delta_{\nu\nu'}\big)\sum_{i,\ell=j}^{n-1} a_{ij\nu}a_{\ell j\nu'}\phi_{i\nu k}\phi_{\ell \nu' k'}|^2 \label{eq:noise:1:delta:con}\\
\leq & \big(\sum_{\nu,\nu'=1}^\infty \mu_\nu^{\frac{1+\varepsilon}{\alpha}}\mu_{\nu'}^{\frac{1+\varepsilon}{\alpha}}\big) \big(\sum_{\nu, \nu'=1}^\infty \mu_\nu^{2-\frac{1+\varepsilon}{\alpha}}\mu_{\nu'}^{2-\frac{1+\varepsilon}{\alpha}} \EE |\sum_{j=1}^{n-1} \big(\phi_\nu(X_j)\phi_{\nu'}(X_j)-\delta_{\nu\nu'}\big) \sum_{i,\ell=j}^{n-1} a_{ij\nu}a_{\ell j\nu'}\phi_{i\nu k}\phi_{\ell \nu' k'}|^2\big)\nonumber\\
\lesssim & \big(\sum_{\nu, \nu'=1}^\infty \mu_\nu^{\frac{1+\varepsilon}{\alpha}}\big)^2 
\big(\sum_{\nu, \nu'=1}^\infty \mu_\nu^{2-\frac{1+\varepsilon}{\alpha}}\mu_{\nu'}^{2-\frac{1+\varepsilon}{\alpha}} \sum_{j=1}^{n-1}\EE \big(\phi_\nu(X_j)\phi_{\nu'}(X_j)-\delta_{\nu\nu'}\big)^2\cdot \EE |\sum_{i,\ell=j}^{n-1} a_{ij\nu}a_{\ell j\nu'}\phi_{i\nu k}\phi_{\ell \nu' k'}|^2\big),\nonumber
\end{align} 
where the last step is due to the calculation that 
\begin{align*}
& \EE |\sum_{j=1}^{n-1} \big(\phi_\nu(X_j)\phi_{\nu'}(X_j)-\delta_{\nu\nu'}\big) \sum_{i,\ell=j}^{n-1} a_{ij\nu}a_{\ell j\nu'}\phi_{i\nu k}\phi_{\ell \nu' k'}|^2 \\
= & \sum_{j=1}^{n-1}  \EE \big(\phi_\nu(X_j)\phi_{\nu'}(X_j)-\delta_{\nu\nu'}\big)^2\cdot \EE |\sum_{i,\ell=j}^{n-1} a_{ij\nu}a_{\ell j\nu'}\phi_{i\nu k}\phi_{\ell \nu' k'}|^2 \\
& + 2 \sum_{j_1< j_2} \EE\big(\big(\phi_\nu(x_{j_1})\phi_{\nu'}(x_{j_1})-\delta_{\nu\nu'}\big)\big) \cdot \EE\big( \big(\phi_\nu(x_{j_2})\phi_{\nu'}(x_{j_2})-\delta_{\nu\nu'}\big)\cdot(\sum_{i,\ell=j_1}^{n-1} a_{ij_1\nu}a_{\ell j_1\nu'}\phi_{i\nu k}\phi_{\ell \nu' k'})\\
& \cdot (\sum_{i,\ell=j_2}^{n-1} a_{ij_2\nu}a_{\ell j_2\nu'}\phi_{i\nu k}\phi_{\ell \nu' k'}) \big)=   \sum_{j=1}^{n-1}  \EE \big(\phi_\nu(X_j)\phi_{\nu'}(X_j)-\delta_{\nu\nu'}\big)^2\cdot \EE |\sum_{i,\ell=j}^{n-1} a_{ij\nu}a_{\ell j\nu'}\phi_{i\nu k}\phi_{\ell \nu' k'}|^2,
\end{align*}
with $\EE\big(\phi_\nu(x_{j_1})\phi_{\nu'}(x_{j_1})-\delta_{\nu\nu'}\big) = 0$. 
Note that 
\begin{align*}
& \sum_{j=1}^{n-1}\big(\EE \big(\phi_\nu(X_j)\phi_{\nu'}(X_j)-\delta_{\nu\nu'}\big)^2\cdot \EE |\sum_{i,\ell=j}^{n-1} a_{ij\nu}a_{\ell j\nu'}\phi_{i\nu k}\phi_{\ell \nu' k'}|^2\big)
\leq  \sum_{j=1}^{n-1} \EE |\sum_{i,\ell=j}^{n-1} a_{ij\nu}a_{\ell j\nu'}\phi_{i\nu k}\phi_{\ell \nu' k'}|^2\\
= & \sum_{j=1}^{n-1}\sum_{i_1,i_2=j}^{n-1} \sum_{\ell_1,\ell_2=j}^{n-1} a_{i_1j\nu}a_{\ell_1 j \nu'}a_{i_2 j \nu}a_{\ell_2 j \nu'} \EE(\phi_{i_1\nu k}\phi_{i_2\nu k}\phi_{\ell_1\nu'k'}\phi_{\ell_2 \nu'k'})\\
\overset{(i)}\lesssim & \sum_{j=1}^n \big(\sum_{i=j}^{n-1} a_{ij\nu}^2a_{ij\nu'}^2 + \sum_{i,\ell=j}^{n-1}a_{ij\nu}^2a_{\ell j\nu'}^2  + \sum_{i,\ell=j}^{n-1} a_{ij\nu}a_{ij\nu'}a_{\ell j \nu}a_{\ell j \nu'}\big) \\
\leq & \sum_{j=1}^{n-1}\big(\sum_{i=j}^{n-1}a_{ij\nu}^2\big)\big(\sum_{i=j}^{n-1}a^2_{ij\nu'}\big) + \big(\sum_{i=j}^{n-1}a_{ij\nu}a_{ij\nu'}\big)^2.
\end{align*}
In the $(i)$-step, $\EE(\phi_{i_1\nu k}\phi_{i_2\nu k}\phi_{\ell_1\nu'k'}\phi_{\ell_2 \nu'k'})\neq 0$ if and only if the following cases hold: (1)$i_1 = i_2 = \ell_1= \ell_2$; (2) $i_1 = i_2$ and $\ell_1= \ell_2$; (3) $i_1= \ell_1$ and $i_2= \ell_2$. 
Recall $a_{ij\nu}=(1-\gamma \mu_\nu)^{i-j}$. Then we have 
\begin{align*}
 \sum_{i=j}^{n-1}a_{ij\nu}a_{ij\nu'} 
= & \sum_{i=j}^{n-1}[(1-\gamma \mu_\nu)(1-\gamma \mu_{\nu'})]^{i-j}
\leq   (1-(1-\gamma \mu_\nu)(1-\gamma \mu_{\nu'}))^{-1}
\leq   \gamma^{-1}(\mu_\nu + \mu_{\nu'})^{-1}.
\end{align*}
For $\sum_{i=j}^{n-1} a^2_{ij\nu}$, we have 
$
\sum_{i=j}^{n-1} a^2_{ij\nu} = \sum_{i=j}^{n-1} (1-\gamma \mu_\nu)^{2(i-j)}
\lesssim \gamma^{-1} \mu_\nu^{-1} .  
$ 
Therefore, 
\begin{align*}
& \EE |\sum_{\nu,\nu'=1}^\infty \mu_\nu\mu_{\nu'}\sum_{j=1}^{n-1} \big(\phi_\nu(X_j)\phi_{\nu'}(X_j)-\delta_{\nu\nu'}\big)\sum_{i,\ell=j}^{n-1} a_{ij\nu}a_{\ell j\nu'}\phi_{i\nu k}\phi_{\ell \nu' k'}|^2\\
\lesssim & (\sum_{\nu=1}^\infty \mu_\nu^{\frac{1+\varepsilon}{\alpha}})^2 \sum_{\nu,\nu'=1}^\infty \mu_\nu^{2-\frac{1+\varepsilon}{\alpha}}\mu_{\nu'}^{2-\frac{1+\varepsilon}{\alpha}} \sum_{j=1}^{n-1}\big(\gamma^{-2}(\mu_\nu+ \mu_\nu')^{-2} + \gamma^{-1}\mu_\nu^{-1}\gamma^{-1}\mu_{\nu'}^{-1}\big)\\
\lesssim & n\gamma^{-2}\big(\sum_{\nu,\nu'=1}^\infty \mu_\nu^{1-\frac{1+\varepsilon}{\alpha}}\mu_{\nu'}^{1-\frac{1+\varepsilon}{\alpha}}+\sum_{\nu,\nu'=1}^\infty \mu_\nu^{2-\frac{1+\varepsilon}{\alpha}}\mu_{\nu'}^{2-\frac{1+\varepsilon}{\alpha}}(\mu_\nu + \mu_{\nu'})^{-2} \big) \lesssim n\gamma^{-2}, 
\end{align*} 
with $\varepsilon\to 0$. 
The final step is due to the fact that 
\begin{align*}
& \sum_{\nu,\nu'=1}^\infty \mu_\nu^{2-\frac{1+\varepsilon}{\alpha}}\mu_{\nu'}^{2-\frac{1+\varepsilon}{\alpha}}(\mu_\nu + \mu_{\nu'})^{-2} \\
=& \sum_{\nu,\nu'=1}^\infty \frac{\mu_\nu \mu_{\nu'}}{(\mu_\nu+\mu_{\nu'})^2}\mu_{\nu}^{1-\frac{1+\varepsilon}{\alpha}}\mu_{\nu'}^{1-\frac{1+\varepsilon}{\alpha}} \leq \sum_{\nu,\nu'=1}^\infty \mu_{\nu}^{1-\frac{1+\varepsilon}{\alpha}}\mu_{\nu'}^{1-\frac{1+\varepsilon}{\alpha}} = (\sum_{\nu=1}^\infty \mu_{\nu}^{1-\frac{1+\varepsilon}{\alpha}})^2 \leq C. 
\end{align*}
Since $b_{k}\leq \min\{1, n\gamma\mu_k\}$ and accordingly, 
$
 \sum_{k,k' =1}^\infty b_k b_{k'}  =  (\sum_{k=1}^\infty  (1-(1-\gamma \mu_k)^n))^2  \leq (n\gamma)^{\frac{2}{\alpha}}. 
$ Therefore, we have 
\begin{equation}\label{eq:delta_1}
\EE \Delta_1 \leq \sqrt{\EE \Delta_1^2} \leq \sqrt{n} \gamma^{-3} (n\gamma)^{\frac{2}{\alpha}}.
\end{equation}

For $\Delta_2$, we rewrite $\Delta_2$ as 
\begin{align}
\Delta_2= & \gamma^{-2}\sum_{\nu=1}^\infty \mu_\nu^2 \sum_{j=1}^{n-1}\sum_{i,\ell=j}^{n-1} a_{ij\nu}a_{\ell j\nu}\sum_{k,k'=1}^\infty b_{ik}b_{\ell k'}\phi_{i\nu k}\phi_{\ell \nu k'} \nonumber\\ 
= & \gamma^{-2}\sum_{\nu=1}^\infty \mu_\nu^2 \sum_{j=1}^{n-1}\sum_{j\leq i< \ell \leq n-1} a_{ij\nu}a_{\ell j\nu}\sum_{k,k'=1}^\infty b_{ik}b_{\ell k'}\phi_{i\nu k}\phi_{\ell \nu k'}\nonumber\\ 
 & \quad + \gamma^{-2}\sum_{\nu=1}^\infty \mu_\nu^2 \sum_{j=1}^{n-1}w_j^2 \sum_{i=j}^{n-1} a^2_{ij\nu}\sum_{k,k'=1}^\infty b_{ik}b_{\ell k'}\phi_{i\nu k}\phi_{i \nu k'} \nonumber\\
= &\Delta_{21} + \Delta_{22}, \label{eq:noise:con:delta2}
\end{align}
where $\Delta_{21}$ includes the terms that $i\neq \ell$ and $\Delta_{22}$ includes the terms that  $i=\ell$. 
For $\Delta_{21}$, with any positive $\varepsilon\to 0$, we have 
\begin{align*}
 \Delta_{21}
\leq  &  2\gamma^{-2}\sum_{k,k'=1}^\infty b_{k}b_{k'}\sum_{\nu=1}^\infty \mu_\nu^2  \sum_{j=1}^{n-1} \sum_{j\leq i< \ell \leq n-1} a_{ij\nu}a_{\ell j\nu} \phi_{i\nu k}\phi_{\ell \nu k'} \\
= & 2 \gamma^{-2}\sum_{k,k'=1}^\infty b_{k}b_{k'}  \sum_{\nu=1}^\infty \mu_\nu^{\frac{1+\varepsilon}{2\alpha}} \mu_\nu^{2-\frac{1+\varepsilon}{2\alpha}}\sum_{j=1}^{n-1} \sum_{j\leq i< \ell \leq n-1} a_{ij\nu}a_{\ell j\nu} \phi_{i\nu k}\phi_{\ell \nu k'}\\
\leq & 2 \gamma^{-2}
 \sum_{k,k'=1}^\infty b_{k}b_{k'} \sqrt{\sum_{\nu=1}^\infty \mu_\nu^{\frac{1+2\varepsilon}{\alpha}} }\sqrt{ \sum_{\nu=1}^\infty  \mu_\nu^{4-\frac{1+2\varepsilon}{\alpha}}\big(\sum_{j=1}^{n-1}\sum_{j\leq i< \ell \leq n-1} a_{ij\nu}a_{\ell j\nu} \phi_{i\nu k}\phi_{\ell \nu k'}\big)^2}.
\end{align*}
To bound the expectation of $\Delta_{21}$, we need to bound $\EE |\sum_{j=1}^{n-1}\sum_{j\leq i< \ell \leq n-1} a_{ij\nu}a_{\ell j\nu} \phi_{i\nu k}\phi_{\ell \nu k'}|^2$. Note that 
\begin{align*}
& \EE |\sum_{j=1}^{n-1}\sum_{j\leq i< \ell \leq n-1} a_{ij\nu}a_{\ell j\nu} \phi_{i\nu k}\phi_{\ell \nu k'}|^2
\lesssim  \EE |\sum_{j=1}^{n-1}\sum_{j\leq i< \ell \leq n-1} a_{ij\nu}a_{\ell j\nu} \phi_{i\nu k}\phi_{\ell \nu k'}|^2\\
= & \sum_{j,d=1}^{n-1} \EE \big(\sum_{j\leq i< \ell \leq n-1} a_{ij\nu}a_{\ell j\nu} \phi_{i\nu k}\phi_{\ell \nu k'}\big)\big(\sum_{d\leq i< \ell \leq n-1} a_{id\nu}a_{\ell d\nu}\phi_{i\nu k}\phi_{\ell \nu k'}\big)\\
= & \sum_{j=1}^n \EE |\sum_{j\leq i< \ell \leq n-1} a_{ij\nu}a_{\ell j\nu} \phi_{i\nu k}\phi_{\ell \nu k'}|^2 \\
& + 2 \sum_{d=1}^{n-1}\sum_{j=1}^{d-1} \EE \big(\sum_{d\leq i< \ell \leq n-1} a_{ij\nu}a_{\ell j\nu}\phi_{i\nu k}\phi_{\ell \nu k'}\big)\big(\sum_{d\leq i< \ell \leq n-1} a_{id\nu}a_{\ell d\nu} \phi_{i\nu k}\phi_{\ell \nu k'} \big) \\
& \overset{(i)}+ 2 \sum_{d=1}^{n-1}\sum_{j=1}^{d-1} \EE \big(\sum_{j\leq i< \ell \leq d-1} a_{ij\nu}a_{\ell j\nu}\phi_{i\nu k}\phi_{\ell \nu k'}\big)\big(\sum_{d\leq i< \ell \leq n-1} a_{id\nu}a_{\ell d\nu} \phi_{i\nu k}\phi_{\ell \nu k'} \big),
\end{align*}
where the last term $(i)$ is $0$. Then we have 
\begin{align*}
&  \sum_{d=1}^{n-1}\sum_{j=1}^{d-1} \EE \big(\sum_{d\leq i< \ell \leq n-1} a_{ij\nu}a_{\ell j\nu} \phi_{i\nu k}\phi_{\ell \nu k'}\big)\big(\sum_{d\leq i< \ell \leq n-1} a_{id\nu}a_{\ell d\nu} \phi_{i\nu k}\phi_{\ell \nu k'} \big)\\
= & \sum_{d=1}^{n-1}\sum_{j=1}^{d-1} \sum_{d\leq i< \ell \leq n-1} a_{ij\nu}a_{\ell j\nu} a_{id\nu}a_{\ell d\nu}\EE \phi^2_{i\nu k}\phi^2_{\ell \nu k'} 
\lesssim \sum_{j<d} \sum_{d\leq i< \ell \leq n-1} a_{ij\nu}a_{\ell j\nu} a_{id\nu}a_{\ell d\nu}\\
= & \sum_{d=1}^{n-1}\sum_{j=1}^{d-1} \sum_{d\leq i< \ell \leq n-1} (1-\gamma \mu_\nu)^{i-j}(1-\gamma \mu_\nu)^{\ell - j} (1-\gamma \mu_\nu)^{i-d}(1-\gamma \mu_\nu)^{\ell-d} \\
= & 2 \sum_{d=1}^{n-1}\sum_{j=1}^{d-1} \big[\sum_{d\leq i < \ell \leq n-1} (1-\gamma \mu_\nu)^{2(i-d)}(1-\gamma \mu_\nu)^{2(\ell-d)} \big](1-\gamma \mu_\nu)^{2(d-j)}\\
\leq & 2\big( \sum_{d=1}^{n-1}\sum_{j=1}^{d-1} (1-\gamma \mu_\nu)^{2(d-j)}\big)\big(\sum_{i=d}^{n-1}(1-\gamma \mu_\nu)^{2(i-d)}\big)\big(\sum_{\ell=d}^{n-1}(1-\gamma \mu_\nu)^{2(\ell-d)}\big)
\lesssim n(\gamma \mu_\nu)^{-3}.
\end{align*}
Then accordingly, 
\begin{equation}\label{eq:delta_21}
\EE \Delta_{21}\leq \gamma^{-2}\sum_{k,k'=1}^\infty b_k b_{k'} \sqrt{\sum_{\nu=1}^\infty \mu_\nu^{\frac{1+2\epsilon}{\alpha}} } \cdot \sqrt{\sum_{\nu=1}^\infty \mu_\nu^{4-\frac{1+2\epsilon}{\alpha}}|\sum_{j=1}^{n-1}\sum_{j\leq i< \ell \leq n-1} a_{ij\nu}a_{\ell j\nu} \phi_{i\nu k}\phi_{\ell \nu k'}|^2}
\lesssim  (n\gamma)^{\frac{2}{\alpha}}\sqrt{n}\gamma^{-\frac{7}{2}}. 
\end{equation}
For $\Delta_{22}$, we have 
\begin{align*}
& \gamma^{-2}\sum_{k,k'=1}^\infty b_k b_{k'} \sum_{\nu=1}^\infty \mu_\nu^2 \sum_{j=1}^{n-1} \sum_{i=j}^{n-1} a^2_{ij\nu}\phi_{i\nu k}\phi_{i \nu k'}\\
= & \gamma^{-2} \sum_{k,k'=1}^\infty b_k b_{k'} \sum_{\nu=1}^\infty \mu_\nu^2 \sum_{j=1}^{n-1} \sum_{i=j}^{n-1} a^2_{ij\nu} (\phi_{i\nu k}\phi_{i \nu k'}- \EE(\phi_{i\nu k}\phi_{i \nu k'}))\\ + &  \gamma^{-2} \sum_{k,k'=1}^\infty b_k b_{k'} \sum_{\nu=1}^\infty \mu_\nu^2 \sum_{j=1}^{n-1} \sum_{i=j}^{n-1} a^2_{ij\nu} \EE(\phi_{i\nu k}\phi_{i \nu k'})\\
= & \Delta_{22}^{(1)}+ \Delta_{22}^{(2)}. 
\end{align*}
We first bound $| \Delta_{22}^{(1)} | $. 
\begin{align*}
& | \Delta_{22}^{(1)} | 
\leq & \gamma^{-2}
 \sum_{k,k'=1}^\infty b_{k}b_{k'}\sqrt{\sum_{\nu=1}^\infty \mu_\nu^{\frac{1+2\epsilon}{\alpha}} }\sqrt{ \sum_{\nu=1}^\infty  \mu_\nu^{4-\frac{1+2\epsilon}{\alpha}}\big(\sum_{j=1}^{n-1} \sum_{i=j}^{n-1} a^2_{ij\nu} (\phi_{i\nu k}\phi_{i \nu k'}- \EE(\phi_{i\nu k}\phi_{i \nu k'}))\big)^2}
\end{align*} 
Notice that 
\begin{align*}
& \EE\big( \sum_{j=1}^{n-1} \sum_{i=j}^{n-1}  a^2_{ij\nu} (\phi_{i\nu k}\phi_{i \nu k'}- \EE(\phi_{i\nu k}\phi_{i \nu k'})) |^2
=  \sum_{i=1}^{n-1} \EE \big(\sum_{j=1}^{i}  a^2_{ij\nu} (\phi_{i\nu k}\phi_{i \nu k'}- \EE(\phi_{i\nu k}\phi_{i \nu k'})) \big)^2\\
= & \sum_{i=1}^{n-1} \EE \big(\sum_{j=1}^{i}  a^2_{ij\nu} (\phi_{i\nu k}\phi_{i \nu k'}- \EE(\phi_{i\nu k}\phi_{i \nu k'}))\big) ^2 \\
& + 2\sum_{1\leq i_1< i_2\leq n-1} 
\big( \sum_{j=1}^{i_1} a^2_{i_1j\nu} (\phi_{i_1\nu k}\phi_{i_1 \nu k'}- \EE(\phi_{i_1\nu k}\phi_{i_1 \nu k'}))\big) \cdot \big( \sum_{j=1}^{i_2}  a^2_{i_2j\nu} (\phi_{i_2\nu k}\phi_{i_2 \nu k'}- \EE(\phi_{i_2\nu k}\phi_{i_2 \nu k'}))\big)\\
= & \sum_{i=1}^{n-1} \EE \big(\sum_{j=1}^{i}  a^2_{ij\nu} (\phi_{i\nu k}\phi_{i \nu k'}- \EE(\phi_{i\nu k}\phi_{i \nu k'}))\big) ^2,
\end{align*}
since $\EE \big(\phi_{i_1\nu k}\phi_{i_1 \nu k'}- \EE(\phi_{i_1\nu k}\phi_{i_1 \nu k'})\big)=0$. Then we have 

$$
\EE\big( \sum_{j=1}^{n-1} \sum_{i=j}^{n-1}a^2_{ij\nu} (\phi_{i\nu k}\phi_{i \nu k'}- \EE(\phi_{i\nu k}\phi_{i \nu k'})) \big)^2
\lesssim  \sum_{i=1}^{n-1}(\sum_{j=1}^i a^2_{ij\nu})^2  \EE (\phi_{i\nu k}\phi_{i \nu k'}- \EE(\phi_{i\nu k}\phi_{i \nu k'}))^2
\lesssim n (\gamma^{-1}\mu_\nu^{-1})^2 
$$
due to the property that $\sum_{j=1}^i a^2_{ij\nu} = \sum_{j=1}^i (1-\gamma \mu_\nu)^{2(i-j)}\leq  \gamma^{-1}\mu_\nu^{-1}$. Accordingly, we have 
$$
 \sum_{k,k'=1}^\infty b_{k}b_{k'} \sqrt{\sum_{\nu=1}^\infty \mu_\nu^{\frac{1+2\epsilon}{\alpha}} }\sqrt{ \sum_{\nu=1}^\infty  \mu_\nu^{4-\frac{1+2\epsilon}{\alpha}}\big(\sum_{j=1}^{n-1} \sum_{i=j}^{n-1} a^2_{ij\nu} (\phi_{i\nu k}\phi_{i \nu k'}- \EE(\phi_{i\nu k}\phi_{i \nu k'}))\big)^2}
= O_P(\sqrt{n}\gamma^{-1}(n\gamma)^{\frac{2}{\alpha}}). 
$$
Therefore, 
\begin{equation}\label{eq:delta22_1}
\EE \Delta^{(1)}_{22}\lesssim \sqrt{n}\gamma^{-3}(n\gamma)^{\frac{2}{\alpha}}. 
\end{equation}

We next deal with $\Delta^{(2)}_{22}$. 
\begin{align}
&\EE  \Delta^{(2)}_{22}
= \gamma^{-2} \sum_{k,k'=1}^\infty   \sum_{\nu=1}^\infty \mu_\nu^2 \sum_{j=1}^{n-1} \sum_{i=j}^{n-1} a^2_{ij\nu}b_{ik}b_{ik'} \EE(\phi_{i\nu k}\phi_{i \nu k'}) \nonumber\\
\leq & \gamma^{-2}\sum_{k,k'=1}^\infty \sum_{\nu=1}^\infty\mu_\nu^2 \sum_{j=1}^{n-1}\sum_{i=j}^{n-1} a^2_{ij\nu} b_{ik}b_{ik'}\EE(\phi^2_\nu(X_{i+1})\phi_k(X_{i+1})\phi_{k'}(X_{i+1}))\nonumber\\
= & \gamma^{-2} \sum_{\nu=1}^\infty\mu_\nu^2 \sum_{j=1}^{n-1}\sum_{i=j}^{n-1} a^2_{ij\nu} \EE \big(\phi^2_\nu(X_{i+1})\cdot\big(\sum_{k=1}^\infty b_{ik}\phi_k(X_{i+1})\big)^2\big)\nonumber\\
\leq& c_\phi^2\gamma^{-2} \sum_{\nu=1}^\infty\mu_\nu^2 \sum_{j=1}^{n-1}\sum_{i=j}^{n-1} a^2_{ij\nu} \EE \big(\sum_{k=1}^\infty b_{ik}\phi_k(X_{i+1})\big)^2
= c_\phi^2\gamma^{-2}  \sum_{\nu=1}^\infty\mu_\nu^2 \sum_{j=1}^{n-1}\sum_{i=j}^{n-1} a^2_{ij\nu} \sum_{k=1}^\infty b^2_{ik} \lesssim  \gamma^{-3}   n (n\gamma)^{\frac{1}{\alpha}} \label{eq:delta_22_2}
\end{align}
Combine equation (\ref{eq:delta_1}), (\ref{eq:delta_21}), (\ref{eq:delta22_1}), and (\ref{eq:delta_22_2}) together, and notice that 
 $\sqrt{n}\gamma^{-\frac{7}{2}}(n\gamma)^{\frac{2}{\alpha}} \leq n\gamma^{-3}(n\gamma)^{\frac{1}{\alpha}}$ for $\gamma \geq n^{-1}$, we have 
$$
\EE \sum_{j=1}^{n-1} g^2_j(s)\leq n\gamma^{-3}(n\gamma)^{\frac{1}{\alpha}}. 
$$
Define an event $\mathcal{E}_1= \{\sum_{j=1}^{n-1} g^2_j(s)\leq \gamma^{-7/2} n (n\gamma)^{1/\alpha}\} $, by Markov inequality, $\PP\big(\mathcal{E}_1\big)> 1-\gamma^{1/2}$. 
Conditional on the event $\mathcal{E}_1$, and let $u=C n^{-\frac{1}{2}}\gamma^{\frac{1}{4}}(n\gamma)^{\frac{1}{2\alpha}}\sqrt{\log n}$ in equation (\ref{eq:noise1:con:original}), we have 
\begin{equation}\label{eq:noise1:con}
\PP \Big( \frac{\gamma^2}{n}\bigl|\sum_{j=1}^{n-1} \epsilon_j\cdot g_{j}(s)\bigr| > C n^{-\frac{1}{2}}\gamma^{\frac{1}{4}}(n\gamma)^{\frac{1}{2\alpha}}\sqrt{\log n} \bigm| \mathcal{E}_1\Big) \leq \exp\Big(-C' \log n\Big). 
\end{equation}
Combined with the Lemma bridging $\bar{\eta}_n^{noise,1}(t)$ and $\|\bar{\eta}_n^{noise,1}\|_\infty$ 
in Supplementary \cite{liu2023supp}, we achieve the result.

\underline{\bf Next, we prove Lemma \ref{app:le:noise_rem:con} (b) and analyze $\|\bar{\eta}_n^{noise,d}\|_{\infty}$  for $d\geq 2$.} 

Note that $\|\bar{\eta}_n^{noise,d}\|_{\infty} \leq \|\Sigma^a \bar{\eta}_n^{noise,d}\|_\mathbb{H}$. In the following part, we focus on $\EE \|\Sigma^a \bar{\eta}_n^{noise,d}\|^2_\mathbb{H}$. 
Recall in Section \ref{sec:proof_sketch}, $\eta_n^{noise,d}$ follows the recursion as 
$
\eta_n^{noise,d} = (I-\gamma \Sigma)\eta_{n-1}^{noise,d} + \gamma \mathcal{E}_n^{d},
$
where $\mathcal{E}_n^{d} = (\Sigma -K_{X_n}\times K_{X_n})\eta_{n-1}^{noise,d}$ for $d\geq 1$ and $\mathcal{E}_n^{d} = \varepsilon_n$ for $d=0$. 

Let $T=I-\gamma \Sigma$, then $\eta_j^{noise,d} =\gamma \sum_{k=1}^j T^{j-k}\mathcal{E}_k^{d}$, and $\bar{\eta}_n^{noise,d} = \gamma \frac{1}{n}\sum_{j=1}^n \sum_{k=1}^j T^{j-k}\mathcal{E}_k^{d}$. 
\begin{align*}
&  \EE \langle \bar{\eta}_n^{noise,d}, \Sigma^{2a}\bar{\eta}_n^{noise,d} \rangle \\
= & \frac{\gamma^2}{n^2} \EE \langle \sum_{j=1}^n \sum_{k=1}^j T^{j-k} \mathcal{E}_k^{d}, \Sigma^{2a}\sum_{j=1}^n \sum_{k=1}^j T^{j-k}\mathcal{E}_k^{d} \rangle\\
= & \frac{\gamma^2}{n^2} \sum_{k=1}^n \EE \langle  \sum_{k=1}^n (\sum_{j=k}^n T^{j-k}) \mathcal{E}_k^{d}, \Sigma^{2a}\sum_{k=1}^n (\sum_{j=k}^n T^{j-k})\mathcal{E}_k^{d}\rangle \\
= & \frac{\gamma^2}{n^2} \sum_{k=1}^n \EE \langle M_{n,k}\mathcal{E}_k^{d}, \Sigma^{2a}M_{n,k}\mathcal{E}_k^{d} \rangle 
=  \frac{\gamma^2}{n^2} \sum_{k=1}^n \EE \tr(\mathcal{E}_k^{d} M_{n,k}\Sigma^{2a}M_{n,k}\mathcal{E}_k^{d})
=  \frac{\gamma^2}{n^2} \sum_{k=1}^n \EE \tr(M_{n,k}\Sigma^{2a}M_{n,k}\mathcal{E}_k^{d} \otimes \mathcal{E}_k^{d})\\
= & \frac{\gamma^2}{n^2} \sum_{k=1}^n  \tr(M_{n,k}\Sigma^{2a}M_{n,k}\EE\big(\mathcal{E}_k^{d} \otimes \mathcal{E}_k^{d})\big)
\lesssim  \frac{\gamma^{2+d}}{n^2} \sum_{k=1}^n \tr\big(M_{n,k} \Sigma^{2a} M_{n,k}\Sigma \big)
\end{align*}
where we use the property that $\EE\big(\mathcal{E}_k^{d} \otimes \mathcal{E}_k^{d}\big)\lesssim \gamma^d\Sigma$. 
Since $M_{n,k}= \sum_{j=k}^n T^{j-k} = I + T + T^2 + \cdots + T^{n-k} \leq nI$,
 then $M_{n,k}\Sigma^{2a} \leq n \Sigma^{2a}$. On the other hand, 
$M_{n,k}\Sigma^{2a} = \gamma^{-1}\Sigma^{-1}(I-T^{n-k})\Sigma^{2a} \preceq \gamma^{-1}\Sigma^{2a-1}$. Therefore, we have 
$$
M_{n,k}\Sigma^{2a} \preceq (n\Sigma^{2a})^{q}(\gamma^{-1}\Sigma^{2a-1})^{1-q} 
$$
with $0\leq q \leq 1$. 
Also, $M_{n,k}\Sigma \preceq \gamma^{-1}\Sigma^{-1}(I-T^{n-k})\Sigma \preceq \gamma^{-1}I$. Then 
$$
\tr\big(M_{n,k} \Sigma^{2a} M_{n,k}\Sigma \big) \leq 
n^q \gamma^{q-1}\gamma^{-1}\sum_{\nu=1}^\infty \mu_\nu^{2aq+ (2a-1)(1-q)}. 
$$
Therefore, 
$
\EE \langle \bar{\eta}_n^{noise,d}, \Sigma^{2a}\bar{\eta}_n^{noise,d} \rangle \lesssim 
\frac{\gamma^{2+d}}{n^2} n \gamma^{-1}n^{q}\gamma^{q-1}\sum_{\nu=1}^\infty \mu_\nu^{2a-1+q}
\leq (n\gamma)^q n^{-1}\gamma^d \sum_{\nu=1}^\infty \mu_\nu^{2a-1+q}.
$
Let $2a-1+q=1/\alpha + \varepsilon$ with $a=1/2-1/(2\alpha)-\varepsilon$ and $\varepsilon \to 0$, then we have $\sum_{\nu=1}^\infty \mu_\nu^{2a-1+q}=\sum_{j=1}^\infty \mu_j^{1/\alpha+\varepsilon} = \sum_{j=1}^\infty j^{-1-\alpha \varepsilon} < \infty$ and 
$$
\EE \langle \bar{\eta}_n^{noise,d}, \Sigma^{2a}\bar{\eta}_n^{noise,d} \rangle \lesssim (n\gamma)^{1/\alpha} n^{-1} (n\gamma)^{1/\alpha + 2 \varepsilon}\gamma^{d}. 
$$
Through Markov's inequality, we have 
\begin{align*}
 \PP\Big( \|\bar{\eta}_n^{noise,d}\|^2_{\infty} \geq \gamma^{1/4}(n\gamma)^{1/\alpha} n^{-1} \Big)  
\leq &  \frac{\EE\|\Sigma^a \bar{\eta}_n^{noise,d}\|^2_{\mathbb{H}}}{\gamma^{1/4}(n\gamma)^{1/\alpha} n^{-1}}\leq  (n\gamma)^{1/\alpha + 2 \varepsilon}\gamma^{d-1/4}.
\end{align*}
For $d\geq 2$ and $0<\gamma < n^{-\frac{2}{2+3\alpha}}$, we have $(n\gamma)^{1/\alpha + 2 \varepsilon}\gamma^{d-1/4}\leq 1/2\gamma^{1/4}$.

\underline{\bf Proof of Lemma \ref{app:le:noise_rem:con} (c) - the reminder term $\bar{\eta}_n^{noise} - \sum_{d=0}^r \bar{\eta}_n^{noise,d}$}. 
Note that for any $f\in \mathbb{H}$, $|f(x)| = |\langle f, K_x \rangle_\mathbb{H}| \leq |K_x|_\mathbb{H} \|f\|_\mathbb{H} \leq C \|f\|_\mathbb{H}$.  Therefore, $\|\bar{\eta}_n^{noise} - \sum_{d=0}^r \bar{\eta}_n^{noise,d}\|_{\infty} \leq \|\bar{\eta}_n^{noise} - \sum_{d=0}^r \bar{\eta}_n^{noise,d}\|_{\mathbb{H}}$. Next, we will bound $\|\bar{\eta}_n^{noise} - \sum_{d=0}^r \bar{\eta}_n^{noise,d}\|_{\mathbb{H}}$. 

For $i=1,\dots,n$, recall
$
\eta_i^{noise} - \sum_{d=0}^r \eta_i^{noise,d} = (I-\gamma  K_{X_i}\otimes K_{X_i})(\eta_{i-1}^{noise} - \sum_{d=0}^r \eta_{i-1}^{noise,d}) + \gamma \mathcal{E}_i^{r+1},
$
we have 
$$
\|\eta_i^{noise} - \sum_{d=0}^r \eta_i^{noise,d}\|_\mathbb{H} \leq \|\eta_{i-1}^{noise} - \sum_{d=0}^r \eta_{i-1}^{noise,d}\|_\mathbb{H} + \gamma \|\mathcal{E}_i^{r+1}\|_\mathbb{H}\leq \sum_{j=1}^i \gamma\|\mathcal{E}_j^{r+1}\|_\mathbb{H}.  
$$
Accordingly, 
$\EE \|\eta_i^{noise} - \sum_{k=0}^d \eta_i^{noise,k}\|_\mathbb{H}^2 \leq \gamma^2 \sum_{j=1}^i \big(\sum_{j=1}^i  \EE \|\mathcal{E}_j^{d+1}\|^2_\mathbb{H} \big)  %
$.  Since $\EE \|\mathcal{E}_j^{d+1}\|_\mathbb{H}^2 = \EE \tr(\mathcal{E}_j^{d+1}\otimes \mathcal{E}_j^{d+1}) = \tr \EE(\mathcal{E}_j^{d+1}\otimes \mathcal{E}_j^{d+1}) \leq \sigma^2 \gamma^{d+1}R^{2d+2}\tr(\Sigma)$, we have 
\begin{align*}
\EE \|\eta_i^{noise} - \sum_{d=0}^r \eta_i^{noise,d}\|_\mathbb{H}^2 \leq  \gamma^2 i^2 \sigma^2 \gamma^{r+1}R^{2r+2}\tr(\Sigma),
\end{align*}
and accordingly 
\begin{align}
\EE \|\bar{\eta}_n^{noise} - \sum_{d=0}^r \bar{\eta}_n^{noise,d}\|_\mathbb{H}^2 \leq &  \frac{2}{n}\sum_{i=1}^n \EE \|\eta_i^{noise} - \sum_{d=0}^r \eta_i^{noise,d}\|_\mathbb{H}^2 \nonumber\\
\leq &  \sigma^2 \gamma^{r+3}R^{2r+2}\tr(\Sigma) \frac{1}{n}\sum_{i=1}^n 
i^2  \leq \sigma^2 \gamma^{r+3}R^{2r+2}\tr(\Sigma) n^2. \label{eq:app:remind_b}
\end{align}
By Markov inequality, 
$$
\PP\Big( \|\bar{\eta}_i^{noise} - \sum_{d=0}^r \bar{\eta}_i^{noise,d}\|^2_{\infty} \geq \gamma^{1/4}(n\gamma)^{1/\alpha} n^{-1} \Big) \leq \frac{\EE\|\bar{\eta}_n^{noise} - \sum_{d=0}^r \bar{\eta}_n^{noise,d}\|^2_{\mathbb{H}}}{\gamma^{1/4}(n\gamma)^{1/\alpha} n^{-1}} \leq 1/n
$$
with the constant $r$ large enough. 

Finally, we have 
\begin{align*}
& \PP\Big( \|\bar{\eta}_n^{noise} - \bar{\eta}_n^{noise,0}\|^2_{\infty}  \geq (r+1)\gamma^{1/4}(n\gamma)^{1/\alpha} n^{-1} 
 \Big)\\
 \leq & \sum_{d=1}^r\PP\Big( \|\bar{\eta}_n^{noise,d}\|^2_{\infty}  \geq \gamma^{1/4}(n\gamma)^{1/\alpha} n^{-1} 
 \Big) + \PP\Big( \|\bar{\eta}_i^{noise} - \sum_{d=0}^r \bar{\eta}_i^{noise,d}\|^2_{\infty} \geq \gamma^{1/4}(n\gamma)^{1/\alpha} n^{-1} \Big)
\leq \gamma^{1/4} . 
\end{align*}
\end{proof}

\subsection{Bootstrap SGD decomposition}\label{app:bSGD:decomp:con}
Similar to the SGD recursion decomposition in Section \ref{sec:proof_sketch}, we define the Bootstrap SGD recursion decomposition as follows. 
Based on (\ref{eq:bootstrap:sgd}), denote $\eta_n^b = \widehat{f}_n^b - f^\ast$, then 
\begin{equation}\label{eq:boot:recursion}
\eta_n^b = (I - \gamma_n w_n K_{X_n}\otimes K_{X_n}) (f_{n-1}^b - f^\ast) + \gamma_n w_n^b \epsilon_n K_{X_n}.
\end{equation} 
We split the recursion (\ref{eq:boot:recursion}) in two  recursions $\eta_n^{b,bias}$ and $\eta_n^{b,noise}$ such that $\eta_n^b = \eta_n^{b,bias} + \eta_n^{b,noise}$. Specifically, 
\begin{align}
    \eta_n^{b,bias} = & (I - \gamma_n w_n K_{X_n}\otimes K_{X_n}) \eta^{b,bias}_{n-1} \quad \textrm{with} \quad \eta_0^{b,bias} = f^\ast, \label{eq:eta:init}\\
    \eta_n^{b,noise} = & (I - \gamma_n w_n K_{X_n}\otimes K_{X_n}) \eta^{b,noise}_{n-1} + \gamma_n w_n \epsilon_n K_{X_n} \quad \textrm{with} \quad \eta_0^{b,noise} = 0 . \label{eq:eta:noise}
\end{align} 
Since $\EE[w_n K_{X_n}\otimes K_{X_n}] = \Sigma$, we further 
decompose $\eta_n^{b,bias}$ to two parts:  
(1) its main recursion terms which determine the bias order; (2) residual recursion terms. That is,   
\begin{align*}
 \eta_n^{b,bias,0}= & (I - \gamma_n \Sigma) \eta_{n-1}^{b,bias,0} \quad \quad \textrm{with} \quad  \eta_{0}^{b,bias,0}=f^\ast \\
 \eta_n^{b,bias} - \eta_n^{b,bias,0} = & (I- \gamma_n w_n K_{X_n}\otimes K_{X_n}) (\eta_{n-1}^{b,bias} - \eta_{n-1}^{b,bias,0}) + \gamma_n (\Sigma - w_n  K_{X_n}\otimes K_{X_n} )\eta_{n-1}^{b,bias,0}, 
\end{align*}
Similarly, we decompose $\eta_n^{b,noise}$ to its main recursion term that dominates the variation and residual recursion terms as 
\begin{align}
\eta_n^{b,noise,0} = & (I - \gamma_n\Sigma) \eta^{b,noise,0}_{n-1} + \gamma_n w_n \epsilon_n K_{X_n} \label{eq:noise:main:boot}\\
\eta_n^{b,noise} - \eta_n^{b,noise,0} = &  (I - \gamma_n w_n K_{X_n}\otimes K_{X_n}) (\eta_{n-1}^{b,noise} - \eta_{n-1}^{b,noise,0})
+ \gamma_n (\Sigma - w_n K_{X_n}\otimes K_{X_n} ) \eta_{n-1}^{b, noise, 0}, \nonumber
\end{align}
with $\eta_0^{b,noise,0}=0$. 

We aim to quantify the distribution behavior of $(\bar{f}_n^b - \bar{f}_n)  \mid \mathcal{D}_n$ given $\mathcal{D}_n$. Denote $\bar{\eta}_n^b=\frac{1}{n}\sum_{i=1}^n \widehat{f}_i^b$. Then 
\begin{align*}
\bar{f}_n^b -\bar{f}_n = & \bar{\eta}_n^b  - \bar{\eta}_n = \frac{1}{n}\sum_{i=1}^n \big(\widehat{f}_i^b - f^\ast\big) - \frac{1}{n}\sum_{i=1}^n \big(\widehat{f}_i - f^\ast \big) \\
= & \underbrace{\bar{\eta}_n^{b,bias,0}- \bar{\eta}_n^{bias,0}}_{\textrm{leading bias}} + \underbrace{\bar{\eta}_n^{b,noise,0}-\bar{\eta}_n^{noise,0}}_{\textrm{leading noise}} + \underbrace{Rem_{noise}^b + Rem_{bias}^b- Rem_{noise} - Rem_{bias}}_{\textrm{negligible terms}}, 
\end{align*}
where $Rem^b_{noise} = \bar{\eta}_n^{b,noise} - \bar{\eta}_n^{b,noise,0} $,  $Rem^b_{bias} = \bar{\eta}_n^{b,bias} - \bar{\eta}_n^{b,bias,0}$, and $Rem_{noise}, Rem_{bias}$ are remainder terms in original SGD recursion with  
$Rem_{noise} = \bar{\eta}_n^{noise}-\bar{\eta}_n^{noise,0}$ (bounded in Section \ref{app:le:noise_rem:con}),  $Rem_{bias} = \bar{\eta}_n^{bias}-\bar{\eta}_n^{bias,0}$ (bounded in Section \ref{app:le:rem_bias:con}).

Since $\bar{\eta}_n^{b,bias,0}$ and  $\bar{\eta}_n^{bias,0}$ follow the same recursion, we have the leading bias of $\bar{f}_n^b -\bar{f}_n$ as 0. We next need to: (1) characterize the distribution behavior of $\bar{\eta}_n^{b,noise,0} - \bar{\eta}_n^{noise,0}$ conditional on $\mathcal{D}_n$; and (2) prove the term $Rem_{noise}^b + Rem_{bias}^b- Rem_{noise} - Rem_{bias}$ are negligible.

In the following, we provide a clear express on $\bar{\eta}_n^{b,noise,0} - \bar{\eta}_n^{noise,0}$. 
Similar to the expression of $\eta_n^{noise,0} = \sum_{i=1}^n D(i+1,n, \gamma_i) \gamma_i \epsilon_i K_{X_i}$ in (\ref{eq:noise:lead:exp}), 
A simple calculation from the recursion (\ref{eq:noise:main:boot}) shows that 
$
\eta_n^{b,noise,0} =  \sum_{i=1}^n D(i+1,n, \gamma_i) \gamma_i w_i \epsilon_i K_{X_i}. 
$ 
Accordingly, 
$$
\eta_n^{b,noise,0} - \eta_n^{noise,0} =  \sum_{i=1}^n D(i+1,n, \gamma_i) \gamma_i (w_i -1)  \epsilon_i K_{X_i} .
$$
Then 
\begin{align}
& \bar{\eta}_n^{b,noise,0} - \bar{\eta}_n^{noise,0} \nonumber\\
= & \frac{1}{n}\sum_{j=1}^n \sum_{i=1}^j  D(i+1,j, \gamma_i) \gamma_i (w_i -1)  \epsilon_i K_{X_i} 
= \frac{1}{n} \sum_{i=1} ^n \big(\sum_{j=i}^n D(i+1,j, \gamma_i) \big) \gamma_i (w_i -1)  \epsilon_i K_{X_i}. \label{eq:noise:lead}
\end{align}

\subsection{Proof of the Bootstrap consistency in Theorem \ref{thm:global:main1} for constant step size case}\label{pf:thm:global:con}
We follow the proof sketch in Section \ref{sec:sketch:pf:GA} and complete the proof of Step \RNum{2}, \RNum{3} and \RNum{4} in this section. 

For the reader's convenience, we restate the following notations. Denote 
\begin{equation*}
  \begin{array}{rcl@{\qquad}rcl}
    \bar{\alpha}_n (\cdot) & = & \frac{1}{\sqrt{n(n\gamma)^{1/\alpha}}}\sum_{i=1}^n \epsilon_i \cdot \Omega_{n,i}(\cdot) 
    &\bar{\alpha}_n^b(\cdot) & = &  \frac{1}{\sqrt{n(n\gamma)^{1/\alpha}}}\sum_{i=1}^n (w_i-1)\cdot \epsilon_i \cdot \Omega_{n,i}(\cdot) \\
    \bar{\alpha}_n^e (\cdot)  & = & \frac{1}{\sqrt{n(n\gamma)^{1/\alpha}}}\sum_{i=1}^n e_i \cdot \epsilon_i \cdot \Omega_{n,i}(\cdot)
    & \bar{Z}_n (\cdot) & = &  \frac{1}{\sqrt{n(n\gamma)^{1/\alpha}}}\sum_{i=1}^n Z_{i}(\cdot)
  \end{array}
\end{equation*}
where $e_i$'s, for $i=1,\cdots, n$, are i.i.d.~standard normal random variables, and $Z_{i}(t) \sim N \big(0, (n\gamma)^{-1/\alpha} \sum_{\nu=1}^\infty (1-(1-\gamma \mu_\nu)^{n-i})^2\phi_\nu^2(t)\big)$ satisfying 
$\EE \big(Z_{i}(t_k)\cdot Z_{i}(t_\ell)\big) =(n\gamma)^{-1/\alpha} \sum_{\nu=1}^\infty (1-(1-\gamma \mu_\nu)^{n-i})^2\phi_\nu(t_k)\phi_\nu(t_\ell)$, and $\EE  \big(Z_{i}(t_k)\cdot Z_{j}(t_\ell)\big) = 0 $ for $i\neq j$.

\begin{Lemma}\label{app:le:GP:atoz:con} (Proof of Step \RNum{2})
Suppose $\alpha>2$ and $\gamma= n^{-\xi}$ with $\xi > \max\{1-\alpha/3, 0\}$. We have 
\begin{equation}\label{eq:max:1}
\sup_{\nu\in\bbR}\Big|\PP(\max_{1\le k\le N} \bar{\alpha}_n(t_k)\le\nu)
-\PP(\max_{1\le k\le N}\bar{Z}_n(t_k)\le\nu)\Big|\leq \frac{(\log N)^{3/2}}{\big(n(n\gamma)^{-3/\alpha}\big)^{1/8}},
\end{equation}
which converges to $0$ with increased $n$.
\end{Lemma}

\begin{Lemma}\label{app:le:GP:ztoe:con}(Proof of Step \RNum{3})
Suppose $\alpha>2$ and $\gamma= n^{-\xi}$ with $\xi > \max\{1-\alpha/3, 0\}$. 
With probability at least $1-\exp(-C \log n)$, 
$$
\sup_{\nu\in \mathbb{R}}\Big| \PP^*\Big(\max_{1\leq j \leq N} \bar{\alpha}^e_n(t_j) \leq \nu  \Big)  - 
\PP\Big(\max_{1\leq j \leq N} \bar{Z}_n(t_j) \leq \nu \Big) \Big|  \preceq \big((n\gamma)^{1/\alpha}n^{-1}\big)^{1/6}(\log n)^{1/3} (\log N)^{2/3}.
$$
\end{Lemma}

\begin{Lemma}\label{app:le:GP:etob:con}(Proof of Step \RNum{4})
Suppose $\alpha>2$ and $\gamma= n^{-\xi}$ with $\xi > \max\{1-\alpha/3, 0\}$. With probability at least $1-4/n$,
$$
 \sup_{\zeta\in\bbR}\Big|\PP^*(\max_{1\le k\le N}\bar{\alpha}_n^b(t_k)\le\zeta)
-\PP^*(\max_{1\le k\le N}\bar{\alpha}_n^e(t_k)\le\zeta)\Big| \leq \frac{(\log N)^{3/2}}{\big(n(n\gamma)^{-3/\alpha}\big)^{1/8}}. 
$$
\end{Lemma}

{\underline{\bf{Proof of Lemma \ref{app:le:GP:atoz:con}}}}
\begin{proof}
 We define $g_{m} (i, X_i, \epsilon_i) = \frac{1}{\sqrt{(n\gamma)^{1/\alpha}}} \epsilon_i \cdot \Omega_{n,i} (t_m)$ for $t_m \in \{t_1, \dots, t_N\}$. With little abuse of notation, we use $g_{i,m}$ to represent $g_{m}(i,X_i,\epsilon_i)$. Then $\bar{\alpha}_n(t_m) = \frac{1}{\sqrt{n}} \sum_{i=1}^n g_{i,m}$.  Define $\bg_i = (g_{i,1},\cdots,g_{i,N})^\top$ and $\bar{\balpha}_n =\big(\bar{\alpha}_n(t_1), \cdots, \bar{\alpha}_n(t_N)\big)^\top\in \mathbb{R}^N$, then $\bar{\balpha}_n = \frac{1}{\sqrt{n}} \sum_{i=1}^n\bg_i$. For $1\leq m\leq k\leq N$, 
\begin{align*}
\EE(g_{i,m}\cdot g_{i,k}) = & \frac{\sigma^2}{(n\gamma)^{1/\alpha}} \EE [\big( \sum_{\nu=1}^\infty (1-(1-\gamma \mu_\nu)^{n-i})\phi_\nu(t_m)\phi_\nu(X_i)\big)\big(  \sum_{\nu=1}^\infty (1-(1-\gamma \mu_\nu)^{n-i})\phi_\nu(t_k)\phi_\nu(X_i)\big))]\\
= & \frac{\sigma^2}{(n\gamma)^{1/\alpha}} \sum_{\nu=1}^\infty (1-(1-\gamma \mu_\nu)^{n-i})^2\phi_\nu(t_m)\phi_\nu(t_k).
\end{align*} 
When $m=k$, $\EE(g_{i,m}\cdot g_{i,m}) = \frac{1}{(n\gamma)^{1/\alpha}} \sum_{\nu=1}^\infty (1-(1-\gamma \mu_\nu)^{n-i})^2\phi_\nu^2(t_m)$. We also have $\EE (g_{i,m}\cdot g_{j,m}) =0$ for $i\neq j$.  

We also use the notation $Z_{i,m}$ to represent $Z_{i}(t)$ defined in Section \ref{sec:sketch:pf:GA}. 
Let $\bZ_i= (Z_{i,1},\cdots, Z_{i,N})^\top\in \mathbb{R}^N$ for $i=1,\dots, n$, and $\bar{\bZ}_n= \frac{1}{\sqrt{n}}\sum_{i=1}^n \bZ_i  = (\bar{Z}_n(t_1),\cdots, \bar{Z}_n(t_N))^\top \in \mathbb{R}^N$. 
We remark that $\bar{\balpha}_n$ has the same mean and covariance structure as $\bar{\bZ}_n$. 

For $q$ as an underdetermined scalar that depends on $n$, and $\bbeta=(\beta_1,\ldots,\beta_{N})^\top\in\bbR^{N}$,
define 
$
F_q(\bbeta)=q^{-1}\log(\sum_{l=1}^{N}\exp(q\beta_l)).
$ 
It follows by \cite{chernozhukov2013gaussian} that $F_q(\bbeta)$ satisfies
$
0\le F_q(\bbeta)-\max_{1\le l\le N}\beta_l\le q^{-1}\log{N}.
$ 
Let $U_0:\bbR\rightarrow[0,1]$ be a $C^3$-function
such that $U_0(s)=1$ for $s\le0$ and $U_0(s)=0$ for $s\ge1$.
Let $U_\zeta(s)=U_0(\psi_n(s-\zeta-q^{-1}\log{N}))$, for $\zeta\in\bbR$,
where $\psi_n$ is underdetermined. 
Then 
\[
\PP(\max_{1\le m\le N}\frac{1}{\sqrt{n}}\sum_{i=1}^{n}g_{i,m} \le \zeta)
\le \PP(F_q(\bar{\balpha}_n)\le\zeta+q^{-1}\log{N})\le \EE\{U_\zeta(F_q(\bar{\balpha}_n))\}.
\]

To proceed, we approximate $\EE\{U_\zeta(F_q(\bar{\balpha}_n))-U_\zeta(F_q(\bar{\bZ}_n))\}$ using the techniques used in \cite{chernozhukov2013gaussian}. 
Let $G=U_\zeta\circ F_b$.
Define $\Psi(t)=\EE\{G(\sqrt{t}\bar{\balpha}_n+\sqrt{1-t}\bar{\bZ}_n)\}$, $W(t)=\sqrt{t}\bar{\balpha}_n+\sqrt{1-t}\bar{\bZ}_n$, 
$W_i(t)=\frac{1}{\sqrt{n}}(\sqrt{t}\bg_{i}+\sqrt{1-t}\bZ_{i})$
and $W_{-i}(t)=W(t)-W_i(t)$, for $i=1,\ldots,n$.
Let $G_{k}(\bbeta)=\frac{\partial}{\partial \beta_k}G(\bbeta)$, $G_{kl}(\bbeta)=\frac{\partial^2}{\partial \beta_k\partial \beta_l}G(\bbeta)$
and $G_{klq}(\bbeta)=\frac{\partial^3}{\partial \beta_k\partial \beta_l\partial \beta_q}G(\bbeta)$,
for $1\le k,l,q\le N$. Then $W'_{ik}(t) = \frac{1}{2\sqrt{n}} (g_{i,k}/\sqrt{t}-Z_{i,k}/\sqrt{1-t})$.

Then
\begin{eqnarray*}
&&\EE \{G(\bar{\balpha}_n) - G(\bar{\bZ}_n)\} = \EE\{U_\zeta(F_q(\bar{\balpha}_n))-U_\zeta(F_q(\bar{\bZ}_n))\}
=\Psi(1)-\Psi(0)
=\int_0^1\Psi'(t)dt\\
&=&\frac{1}{2\sqrt{n}}\sum_{k=1}^{N}\int_0^1 \EE\{G_k(W(t))(\sum_{i=1}^{n}
g_{i,k}/\sqrt{t}-\sum_{i=1}^{n}Z_{i,k}/\sqrt{1-t})\}dt\\
&=&\frac{1}{2\sqrt{n}}\sum_{k=1}^{N}\sum_{i=1}^{n}
\int_0^1\EE\big\{G_k(W(t))(g_{i,k}/\sqrt{t}-Z_{i,k}/\sqrt{1-t})\big\}dt\\
&=&\frac{1}{2\sqrt{n}}\sum_{k=1}^{N}\sum_{i=1}^{n}\int_0^1\EE
\big\{\big[G_k(W_{-i}(t))+\frac{1}{\sqrt{n}}\sum_{l=1}^{N}G_{kl}(W_{-i}(t))(\sqrt{t}g_{i,l}+\sqrt{1-t}Z_{i,l})
\\
&&+\frac{1}{n}\sum_{l=1}^{N}\sum_{d=1}^{N}\int_0^1(1-t')G_{kld}(W_{-i}(t)+t'W_i(t))(\sqrt{t}g_{i,l}+\sqrt{1-t}Z_{i,l})(\sqrt{t}g_{i,d}+\sqrt{1-t}Z_{i,d})dt'\big]\\
&&\times (g_{i,k}/\sqrt{t}-Z_{i,k}/\sqrt{1-t})\big\}dt\\
&=&\frac{1}{2\sqrt{n}}\sum_{k=1}^{N}\sum_{i=1}^{n}\int_0^1 \EE\{G_k(W_{-i}(t))\}\EE\{g_{i,k}/\sqrt{t}-Z_{i,k}/\sqrt{1-t}\}dt\\
&&+\frac{1}{2n}\sum_{k,l=1}^{N}\sum_{i=1}^{n}\int_0^1\EE\{G_{kl}(W_{-i}(t))\}
\times \EE\{(\sqrt{t}g_{i,l}+\sqrt{1-t}Z_{i,l})
(g_{i,k}/\sqrt{t}-Z_{i,k}/\sqrt{1-t})\}dt\\
&&+\frac{1}{2n^{3/2}}\sum_{k,l,d=1}^{N}\sum_{i=1}^{n}\int_0^1\int_0^1
(1-t')\EE\{G_{kld}(W_{-i}(t)+t'W_i(t))(\sqrt{t}g_{i,l}+\sqrt{1-t}Z_{i,l})\\
&&(\sqrt{t}g_{i,d}+\sqrt{1-t}Z_{i,d})(g_{i,k}/\sqrt{t}-Z_{i,k}/\sqrt{1-t})\}dtdt'\\
&\equiv&J_1/2+J_2/2 + J_3/2,
\end{eqnarray*}
where 
\begin{eqnarray*}
J_1& = &\frac{1}{\sqrt{n}} \sum_{k=1}^{N}\sum_{i=1}^{n}\int_0^1 \EE\{G_k(W_{-i}(t))\}\EE\{g_{i,k}/\sqrt{t}-Z_{i,k}/\sqrt{1-t}\}dt = 0\\
J_2&=&\frac{1}{n}\sum_{k,l=1}^{N}\sum_{i=1}^{n}\int_0^1\EE\{G_{kl}(W_{-i}(t))\}\times \EE\{(\sqrt{t}g_{i,l}+\sqrt{1-t}Z_{i,l})
(g_{i,k}/\sqrt{t}-Z_{i,k}/\sqrt{1-t})\}dt \\  
J_3&=&\frac{1}{n^{3/2}}\sum_{k,l,d=1}^{N}\sum_{i=1}^{n}\int_0^1\int_0^1
(1-t')\EE\{G_{kld}(W_{-i}(t)+t'W_i(t))(\sqrt{t}g_{i,l}+\sqrt{1-t}Z_{i,l})
(\sqrt{t}g_{i,d}+\sqrt{1-t}Z_{i,d}) \\  
&&(Z_{i,k}/\sqrt{t}-Z_{i,k}/\sqrt{1-t})\}dtdt'
\end{eqnarray*}
We further note that $J_2= 0$ since $\EE\{(\sqrt{t}g_{i,l}+\sqrt{1-t}Z_{i,l})
(g_{i,k}/\sqrt{t}-Z_{i,k}/\sqrt{1-t})\} = \EE(g_{i,l}g_{i,k}) - \EE (Z_{i,l}Z_{i,k})=0$. 
For $J_3$, it follows from \cite{chernozhukov2013gaussian} that for any $z\in\bbR^{N}$,
\[
\sum_{k,l,d=1}^{N}|G_{kld}(z)|\le (C_3\psi_n^3+6C_2q\psi_n^2+6C_1q^2\psi_n),
\]
where $C_3=\|U_0'''\|_{\infty}$ is a finite constant. Then %
\begin{align}
|J_3|\le& \frac{1}{n^{3/2}}\sum_{k,l,d=1}^{N}\sum_{i=1}^{n}\int_0^1\int_0^1
\EE\{|G_{kld}(W_{-i}(t)+t'W_i(t))|\max_{1\le k\le N}(|g_{i,k}|+|Z_{i,k}|)^3\}\nonumber\\
&\times (1/\sqrt{t}+1/\sqrt{1-t})dtdt'\nonumber\\
\le&\frac{1}{n^{3/2}}4(C_3\psi_n^3+6C_2q\psi_n^2+6C_1q^2\psi_n)\sum_{i=1}^{n}
\EE\{\max_{1\le k\le N}(|g_{i,k}|+|Z_{i,k}|)^3\}\nonumber\\
\le&\frac{1}{n^{3/2}}32(C_3\psi_n^3+6C_2q\psi_n^2+6C_1q^2\psi_n)
\big\{\sum_{i=1}^{n}(\EE\{\max_{1\leq k\leq N}|g_{i,k}|^3\}+\EE\{\max_{1\leq k\leq N}|Z_{i,k}|^3\})\big\}\label{pf:J3:con}
\end{align}
We need to bound $\sum_{i=1}^{n}(\EE\{\max\limits_{1\leq k\leq N}|g_{i,k}|^3\}$ and $\EE\{\max\limits_{1\leq k\leq N}|Z_{i,k}|^3\}$. 
\begin{align*}
\sum_{i=1}^n \EE \max_{1\leq k \leq N} |g_{i,k}|^3 = &  \frac{1}{(n\gamma)^{3/(2\alpha)}} \sum_{i=1}^n \EE \max_{1\leq k \leq N} |\epsilon_i \cdot \Omega_{n,i}(t_k)|^3\\
\leq &  \frac{1}{(n\gamma)^{3/(2\alpha)}} \sum_{i=1}^n \EE |\epsilon_i|^3 \cdot \EE  \max_{1\leq k \leq N} |\Omega_{n,i}(t_k)|^3\\
\lesssim & \frac{\sigma^3}{(n\gamma)^{3/(2\alpha)}}\sum_{i=1}^n\EE\max_{1\leq k \leq N} |\Omega_{n,i}(t_k)|^3 \leq c_\phi^6 \sigma^3 n (n\gamma)^{3/(2\alpha)}, 
\end{align*}
where the last step is due to the property that  $|\Omega_{n,i}| = \sum_{\nu=1}^\infty (1-(1-\gamma \mu_\nu)^{n-i})\cdot |\phi_\nu(X_i)|\cdot |\phi_\nu(t_k)| \leq c_\phi^2 (n\gamma)^{1/\alpha}$, and  $\max\limits_{1\leq k \leq N} |\Omega_{i,k}|^3  \leq c_\phi^6 (n\gamma)^{3/\alpha}$. 

Next we deal with $\EE \max\limits_{1\leq k \leq N} |Z_{i,k}|^3$, where 
$Z_{i,k} \sim N \big(0, \frac{1}{(n\gamma)^{1/\alpha}} \sum_{\nu=1}^\infty (1-(1-\gamma \mu_\nu)^{n-i})^2\phi_\nu^2(t_k)\big)$, and $\EE (Z_{i,k}\cdot Z_{i,l}) = \frac{1}{(n\gamma)^{1/\alpha}} \sum_{\nu=1}^\infty (1-(1-\gamma \mu_\nu)^{n-i})^2\phi_\nu(t_k)\phi_\nu(t_l)$. For $p>3$, we have 
\begin{align*}
\EE \max_{1\leq k \leq N} |Z_{i,k}|^3=  & \EE \max_{1\leq k \leq N} (|Z_{i,k}|^p)^{3/p} \leq \big(\EE \max_{1\leq k \leq N} |Z_{i,k}|^p\big)^{3/p} \leq 
\big(\sum_{k=1}^N \EE |Z_{i,k}|^p\big)^{3/p}\\
= & \frac{1}{(n\gamma)^{3/(2\alpha)}}\big[\sum_{k=1}^N \big(\sum_{\nu=1}^\infty (1-(1-\gamma \mu_\nu)^{n-i})^2\phi_\nu^2(t_k)\big)^{p/2}\big]^{3/p} ((p-1)!!)^{3/p}\\
\leq & c_\phi^2 ((p-1)!!)^{3/p} \frac{1}{(n\gamma)^{3/(2\alpha)}} N^{3/p} [(n-i)\gamma]^{\frac{3}{2\alpha}} 
\end{align*}

Then we have 
\begin{align*}
J_3 \leq & 32\sigma^3(C_3\psi_n^3+6C_2q\psi_n^2+6C_1q^2\psi_n)\big(n^{-3/2}  n(n\gamma)^{3/(2\alpha)} + n^{-3/2}  c_\phi^2 ((p-1)!!)^{3/p} N^{3/p} n(n\gamma)^{\frac{3}{2\alpha}}\big) \\ 
\leq & C' (C_3\psi_n^3+6C_2q\psi_n^2+6C_1q^2\psi_n)\big(n^{-1/2} (n\gamma)^{\frac{3}{2\alpha}} 
+ ((p-1)!!)^{3/p} N^{3/p}n^{-1/2}\big) 
\end{align*}
Therefore, 
\begin{align}\label{at:eq:gaussian:bound}
& |\EE\{U_\zeta(F_q(\bar{\balpha}_n))-U_\zeta(F_q(\bar{\bZ}_n))\}|\nonumber\\
\le & C' (C_3\psi_n^3+6C_2q\psi_n^2+6C_1q^2\psi_n)\big(n^{-1/2} (n\gamma)^{\frac{3}{2\alpha}}  
+ ((p-1)!!)^{3/p} N^{3/p}n^{-1/2}\big) 
\end{align}
In the meantime, it follows by Lemma 2.1 of \cite{chernozhukov2013gaussian} that
\begin{eqnarray*}
\EE\{U_\zeta(F_q(\bar{\bZ}_n))\}
&\le& \PP(\max_{1\le k\le N}\sum_{i=1}^{n}Z_{i,k}\le\zeta+b^{-1}\log{N}+\psi_n^{-1})\\
&\le& \PP(\max_{1\le k\le N}\sum_{i=1}^{n}Z_{i,k}\le\zeta)+C'(b^{-1}\log{N}+\psi_n^{-1})(1+\sqrt{2\log{N}}),
\end{eqnarray*}
where $C'>0$ is a universal constant. Therefore, for any $\zeta\in \mathbb{R}$, 
\begin{align*}
& \PP(\max_{1\le k\le N}\bar{\alpha}_n (t_k)\le\zeta)
-\PP(\max_{1\le k\le N}\bar{Z}_n(t_k)\le\zeta)\\
\le &  C_4(C_3\psi_n^3+6C_2q\psi_n^2+6C_1q^2\psi_n)\big(n^{-1/2} (n\gamma)^{\frac{3}{2\alpha}}  
+ ((p-1)!!)^{3/p} N^{3/p}n^{-1/2}\big)  \\
& \quad \quad  + c'(b^{-1}\log N + \psi_n^{-1})(1+\sqrt{2\log N}). 
\end{align*}

On the other hand, let $V_\zeta(s)=U_0(\psi_n(s-\zeta)+1)$. 
Then
\[
\PP(\max_{1\le k\le N}\frac{1}{\sqrt{n}}\sum_{i=1}^{n}g_{i,k}\le\zeta)
\ge \PP(F_q(\bar{\balpha}_n)\le\zeta)\ge \EE\{V_\zeta(F_q(\bar{\balpha}_n))\}.
\]
Using the same arguments, it can be shown that
$|\EE\{V_\zeta(F_q(\bar{\balpha}_n))-V_\zeta(F_q(\bar{\bZ}_n))\}|$ 
has the same upper bound specified as in (\ref{at:eq:gaussian:bound}). 
Furthermore, by Lemma 2.1 of \cite{chernozhukov2013gaussian} and direct calculations, we have 
\begin{eqnarray*}
\EE\{V_\zeta(F_q(\bar{\bZ}_n))\}
&\ge& \PP(F_q(\bar{\bZ}_n)\le\zeta-\psi_n^{-1})
\ge \PP(\max_{1\le k\le N}\frac{1}{\sqrt{n}}\sum_{i=1}^{n} Z_{i,k}\le\zeta-(\psi_n^{-1}+b^{-1}\log{N}))\\
&\ge& \PP(\max_{1\le k\le N}\frac{1}{\sqrt{n}}\sum_{i=1}^{n}Z_{i,k}\le\zeta)-C'(\psi_n^{-1}+b^{-1}\log{N})(1+\sqrt{2\log{N}}).
\end{eqnarray*}
Therefore, 
\begin{eqnarray*}
&&\PP(\max_{1\le k\le N}\bar{\alpha}_n(t_k)\le\zeta)
-\PP(\max_{1\le k\le N}\bar{Z}_n(t_k)\le\zeta)\\
&\ge&-C_0^{''}(C_3\psi_n^3+C_2q\psi_n^2 + 6C_1 q^2\psi_n)\big(  n^{-1/2} (n\gamma)^{\frac{3}{2\alpha}}
+ ((p-1)!!)^{3/p} N^{3/p}n^{-1/2}\big)\\
&&  - C^{''}(\psi_n^{-1} + q^{-1}\log N)(1+\sqrt{2\log N}).
\end{eqnarray*}
Consequently, let $\psi_n = q = \big(n(n\gamma)^{-3/\alpha}\big)^{1/8}$ and $p$ large enough, we have 
\begin{equation}
\sup_{\zeta\in\bbR}\Big|\PP(\max_{1\le k\le N} \bar{\alpha}_n(t_k)\le\zeta)
-\PP(\max_{1\le k\le N}\bar{Z}_n(t_k)\le\zeta)\Big|\leq \frac{(\log N)^{3/2}}{\big(n(n\gamma)^{-3/\alpha}\big)^{1/8}},
\end{equation}
which converges to $0$ with increased $n$ when $\alpha>2$ and $\gamma= n^{-\xi}$ with $\xi > \max\{1-\alpha/3, 0\}$. 
\end{proof}

{\underline{\bf{Proof of Lemma \ref{app:le:GP:ztoe:con}}}}
\begin{proof}
Let $\bar{\balpha}^e_n =\big(\bar{\alpha}^e_n(t_1), \cdots, \bar{\alpha}^e_n(t_N)\big)^\top$ and $\bar{\bZ}_n= \frac{1}{\sqrt{n}}\sum_{i=1}^n \bZ_i  = (\bar{Z}_n(t_1),\cdots, \bar{Z}_n(t_N))^\top$. Then $\bar{\balpha}^e_n \mid \mathcal{D}_n \sim N(0, \Sigma^{\bar{\alpha}^e_n})$ and $\bar{\bZ}_n \sim N(0,\Sigma^{\bar{Z}_n})$. Denote the $jk$-th element of the covariance matrices as $\Sigma_{j,k}^{\bar{\alpha}^e_n}$ and $\Sigma_{j,k}^{\bar{Z}_n}$, respectively. Set $b_{i\nu} = (1-(1-\gamma \mu_\nu)^{n-i})$. Then 
\begin{align*}
\Sigma_{j,k}^{\bar{\alpha}^e_n} = &  \frac{1}{n(n\gamma)^{1/\alpha}} \sum_{i=1}^n \epsilon_i^2 \big(\sum_{\nu=1}^\infty b_{i\nu}\phi_\nu(X_i)\phi_\nu(t_j)\big)\cdot \big(\sum_{\nu=1}^\infty b_{i\nu}\phi_\nu(X_i)\phi_\nu(t_k)\big), 
\end{align*}
and $\Sigma_{j,k}^{\bar{Z}_n} = \frac{1}{n(n\gamma)^{1/\alpha}}\sum_{i=1}^n \sum_{\nu=1}^\infty b_{i\nu}^2 \phi_\nu(t_k)\phi_\nu(t_j)$.

Following Lemma in \cite{liu2023supp}, we have 
$$
\PP\Big(|\Sigma_{j,k}^{\bar{\alpha}^e_n} - \Sigma_{j,k}^{\bar{Z}_n}| \geq C(n\gamma)^{1/(2\alpha)}n^{-1/2} \log n \Big) \leq \exp(-C_1 \log n). 
$$
Then  
\begin{align*}
 & \PP\Big(\max_{1\leq j,k \leq N} |\Sigma_{j,k}^{\bar{\alpha}^e_n} - \Sigma_{j,k}^{\bar{Z}_n}| \geq C(n\gamma)^{1/(2\alpha)}n^{-1/2} \log n \Big) 
\leq N^2\exp\{-C_1\log n\}. 
\end{align*}
Consequently, we have with probability at least $1-\exp\{-C\log n\}$, 
$$
\sup_{\nu\in \mathbb{R}}\Big| \PP^*\Big(\max_{1\leq j \leq N} \bar{\alpha}^e_n(t_j) \leq \nu  \Big)  - 
\PP\Big(\max_{1\leq j \leq N} \bar{Z}_n(t_j) \leq \nu \Big) \Big|  \preceq \big((n\gamma)^{1/\alpha}n^{-1}\big)^{1/6}(\log n)^{1/3} (\log N)^{2/3}.  
$$
\end{proof}

{\underline{\bf{Proof of Lemma \ref{app:le:GP:etob:con}}}}
\begin{proof}
Define $\alpha_{i,j}^e = e_i \cdot g_{j} (i, X_i, \epsilon_i)$ with $g_{j} (i, X_i, \epsilon_i) = \frac{1}{\sqrt{(n\gamma)^{1/\alpha}}} \epsilon_i \cdot \Omega_{n,i}(t_j)=\sum_{\nu=1}^\infty \big(1-(1-\gamma \mu_\nu)^{n-i}\big)\phi_\nu(X_i)\phi_\nu(t_j)\epsilon_i. 
$
We have $\EE^* (\alpha_{i,j}^e \cdot \alpha_{i,\ell}^e )= g_{j} (i, X_i, \epsilon_i)g_{\ell} (i, X_i, \epsilon_i)$ and $\EE^*(\alpha_{i,j}^e \cdot \alpha_{k,j}^e) =0$.  
Define $\balpha_{i}^e= (\alpha_{i,j}^e,\dots, \alpha_{i,N}^e)^\top$,
then $\balpha_{i}^e$ and $\balpha_{k}^e$ are independent for $i\neq k$ and $i,k=1,\dots,n$. Let $\bar{\balpha}_n^e = \frac{1}{\sqrt{n}}\sum_{i=1}^n  \balpha_{i}^e= \big( 
\bar{\alpha}_n^e(t_1), \dots, \bar{\alpha}_n^e(t_N)
\big)^\top$ with  $\bar{\alpha}_n^e(t_j) = \frac{1}{\sqrt{n}} \sum_{i=1}^n  \alpha^e_{i,j} = \frac{1}{\sqrt{n}} \sum_{i=1}^n e_i \cdot g_{j} (i, X_i, \epsilon_i)$ for $j=1,\dots, N$.

Similarly, denote $\alpha_{i,j}^b = (w_i-1) \cdot g_{j} (i, X_i, \epsilon_i)$ and $\balpha_{i}^b= (\alpha_{i,j}^b,\dots, \alpha_{i,N}^b)^\top$. Then we have $\EE^* (\alpha_{i,j}^b \cdot \alpha_{i,\ell}^b )= g_{j} (i, X_i, \epsilon_i)g_{\ell} (i, X_i, \epsilon_i)$, and $\EE^*(\alpha_{i,j}^b \cdot \alpha_{k,j}^b) =0.$  Denote $\bar{\balpha}_n^b = \frac{1}{\sqrt{n}}\sum_{i=1}^n  \balpha_{i}^b= \big( 
\bar{\alpha}_n^b(t_1), \dots, \bar{\alpha}_n^b(t_N)
\big)^\top$ with  $\bar{\alpha}_n^b(t_j) = \frac{1}{\sqrt{n}} \sum_{i=1}^n  \alpha^b_{i,j}$. 

The proof of Lemma \ref{app:le:GP:etob:con} follows the proof of Lemma \ref{app:le:GP:atoz:con}. We adopt the notation and follow the proof of Lemma \ref{app:le:GP:atoz:con} step by step, with only the following changes: (1) replacing $\bar{\balpha}_n$ with $\bar{\balpha}_n^b$; (2) replacing $\bar{\bZ}_n$ with $\bar{\balpha}_n^e$; (3) replacing the probability $\PP(\cdot)$ and expectation $\EE(\cdot)$ to conditional probabilities $\PP*(\cdot)=\PP(\cdot \mid \mathcal{D}_n)$ and conditional expectation $\EE^*(\cdot)= \EE(\cdot \mid \mathcal{D}_n)$.  Then equation (\ref{pf:J3:con}) here will be adapted to 
\begin{equation}\label{eq:J3:conditional:con}
|J_3|
\le\frac{1}{n^{3/2}}32(C_3\psi_n^3+6C_2q\psi_n^2+6C_1q^2\psi_n)
\Big(\sum_{i=1}^{n}(\EE^*\{\max_{1\leq k\leq N}|\alpha^b_{i,k}|^3\}+\EE^*\{\max_{1\leq k\leq N}|\alpha^e_{i,k}|^3\})\Big)
\end{equation}
Since 
\begin{align*}
\EE^* \max_{1\leq k \leq N} |\alpha^b_{i,k}|^3 \leq &  \frac{1}{(n\gamma)^{3/(2\alpha)}}  \max_{1\leq k \leq N} |\epsilon_i\cdot \Omega_{n,i}(t_k)|^3 \cdot \EE|w_i-1|^3 
\lesssim  |\epsilon_i|^3(n\gamma)^{3/(2\alpha)}, 
\end{align*} 
where the last inequality follows the proof in Lemma \ref{app:le:GP:atoz:con} that $\max\limits_{1\leq k \leq N} |\Omega_{i,k}|^3  \leq c_\phi^6 (n\gamma)^{3/\alpha}$. 
Then with probability at least $1-n^{-1}$, we have  
\begin{align*}
\frac{1}{n^{3/2}}\sum_{i=1}^n\EE^*\big(\max_{1\leq k \leq N} |\alpha^b_{i,k}|^3\big) \leq & 
n^{-1/2} (n\gamma)^{3/(2\alpha)} \frac{1}{n}\sum_{i=1}^n |\epsilon_i|^3 \leq C n^{-1/2} (n\gamma)^{3/(2\alpha)}. 
\end{align*}
Similarly, with probability at least $1-n^{-1}$,  
$
\frac{1}{n^{3/2}}\sum_{i=1}^n\EE^*\big(\max_{1\leq k \leq N} |\alpha^e_{i,k}|^3\big) \leq  C n^{-1/2} (n\gamma)^{3/(2\alpha)}, 
$
where $C$ is a constant independent of $n$. Then we have with probability at least $1-2n^{-1}$, 
\begin{align*}
J_3 \leq  & C (C_3\psi_n^3+6C_2q\psi_n^2+6C_1q^2\psi_n)  n^{-1/2} (n\gamma)^{3/(2\alpha)}. 
\end{align*}
Therefore, follow the proofs in Lemma \ref{app:le:GP:atoz:con}, we have 
\begin{align*}\label{eq:max:1}
& |\PP^*(\max_{1\le k\le N}\bar{\alpha}_n^b(t_k)\le\zeta)
-\PP^*(\max_{1\le k\le N}\bar{\alpha}_n^e(t_k)\le\zeta)| \\
\leq & C_0(C_3\psi_n^3+C_2q\psi_n^2 + 6C_1 q^2\psi_n) n^{-1/2} (n\gamma)^{3/(2\alpha)}+ C^{''}(\psi_n^{-1} + q^{-1}\log N)(1+\sqrt{2\log N}). 
\end{align*}
Consequently, let $\psi_n = q = \big(n(n\gamma)^{-3/\alpha}\big)^{1/8}$,  we have with probability at least $1-4/n$,
$$
 \sup_{\zeta\in\bbR}\Big|\PP^*(\max_{1\le k\le N}\bar{\alpha}_n^b(t_k)\le\zeta)
-\PP^*(\max_{1\le k\le N}\bar{\alpha}_n^e(t_k)\le\zeta)\Big| \leq \frac{(\log N)^{3/2}}{\big(n(n\gamma)^{-3/\alpha}\big)^{1/8}}. 
$$

\end{proof}

\bibliographystyle{unsrt}
\bibliography{ref}

\end{document}